\documentclass[11pt]{amsart}

\DeclareMathOperator{\Lie}{Lie}

\DeclareMathOperator{\rk}{rk}
\DeclareMathOperator{\Ad}{Ad}
\DeclareMathOperator{\ad}{ad}
\DeclareMathOperator{\Der}{Der}
\DeclareMathOperator{\soc}{Soc}
\usepackage{amsmath,amssymb,amsthm,graphics,amscd,mathrsfs}
\usepackage{longtable}
\usepackage{cite,xy}
\usepackage{color}
\usepackage{graphicx}
\usepackage{epsf}
\usepackage{fancyhdr,hyperref}
\usepackage{lscape}
\hypersetup{backref,
	colorlinks=true,backref}

  \renewenvironment{thebibliography}[1]{
    \begin{oldthebibliography}{#1}
      \setlength{\parskip}{0ex}
      \setlength{\itemsep}{0ex}
  }
  {
    \end{oldthebibliography}
  }

\setlongtables
\begin{document}

\newcounter{rownum}
\setcounter{rownum}{0}
\newcommand{\ab}{\addtocounter{rownum}{1}\arabic{rownum}}

\newcommand{\x}{$\times$}
\newcommand{\bb}{\mathbf}

\newcommand{\Ind}{\mathrm{Ind}}
\newcommand{\Char}{\mathrm{char}}
\newcommand{\hra}{\hookrightarrow}
\newtheorem{lemma}{Lemma}[section]
\newtheorem{theorem}[lemma]{Theorem}
\newtheorem*{TA}{Theorem A}
\newtheorem*{TB}{Theorem B}
\newtheorem*{TC}{Theorem C}
\newtheorem*{CorC}{Corollary C}
\newtheorem*{TD}{Theorem D}
\newtheorem*{TE}{Theorem E}
\newtheorem*{PF}{Proposition E}
\newtheorem*{C3}{Corollary 3}
\newtheorem*{T4}{Theorem 4}
\newtheorem*{C5}{Corollary 5}
\newtheorem*{C6}{Corollary 6}
\newtheorem*{C7}{Corollary 7}
\newtheorem*{C8}{Corollary 8}
\newtheorem*{claim}{Claim}
\newtheorem{cor}[lemma]{Corollary}
\newtheorem{conjecture}[lemma]{Conjecture}
\newtheorem{prop}[lemma]{Proposition}
\newtheorem{question}[lemma]{Question}
\theoremstyle{definition}
\newtheorem{example}[lemma]{Example}
\newtheorem{examples}[lemma]{Examples}
\theoremstyle{remark}
\newtheorem{remark}[lemma]{Remark}
\newtheorem{remarks}[lemma]{Remarks}
\newtheorem{obs}[lemma]{Observation}
\theoremstyle{definition}
\newtheorem{defn}[lemma]{Definition}

  \def\hal{\unskip\nobreak\hfil\penalty50\hskip10pt\hbox{}\nobreak
  \hfill\vrule height 5pt width 6pt depth 1pt\par\vskip 2mm}

\renewcommand{\labelenumi}{(\roman{enumi})}
\newcommand{\Hom}{\mathrm{Hom}}
\newcommand{\Int}{\mathrm{int}}
\newcommand{\Ext}{\mathrm{Ext}}
\newcommand{\opH}{\mathrm{H}}
\newcommand{\D}{\mathscr{D}}
\newcommand{\SO}{\mathrm{SO}}
\newcommand{\Sp}{\mathrm{Sp}}
\newcommand{\SL}{\mathrm{SL}}
\newcommand{\GL}{\mathrm{GL}}
\newcommand{\OO}{\mathcal{O}}
\newcommand{\diag}{\mathrm{diag}}
\newcommand{\End}{\mathrm{End}}
\newcommand{\tr}{\mathrm{tr}}
\newcommand{\Stab}{\mathrm{Stab}}
\newcommand{\red}{\mathrm{red}}
\newcommand{\Aut}{\mathrm{Aut}}
\renewcommand{\H}{\mathcal{H}}
\renewcommand{\u}{\mathfrak{u}}
\newcommand{\N}{\mathcal{N}}
\newcommand{\Z}{\mathbb{Z}}
\newcommand{\ud}{\mathrm{d}}
\newcommand{\la}{\langle}\newcommand{\ra}{\rangle}
\newcommand{\gl}{\mathfrak{gl}}
\newcommand{\g}{\mathfrak{g}}
\newcommand{\F}{\mathbb{F}}
\newcommand{\NN}{\mathbb{N}}
\newcommand{\m}{\mathfrak{m}}
\renewcommand{\b}{\mathfrak{b}}
\newcommand{\p}{\mathfrak{p}}
\newcommand{\q}{\mathfrak{q}}
\renewcommand{\l}{\mathfrak{l}}
\newcommand{\del}{\partial}
\newcommand{\h}{\mathfrak{h}}
\renewcommand{\t}{\mathfrak{t}}
\newcommand{\Id}{\mathrm{Id}}
\renewcommand{\k}{\mathfrak{k}}
\newcommand{\Gm}{\mathbb{G}_m}
\renewcommand{\c}{\mathfrak{c}}
\renewcommand{\r}{\mathfrak{r}}
\newcommand{\n}{\mathfrak{n}}
\newcommand{\s}{\mathfrak{s}}
\newcommand{\Q}{\mathbb{Q}}
\newcommand{\C}{\mathbb{C}}
\newcommand{\z}{\mathfrak{z}}
\newcommand{\pso}{\mathfrak{pso}}
\newcommand{\so}{\mathfrak{so}}
\renewcommand{\sl}{\mathfrak{sl}}
\newcommand{\psl}{\mathfrak{psl}}
\newcommand{\pgl}{\mathfrak{pgl}}
\renewcommand{\sp}{\mathfrak{sp}}
\newcommand{\Ga}{\mathbb{G}_a}

\newenvironment{changemargin}[1]{%
  \begin{list}{}{%
    \setlength{\topsep}{0pt}%
    \setlength{\topmargin}{#1}%
    \setlength{\listparindent}{\parindent}%
    \setlength{\itemindent}{\parindent}%
    \setlength{\parsep}{\parskip}%
  }%
  \item[]}{\end{list}}

\parindent=0pt
\addtolength{\parskip}{0.5\baselineskip}

\title{Classification of the maximal subalgebras of exceptional Lie algebras over fields of good characteristic}
\author{Alexander Premet and David I. Stewart}
\address{School of Mathematics, The University of Manchester, Oxford Road, M13 9PL, UK}
\email{alexander.premet@manchester.ac.uk}
\address{University of Newcastle, Newcastle upon Tyne, NE1 7RU, UK} \email{david.stewart@ncl.ac.uk} 
\pagestyle{plain}
\begin{abstract}
Let $G$ be an exceptional simple algebraic group over an algebraically closed field $k$
and suppose that $p={\rm char}(k)$ is a good prime for $G$.
In this paper we classify the maximal Lie subalgebras $\m$ of the Lie algebra $\g=\Lie(G)$. Specifically, we show that
either $\m=\Lie(M)$ for some maximal connected subgroup $M$ of $G$ or $\m$
is a maximal Witt subalgebra of $\g$ or $\m$ is a maximal {\it exotic semidirect product}.
The conjugacy classes of maximal connected subgroups of $G$ are known thanks to the work of
Seitz, Testerman and Liebeck--Seitz.
All maximal Witt subalgebras of $\g$ are $G$-conjugate and they occur when $G$ is not of type ${\rm E}_6$ and  $p-1$ coincides with the Coxeter number of $G$. We show that there are
two conjugacy classes of maximal exotic semidirect products in $\g$, one in characteristic $5$ and one in characteristic $7$, and both occur when $G$ is a group of type ${\rm E}_7$.
\end{abstract}
\maketitle
\section{Introduction}\label{intro}
Unless otherwise specified, $G$ will denote a simple algebraic group of exceptional type defined over an algebraically closed field $k$ of characteristic $p>0$. We always assume that $p$ a good prime for $G$, that is $p>5$ if $G$ is of type ${\rm E}_8$ and $p>3$ in the other cases.
Under this hypothesis, the Lie algebra $\g=\Lie(G)$ is simple. Being the Lie algebra of an algebraic group, it carries a natural $[p]$th power map $\g\ni x\mapsto x^{[p]}\in \g$ equivariant under the adjoint action of $G$.
The goal of this paper is to classify the maximal Lie subalgebras $\m$ of $\g$ up to conjugacy under the adjoint action of $G$. The maximality of $\m$ implies that it is a restricted subalgebra of $\g$, i.e. has the property that $\m^{[p]}\subseteq \m$.   The main result of
\cite{P17} states that if ${\rm rad}(\m)\ne 0$  then
$\m=\Lie(P)$ for some maximal parabolic subgroup of $G$.
Therefore, in this paper we are concerned with the case where $\m$ is a semisimple Lie algebra. In prime characteristic this does not necessarily  mean that
$\m$ is a direct sum of its simple ideals.

We write $\OO(1;\underline{1})$ for the truncated polynomial ring $k[X]/(X^p)$. The derivation algebra of $\OO(1;\underline{1})$, denoted $W(1;\underline{1})$, is known as the  Witt algebra.
This restricted Lie algebra is a free $\OO(1;\underline{1})$-module of rank $1$
generated by the derivative $\partial$ with respect to the image of $X$ in $\OO(1;\underline{1})$. A restricted Lie subalgebra $\mathcal{D}$ of
$W(1;\underline{1})$ is called {\it transitive} if it does not preserve any
 proper nonzero ideals of
$\OO(1;\underline{1})$.

Let
$\h$ be any semisimple restricted Lie subalgebra of $\g$ and let ${\rm Soc}(\h)$ denote the socle of the adjoint $\h$-module $\h$. This is, of course, the sum of all {\it minimal} ideals of $\h$.
As one of the main steps in our classification, we show
that if ${\rm Soc}(\h)$ is indecomposable and not semisimple, then
${\rm Soc}(\h)\cong \sl_2\otimes \OO(1;\underline{1})$ and there exists a transitive Lie subalgebra $\mathcal{D}$ of the Witt algebra
$W(1;\underline{1})$ such that
\begin{equation}
\label{esdp-eq}
\h\,\cong\, (\sl_2\otimes \OO(1;\underline{1}))\rtimes ({\rm Id}\otimes \mathcal{D})\end{equation} as Lie algebras.
We call the semisimple restricted Lie subalgebras of this type {\it exotic semidirect products}, esdp's for short, and we show that under our assumptions on $G$ the Lie algebra $\g$ contains an esdp if and only if either $G$ is of type $E_7$ and
$p\in\{5,7\}$ or $G$ is of type  ${\rm E}_8$ and $p=7$. (Although esdp's do exist in Lie algebras $\Lie(G)$ of type ${\rm E}_8$ over fields of characteristic $5$ we ignore them in the present paper as $p=5$ is a bad prime for $G$.) It turned out that for $p>5$
any esdp in type ${\rm E}_8$ is contained in a proper regular subalgebra of $\g$ and hence is not maximal in $\g$.  We prove that in type ${\rm E}_7$ maximal esdp's of $\g$ do exist and form a single conjugacy class under the adjoint action of $G$. Furthermore, if $\h$ as in (\ref{esdp-eq}) is maximal in $\g$ then $\mathcal{D}\cong \sl_2$ when $p=5$ and $\mathcal{D}=W(1;\underline{1})$ when $p=7$.

As our next step we use the smoothness of centralisers $C_G(V)$, where $V$ is a subspace of $\g$, to show that if ${\rm Soc}(\m)$ is semisimple and contains more than one minimal ideal of $\m$, then there exists a semisimple (and non-simple) maximal connected subgroup $M$ of $G$ such that $\m=\Lie(M)$; see \S~\ref{ss-dec}. We mention that Proposition~\ref{decomp}
 applies to simple algebraic groups of classical types as well.

Having obtained the above results we are left with the case where ${\rm Soc}(\m)$ is a simple Lie algebra.
In this situation, we prove that if the derivation algebra $\Der(\m)$ is isomorphic to the Lie algebra of a reductive $k$-group, then $\m=\Lie(M)$ for some maximal connected
subgroup of $G$. If $\Der(\m)$ is not of that type, we show that $p-1$ equals the Coxeter number of $G$ and
$\m$ is a maximal Witt subalgebra of $\g$, which is unique up to $(\Ad\,G)$-conjugacy by the main result of \cite{HSMax}.

 We let $\N(\g)$ denote the nilpotent cone of the restricted Lie algebra $\g$ and write $\OO(D)$ for the adjoint $G$-orbit in $\N(\g)$ with Dynkin label $D$. A Lie subalgebra of $\g$ is called {\it regular} if it contains a maximal toral subalgebra of $\g$. Our results on semisimple restricted Lie subalgebras $\h$ of $\g$ containing a non-simple minimal ideal do not require the maximality hypothesis.
\begin{theorem}\label{thm:esdps} Let $\h$ be a semisimple restricted Lie subalgebra of $\g$ containing a minimal ideal which is not simple. Then $p\in\{5,7\}$, the group $G$ is of type ${\rm E}_7$ or ${\rm E}_8$, and
the following hold:
\begin{itemize}
\item[(i)\,] If $G$ is of type ${\rm E}_8$ then
$\h$ is contained in a regular subalgebra of type ${\rm E}_7{\rm A}_1$ and hence is not maximal in $\g$.

\smallskip

\noindent
\item[(ii)\,] If $G$ is of type ${\rm E}_7$ then ${\rm Soc}(\h)\cong \sl_2\otimes \OO(1;\underline{1})$ and $\h\cong (\sl_2\otimes \OO(1;\underline{1}))\rtimes ({\rm Id}\otimes\mathcal{D})$ for some transitive restricted Lie subalgebra
$\mathcal{D}$ of $W(1;\underline{1})$. In particular, $\h$ is an esdp.

\smallskip

\noindent
\item[(iii)\,] Suppose $G$ is an adjoint group of type ${\rm E}_7$ and let $N:=N_G({\rm Soc}(\h))$. Then $N$
is a closed connected subgroup of $G$ acting transitively of the set of all nonzero
$\sl_2$-triples of $\g$ contained in ${\rm Soc}(\h)$.
All nilpotent elements of such $\sl_2$-triples lie in
$\OO({\rm A}_3{\rm A}_2{\rm A}_1)$ when $p=5$ and
in $\OO({\rm A}_2{{\rm A}_1}^3)$ when $p=7$.

\smallskip

\noindent
\item[(iv)\,]
 Suppose $G$ is as in part~(iii) and let $\widetilde{\h}:=\n_\g({\rm Soc}(\h))$. Then $\widetilde{\h}=\h+\Lie(N)$
is a semisimple maximal Lie subalgebra of $\g$ and $\Lie(N)$ has codimension $1$ in $\widetilde{\h}$.

\smallskip

\noindent
\item[(v)\,] Suppose $\h$ is a maximal Lie subalgebra of $\g$. Then  $\mathcal{D}\cong \sl_2$ when $p=5$ and $\mathcal{D}=W(1;\underline{1})$ when $p=7$. Any two maximal esdp's of $\g$ are $(\Ad\,G)$-conjugate.

\end{itemize}
\end{theorem}

Thus if $G$ is of type ${\rm E}_7$ and $\h$ satisfies the conditions of  Theorem~\ref{thm:esdps}
then $\widetilde{\h}=\n_\g({\rm Soc}(\h))$ is a semisimple maximal Lie subalgebra of $\g$, unique up to $G$-conjugacy, and its socle coincides with ${\rm Soc}(\h)\cong\sl_2\otimes \OO(1;\underline{1})$.  A more precise description of the group $N$ is given in \S~\ref{3.9}.

Let $T$ be a maximal torus $G$ and denote by $\Phi=\Phi(G,T)$ the root system of $G$ with respect to $T$. Let $\Pi$ be a basis of simple roots in $\Phi$ and $\tilde{\alpha}$ the highest root in the positive system $\Phi^+(\Pi)$. For any $\gamma\in \Phi$ we fix a nonzero element $e_\gamma$ in the root space $\g_\gamma$. A Zariski closed connected subgroup $M$ of $G$ is said to be a {\it maximal connected} subgroup of $G$ if it is maximal among Zariski closed proper connected subgroups of $G$.
\begin{theorem}\label{classicalthm} Let $\m$ be a  maximal Lie subalgebra of $\g$ and suppose that $\m$ is semisimple and all its minimal ideals are simple Lie algebras. Then one of the following two cases occurs:
\begin{itemize}
\item[(i)\,]
There exists a semisimple maximal connected subgroup $M$ of $G$ such that $\m=\Lie(M)$.

\smallskip

\noindent
\item[(ii)\,] The group $G$ is not of type ${\rm E}_6$, the Coxeter number of $G$ equals $p-1$, and
$\m$ is $(\Ad\,G)$-conjugate to the Witt subalgebra of $\g$ generated by the highest root vector $e_{\tilde{\alpha}}$ and the regular nilpotent element $\sum_{\alpha\in\Pi}\,e_{-\alpha}$.
\end{itemize}
\end{theorem}
Since all conjugacy classes of maximal connected subgroups
of $G$ are known thanks to earlier work of Dynkin \cite{Dyn}, Seitz \cite{Sei91}, Testerman \cite{Tes88} and Liebeck--Seitz \cite{LS04}, Theorems~\ref{thm:esdps} and \ref{classicalthm} give a complete answer to the problem of determining the maximal subalgebras of $\g$ up to conjugacy.

Let $G$ be a reductive $k$-group and $\g=\Lie(G)$.
Recall that $G$ is said to satisfy the {\it standard hypotheses} if $p$ is a good prime for $G$, the derived subgroup $\D G$ is simply connected, and $\g$ admits a non-degenerate $G$-invariant symmetric bilinear form.
Given a Lie subalgebra $\h$ of $\g$ we denote by ${\rm nil}(\h)$ the largest ideal of $\h$ consisting of nilpotent elements of $\g$. We are able to prove the following corollary of our classification, a Lie algebra analogue of the well known Borel--Tits theorem for algebraic groups.
\begin{cor}\label{B-T}
If $G$ is reductive $k$-group satisfying the standard hypotheses,
then for any Lie subalgebra $\h$ of $\g$ with ${\rm nil}(\h)\ne 0$ there exists a parabolic subgroup $P$ of $G$ such that
$\h\subseteq \Lie(P)$ and ${\rm nil}(\h)\subseteq \Lie(R_u(P))$.
\end{cor}
We stress that Corollary~\ref{B-T} breaks down very badly if we drop some of our assumptions on $G$. Indeed, if $G={\rm PGL}(V)$, where
$\dim V=p$,  then  there exists a $2$-dimensional abelian Lie subalgebra $\h$ of $\g=\mathfrak{pgl}(V)$ with $\h\subset \N(\g)$ whose inverse image in $\gl(V)$ acts irreducibly on $V$. This means that $\h$ cannot be included into a proper parabolic subalgebra of $\g$. In this example $G=\D G$ is a simple algebraic $k$-group of adjoint type.
 On the other hand, if $G$ is a simply connected $k$-group of
 type ${\rm A}$, ${\rm B}$, ${\rm C}$ or ${\rm D}$ and $p$ is good for $G$, then  $\g$ is isomorphic to one of $\sl(V)$, $\so(V)$ or $\sp(V)$ as a restricted Lie algebra. For such Lie algebras, Corollary~\ref{B-T} is a straightforward consequence of the fact that
${\rm nil}(\m)$
annihilates a nonzero proper subspace of $V$.
This indicates that proving Corollary~\ref{B-T} reduces quickly to the case where $G$ is a simple algebraic group of type ${\rm G}_2$, ${\rm F}_4$, ${\rm E}_6$, ${\rm E}_7$ or ${\rm E}_8$; see \S~\ref{classical case}.
Since there are no good substitutes of $V$ for exceptional groups,
our proof of  Corollary~\ref{B-T} relies very heavily on
Theorems~\ref{thm:esdps} and \ref{classicalthm}.

As an immediate consequence of Corollary~\ref{B-T} we obtain the following generalisation of one of the classical results of Lie theory first proved by Morozov in the characteristic zero case; see \cite[Ch.~VIII, \S 10, Th.~2]{Bour1}:
\begin{cor}\label{Mor}
	Suppose $G$ satisfies the standard hypotheses
	and let $\q$ be a Lie subalgebra of $\g=\Lie(G)$ such that ${\rm nil}(\q)\ne 0$ and $\q=\n_{\g}({\rm nil}(\q))$. Then there exists a proper parabolic subgroup $P$ of $G$ such that $\q=\Lie(P)$ and ${\rm nil}(\q)=\Lie(R_u(P))$.
\end{cor}

Indeed,  by Corollary~\ref{B-T}, there is a proper parabolic subgroup $P$ of $G$ such that $\q\subseteq \Lie(P)$ and ${\rm nil}(\q)\subseteq \Lie(R_u(P))$. Let $\u=\Lie(R_u(P))$
and denote by $\n$ the normaliser of ${\rm nil}(\q)$ in $\u$. Since $\q\subseteq \Lie(P)$ normalises both
$\u$ and ${\rm nil}(\q)$, it is easy to see that $\n\subseteq \n_\g({\rm nil}(\q))=\q$ is an ideal of $\q$ consisting of nilpotent elements of $\g$.
Suppose for a contradiction that ${\rm nil}(\q)\ne \u$. Then it follows from Engel's theorem that ${\rm nil}(\q)$ annihilates a nonzero vector of the factor space $\u/{\rm nil}(\q)$ forcing ${\rm nil}(\q)$ to be a {\it proper} ideal of $\n$. On the other hand, since $\n$ consists of nilpotent elements of $\g$ it must be that $\n\subseteq {\rm nil}(\q)$. This contradiction shows that ${\rm nil}(\q)=\u$. But then $\q=\n_\g(\u)$ contains $\Lie(P)$. Since $\q\subseteq \Lie(P)$, we now deduce that $\q=\Lie(P)$ and ${\rm nil}(\q)=\Lie(R_u(P))$, as wanted.

Corollary~\ref{Mor} answers a question posed to one of the authors by Donna Testerman. We finish the introduction by mentioning some interesting
open problems related to classifying the maximal subalgebras
of Lie algebras of simple algebraic groups.

It is immediate from our classification that if $G$ is an exceptional group and $p$ is good for $G$, then the set of all maximal Lie subalgebras of $\g$ splits into finitely many orbits under the adjoint action of $G$. We do not know whether the number of $(\Ad\,G)$-orbits of maximal Lie subalgebras of $\g=\Lie(G)$  is finite in the case where $G$ is a group of type ${\rm B}$, ${\rm C}$, ${\rm D}$ or $G=\SL(V)$ and $p\nmid \dim V$.
The problem is closely related with the fact that in dimension $d=p^{n}-s$, where $n\ge 4$ and $s\in\{1,2\}$, there exist infinitely many isomorphism classes of $d$-dimensional simple Lie algebras over $k$ (this was first observed by Kac in the early  $1970$s). When $p>3$, it follows from the Block--Wilson--Strade--Premet classification theorem that almost all of them (excepting finitely many isomorphism classes) belong to infinite families of filtered Lie algebras of Cartan type $H$.
In order to clarify the situation one would need to describe all irreducible restricted representations $\rho\colon L_p\to \gl(V)$ of the $p$-envelope $L_p$ of every filtered Hamiltonian algebra $L\cong \ad L$ in $\Der(L)$ and then determine the Lie subalgebras of $\gl(V)$ containing $\rho(L_p)$.

When $p\,\vert\dim V$ the main result of \cite{P17} is no longer valid for $\sl(V)$ which contains maximal subalgebras that are neither semisimple nor parabolic. In fact, every maximal subalgebra of $\sl(V)$ acting irreducibly on $V$ is neither
semisimple nor parabolic as it must contain the scalar endomorphisms of $V$. This leads to intriguing representation-theoretic problems.
As a very special example, it is not known at present whether the non-split central extension of the Witt algebra $W(1;\underline{1})$ given by the
Block--Gelfand--Fuchs cocycle
can appear as a maximal subalgebra of $\sl(V)$ for some vector space $V$ with $p\,\vert \dim V$.

If $\mathcal{H}$ is a simple algebraic $k$-group and $\h=\Lie({\mathcal H})$, then to every linear function $\chi\in\h^*$ one can attach at least one irreducible $\h$-module $E$ with $p$-character $\chi$. Let $\rho\colon\,\h\to \gl(E)$
denote the corresponding representation of $\h$. It is known that under mild assumptions of $p$ and $\h$ the dimension of $E$ is divisible by $p^{d(\chi)}$ where $2d(\chi)$ is the dimension of the coadjoint $\mathcal{H}$-orbit of $\chi$. Another challenging open problem which arises naturally in this setting is to determine all pairs $(\chi, \rho)$
for which $\chi\ne 0$ and $k\,{\rm Id}_E+\rho(\h)$ is a maximal Lie subalgebra of $\sl(E)$. It can be shown by using finiteness of the number of unstable coadjoint $\mathcal{H}$-orbits and earlier results of Block, Kac and Friedlander--Parshall that the number of such pairs (up to a natural equivalence relation) is finite.

Very little is known about maximal Lie subalgebras of exceptional Lie algebras $\g$ over fields of bad characteristic. Recent work of Thomas Purslow \cite{Tom} shows that some strange simple Lie algebras which have no analogues in characteristic $p>3$ do appear as maximal subalgebras of $\g$. It seems that a detailed
investigation of the above-mentioned problems could lead to interesting new results in modular representation theory and structure theory of simple Lie algebras.

\section{Preliminaries}\label{Prelim}
\subsection{Basic properties of restricted Lie algebras}\label{intro0} Let $\mathcal{L}$ be a Lie algebra over an algebraically closed field of characteristic $p>0$. We say that
$\mathcal{L}$ is {\it restrictable}
if $(\ad X)^p$ is an inner derivation for every $X\in\mathcal{L}$. Any restrictable Lie algebra $\mathcal{L}$ can be endowed with a $p$th power map (or a $p$-operation) $X\longmapsto X^{[p]}$ such that 
\begin{itemize}
\item[(i)\,] $(\lambda X)^{[p]}\,=\,\lambda^p X^{[p]}$ for all $\lambda\in k$ and $X\in\mathcal{L}$;

\smallskip

\item[(ii)\,] $(\ad(X^{[p]})\,=\,(\ad X)^p$ for all $X\in \mathcal{L}$;

\smallskip

\item[(iii)\,] $(X+Y)^{[p]}\,=\,X^{[p]}+Y^{[p]}+
\textstyle{\sum}_{i=1}^{p-1}\,s_i(X,Y)$ for all $X,Y\in\mathcal{L}$, where $is_i(X,Y)$ is 
the coefficient of $t^{i-1}$ in $\ad(tX+Y)^{p-1}(X)$ expressed as a sum of Lie monomials in $X$ and $Y$.
\end{itemize}
Such a $p$-operation is unique up to a $p$-linear map
$\varphi\colon \mathcal{L}\to \z(\mathcal{L})$
where $\z(\mathcal{L})$ is the centre of $\mathcal{L}$. Indeed, if $\varphi$ is such a map then the operation $X\longmapsto X^{[p]}+\varphi(X)$ also satisfies the properties (i), (ii) and (iii). 
A pair $(\mathcal{L},[p])$, where $\mathcal{L}$ is a restrictable Lie algebra and $[p]$ is a $p$th power map on $\mathcal{L}$, is called a {\it restricted Lie algebra} (or a $p$-Lie algebra). If $\mathcal{L}$
is restrictable and centreless then it admits a unique restricted Lie algebra structure. 
For any $k$-algebra $\mathcal{A}$ (not necessarily associative or Lie) the Lie algebra $\Der(\mathcal{A})$ of all derivations of $\mathcal{A}$ carries a natural restricted Lie algebra structure which assigns to any $D\in \Der(\mathcal{A})$ the $p$th power of
the endomorphism $D$ in $\End(\mathcal{A})$.

Now suppose that $\mathcal{L}$ is a finite dimensional restricted Lie algebra over $k$. A Lie subalgebra  $\mathcal{M}$ of $\mathcal{L}$ is called {\it restricted} if $x^{[p]}\in\mathcal{M}$ for all
$x\in\mathcal{M}$.
For any $x\in \mathcal{L}$ the centraliser $\mathcal{L}_x:=\{y\in\mathcal{L}\,|\,\,[x,y]=0\}$ is a restricted Lie subalgebra of $\mathcal{L}$. Indeed, $[y^{[p]},x]=(\ad y)^p(x)=-(\ad y)^{p-1}([x,y])=0$ for all $y\in\mathcal{L}_x$. If $\mathcal{M}$ is spanned over $k$ by elements $x_1,\ldots, x_r$ such that $x_i^{[p]}\in\mathcal{M}$ for all $1\le i\le r$, then $\mathcal{M}$
is a restricted Lie subalgebra of $\mathcal{L}$. Indeed, it follows from (iii) that $(\sum_{i=1}^r\lambda_ix_i)^{[p]}-
\sum_{i=1}^r\lambda_i^px_i^{[p]}\in[\mathcal{M},\mathcal{M}]$ for all $\lambda_i\in k$.

If $\mathcal{L}=\Lie(S)$, where $S$ is a linear algebraic $k$-group, then the Lie algebra $\mathcal{L}$ identifies with the Lie algebra of all left invariant derivations of the coordinate algebra $k[S]$. Since in characteristic $p>0$ the associative
$p$th power of any left invariant derivation of $k[S]$ is again a left invariant derivation, $\mathcal{L}$ carries a canonical
restricted Lie algebra structure which has the property that
$((\Ad\,g)(x))^{[p]}=(\Ad\, g)(x^{[p]})$ for all $g\in S$ and $x\in\mathcal{L}$. Moreover, if $H$ is a closed subgroup of $S$ then it follows from \cite[Proposition~3.11]{Borel} that $\Lie(H)$ (regarded with its own canonical $p$th power map) identifies with a restricted Lie subalgebra of $\mathcal{L}$.

An element $x$ of a restricted Lie algebra $\mathcal{L}$ is called {\it semisimple} if it lies in the the $k$-span of all $x^{[p]^i}$ with $i\ge 1$.
If $x$ is semismple then $\ad x$ is a diagonalisable endomorphism of $\mathcal{L}$. A semsimple element $x$ of $\mathcal{L}$ is called {\it toral} if $x^{[p]}=x$. If $x$ is toral then all eigenvalues of $\ad x$ lie in $\F_p$.
A restricted Lie subalgebra of $\mathcal{L}$ is called {\it toral} if it consists of semisimple elements of $\mathcal{L}$. Since $k$ is algebraically closed, any toral subalgebra $\t$ of $\mathcal{L}$ is abelian and contains a $k$-basis consisting of toral elements of $\mathcal{L}$. More precisely, the set $\t^{\rm tor}$ of all toral elements of $\t$
is an $\F_p$-form of $\t$, so that $\t\cong \t^{\rm tor}\otimes_{\F_p}k$ as $k$-vector spaces.
From this it is immediate that $\t$ can be generated under the $p$th power map by a single element $x$, that is $ \t\,=\,{\rm span}\{x^{[p]^i}\,|\,\, i\in\Z_{\ge 0}\}$.
The maximal dimension of toral Lie subalgebras of $\mathcal{L}$ 
is often referred to as the {\it toral rank} of $\mathcal{L}$.
If $\mathcal{L}=\Lie(S)$ and $H$ is a torus of $S$ then our remark in the previous paragraph implies that $\Lie(H)$ is a toral subalgebra of $\mathcal{L}$.
Conversely, it follows from \cite[Proposition~11.8]{Borel} and the preceding remark that any toral subalgebra of $\mathcal{L}$ is contained in $\Lie(T)$ for some maximal torus $T$ of $S$. So the toral rank of $\Lie(S)$ coincides with the rank of the group $S$.

Given a finite dimensional semisimple Lie algebra $L$ over $k$ we denote by $L_p$ the $p$-envelope of $L\cong\ad L$ in $\Der(L)$, that is, the smallest restricted Lie subalgebra of $\Der(L)$ containing $L$. The restricted Lie algebra $L_p$ is semisimple and any finite dimensional semisimple
$p$-envelope of $L$ is isomorphic to $L_p$ in the category of restricted Lie algebras. This is immediate from \cite[Definition~1.1.2 and Corollary~1.1.8]{Str04}. The derived subalgebra of $L_p$ coincides with $[L,L]$ and we have that $L=L_p$ if and only if $L$ is restrictable. If $L$ is simple and non-restrictable then the derived subalgebra
of the restricted Lie algebra $L_p$ does not admit a restricted Lie algebra structure.

Now suppose that $G$ is a connected reductive algebraic group over $k$ and let $\g=\Lie(G)$. In contrast with the preceding remark we have the following:
\begin{lemma}\label{ss-der}
If $x$ is a semisimple element of $\g$ then $[\g_x,\g_x]$  is a restricted Lie subalgebra of $\g$.	
\end{lemma}
\begin{proof}
Let $T$ be a maximal torus of $G$ and let $\Phi$ be the root system of $G$ with respect to $T$. 
By the above discussion we may assume without loss that
$x\in \Lie(T)$.
Given $\alpha\in \Phi$ we denote by $U_\alpha$ the unipotent root subgroup of $G$ associated with $\alpha$ and write $S_\alpha$ for the subgroup of $G$ generated by $U_\alpha$ and $U_{-\alpha}$. Each $S_\alpha$ is a simple algebraic subgroup of type ${\rm A}_1$ in $G$.
Since $\Lie(U_\alpha)=k e_{\alpha}$ is a restricted Lie subalgebra of $\g$ it must be that $e_\alpha^{[p]}=0$. Since $\Lie(S_\alpha)$ is a restricted Lie subalgebra of $\g$ isomorphic to $\sl_2(k)$ or $ \mathfrak{pgl}_2(k)$, we have that $[e_\alpha,e_{-\alpha}]^{[p]}\in k[e_\alpha,e_{-\alpha}]$
(if $k$ has characteristic $2$ and the derived subgroup of $G$ is not simply connected then it may happen that $[e_\alpha,e_{-\alpha}]=0$ for some $\alpha\in\Phi$).

Let $\Phi_x=\{\alpha\in\Phi\,|\,\,({\rm d}_1\alpha)(x)=0\}.$ The Lie algebra $\g_x$ is spanned over $k$ by $\Lie(T)$ and
by all $e_\alpha$ with $\alpha\in\Phi_x$. It follows that the derived subalgebra $[\g_x,\g_x]$ is spanned over $k$ by all 
$e_\alpha$ with ${\rm d}_1\alpha\ne 0$ and all $[e_\alpha,e_\beta]$ with $\alpha,\beta\in\Phi_x$. Each element $v$ in this spanning set has the property that $v^{[p]}\in kv$. Our earlier remarks in this subsection now show that $[\g_x,\g_x]$ is a restricted Lie subalgebra of $\g$.	
\end{proof}

\subsection{Transitive Lie subalgebras of the Witt algebra}\label{intro1}
We denote by $\OO(m;\underline{1})$ the truncated polynomial ring $k[X_1,\ldots,X_m]/(X_1^p,\ldots, X_m^p)$
in $m$ variables and write $x_i$ for the image of $X_i$ in $\OO(m;\underline{1})$. The local $k$-algebra $\OO(m;\underline{1})$ inherits a standard degree function from the polynomial ring $k[X_1,\ldots, X_m]$. Given $i\in\Z_{\ge 0}$ we denote by $\OO(m;\underline{1})_{(i)}$ the subspace of all truncated polynomials in $\OO(m;\underline{1})$ whose initial term has standard degree $\ge i$. Each
 $\OO(m;\underline{1})_{(i)}$ is an ideal of $\OO(m;\underline{1})$ and
 $\OO(m;\underline{1})_{(i)}=0$ for $i>m(p-1)$. The maximal ideal of $\OO(m;\underline{1})$ is $\OO(m;\underline{1})_{(1)}$.

The derivation algebra of $\OO(m;\underline{1})$, denoted $W(m;\underline{1})$, is called the $m$th {\it Witt--Jacobson Lie algebra}.
This restricted Lie algebra is a free $\OO(m;\underline{1})$-module of rank $m$ generated by the partial derivatives $\partial_1,\ldots, \partial_m$ with respect to $x_1,\ldots, x_m$.
The subspaces $W(m;\underline{1})_{(i)}:=\sum_{i=1}^m\,\OO(m;\underline{1})_{(i+1)}\partial_i$ with $-1\le i\le m(p-1)-1$ induce a decreasing Lie algebra filtration
$$W(m;\underline{1})=W(m;\underline{1})_{(-1)}\supset W(m;\underline{1})_{(0)}\supset\cdots
\supset W(m;\underline{1})_{(m(p-1)-1)}\supset 0$$ of $W(m;\underline{1})$ which is called
{\it standard}.
The Lie subalgebra $W(m;\underline{1})_{(0)}$ is often referred to as {\it the standard maximal subalgebra} of $W(m;\underline{1})$.
This is due to the fact that for $p>3$ it
can be characterised as the unique subalgebra of minimal codimension in $W(m;\underline{1})$.
Because of that all members of the standard filtration of $W(m;\underline{1})$ are invariant under the action of the automorphism group of $W(m;\underline{1})$.

A restricted Lie subalgebra $\mathcal{D}$ of
$W(m;\underline{1})$ is called {\it transitive} if it does not preserve any
 proper nonzero ideals of
$\OO(m;\underline{1})$. Given a finite dimensional simple Lie algebra $S$ over $k$ and a restricted transitive Lie subalgebra $\mathcal{D}$ of $W(m;\underline{1})$ we can form a natural semidirect product
$$S(m,\mathcal{D})\,:=\,(S\otimes\OO(m;\underline{1}))\rtimes({\rm Id}_S\otimes \mathcal{D}).$$ It is known (and easy to see) that $S(m,\mathcal{D})$ is a semisimple Lie algebra over $k$ and its semisimple $p$-envelope $S(m,\mathcal{D})_p$  is isomorphic to $(S_p\otimes\OO(m;\underline{1}))\rtimes ({\rm Id}_S\otimes \mathcal{D})$ as restricted Lie algebras, where $S_p$ is the $p$-envelope of $\ad S$ in $\Der(S)$.

 The case $m=1$ will play a special role in what follows and we spell out the above in more detail. The {\it Witt algebra} $W(1;\underline{1})=\Der(\OO(1;\underline{1}))$ has $k$-basis $\{x^i\del\,|\,\, 0\le i\le p-1\}$, and the Lie bracket in $W(1;\underline{1})$ is given by  $[x^i\del,x^j\del]=(j-i)x^{i+j-1}\del$ for all $i,j\le p-1$ (here $x^n=0$ for all $n\ge p$). It is well known that $D^p\in kD$ for all $D\in W(1;\underline{1})$; see \cite{P92}, for example. In conjunction with Jacobson's formula for $[p]$th powers this shows that any Lie subalgebra of $W(1;\underline{1})$ is restricted. It is routine to check that a subalgebra $\mathcal{D}$ of $W(1;\underline{1})$ is transitive if and only if it is not contained in the standard maximal subalgebra
 $W(1;\underline{1})_{(0)}$.

Note that if $p=2$ then the Witt algebra is solvable and if $p=3$ then $W(1;\underline{1})\cong \sl_2(k)$.
But things settle for $p>3$ and an old result of Jacobson says that any automorphism of the Witt algebra $W(1;\underline{1})$ is induced by a unique automorphism of $\OO(1;\underline{1})$;
see \cite[Theorems 9 and 10]{Jac}. Using this fact
it is straightforward to describe the conjugacy classes of transitive Lie subalgebras of the Witt algebra under the action of its automorphism group.

\begin{lemma}\label{tran} If $p>3$ then any transitive Lie subalgebra of
$W(1;\underline{1})$ is conjugate under the action of $\Aut(W(1;\underline{1}))$ to one of the following:
$$(1)\ k\partial;\ \ \ (2)\ k(1+x)\partial;\ \ \ (3)\ k\partial
 \oplus k(x\partial);\ \
\ (4)\ k\partial\oplus k(x\partial)\oplus k(x^2\partial)\,\cong\, \sl_2(k);\ \ \ (5)\ W(1;\underline{1}).
$$
\end{lemma}
\begin{proof}
Let $\mathcal{G}={\rm Aut}(W(1;\underline{1}))$ and let $\mathcal{D}$ be a transitive Lie subalgebra of $W(1;\underline{1})$. Then there exists
$D\in\mathcal{D}$ such that $D\not\in W(1;\underline{1})_{(0)}$. Since $D^p\in kD$ we may assume that either
$D^p=0$ or $D^p=D$. If $D^p=0$ (resp. $D^p=D$) then there is $g\in\mathcal{G}$ such that $g(D)=\partial$
(resp. $g(D)=(1+x)\partial$); see \cite[Lemma~4]{P92} and \cite[\S~7]{Str04}. This proves the lemma in the case where $\dim \mathcal{D}=1$.

From now on we may assume that $\mathcal{D}$ has dimension $\ge 2$
and intersects non-trivially with the set
$\{\partial,\,(1+x)\partial\}$. We first suppose that $(1+x)\partial\in\mathcal{D}$. The subspace $k(1+x)\partial$ is a self-centralising torus of $W(1;\underline{1})$  and the group
$\mathcal{G}$ contains a cyclic subgroup $\Sigma$ of order $p-1$ which permutes transitively the set $\F_p^\times (1+x)\partial$ of all nonzero toral elements of $k(1+x)\partial$; see \cite[\S~1]{P92}, for example. As $\dim\mathcal{D}\ge 2$
and $\ad (1+x)\partial$ is diagonalisable, it must be that $(1+x)^i\partial\in \mathcal{D}$ for some $i\in\{0,2,\ldots, p-1\}$.
Since $\Sigma$ permutes transitively the set of all eigenspaces
for $\ad (1+x)\partial$ corresponding to eigenvalues in $\F_p^\times$, there is $\sigma\in\Sigma$ such that $\sigma(\mathcal{D})$ contains both
$(1+x)\partial$ and $\partial$. So we may assume without loss that
 $\partial\in \mathcal{D}$.

Since $\dim\mathcal{D}\ge 2$  there is
$f(x)\partial\in\mathcal{D}$ such that $f(x)=a_rx^r+\cdots a_1x$, where
$a_i\in k$ and $a_r\ne 0$ for some $1\le r\le p-1$. Then $(\ad \partial)^{r-1}(f(x)\partial)\in \mathcal{D}$ yielding $x\partial \in\mathcal{D}$. If $\dim\mathcal{D}=2$ we arrive at case~(3). If $\dim\mathcal{D}\ge 3$
then $\mathcal{D}\cap W(1;\underline{1})_{(1)}$ contains an eigenvector for $\ad (x\partial)$.
Hence $x^r\partial\in\mathcal{D}$ for some $r\in\{2,\ldots, p-1\}$. As $\partial\in\mathcal{D}$, this yields that $x^i\partial\in\mathcal{D}$ for $0\le i\le r$. If $\dim\mathcal{D}=3$ then $r=2$ and we are in case (4).
If $\dim\mathcal{D}\ge 4$ then the above shows that $\mathcal{D}$ contains $x^3\partial$. Since the Lie algebra $W(1;\underline{1})$ is generated by
$\partial$ and $x^3\partial$, we get $\mathcal{D}=W(1;\underline{1})$ completing the proof.
\end{proof}

\subsection{A property of restricted $\sl_2$-modules}\label{3.3}
Let $L$ be a Lie algebra over $k$ and let $V$ be a finite dimensional $L$-module. Given $x\in L$ we set $V_x:=\{v\in V\,|\,\, x. v=0\}.$

In what follows we shall require very detailed information on certain $\sl_2$-triples $\{e, h, f\}$ of $\g$ such that $e^{[p]}=f^{[p]}=0$ and $h^{[p]}=h$. In particular, it will be very useful for us to know that $\dim \g_{h}\le \,\dim \g_{e}$. Since the $k$-span of $\{e,h,f\}$
is a restricted Lie subalgebra of
$\g$ isomorphic to $\sl_2(k)$, we may regard $\g$ as a restricted $\sl_2(k)$-module.

\begin{lemma}\label{sl2}
	Suppose ${\rm char}(k)>2$ and let $V$ be a finite dimensional restricted module over the restricted Lie algebra $\s=\sl_2(k)$.
	If $x$ is a nonzero semisimple element of $\s$ then  $\dim V_x\le \dim V_y$ for any $y\in\s$.
\end{lemma}
\begin{proof} Let $\mathcal{S}=\SL_2(k)$ and let
	$V(m)$ be the Weyl module for $\mathcal{S}$ of highest weight $m\in\Z_{\ge 0}$.
	The Lie algebra $\s=\Lie(\mathcal{S})$ acts on $V(m)$ via the differential  at identity of the rational representation $\mathcal{S}\to\,\GL(V(m))$.
	It is well known that any irreducible restricted $\s$-module is isomorphic to one of the $V(m)$'s with $m\in\{0,1,\ldots, p-1\}$.	
	
	In proving this lemma we may assume that $V$ is an indecomposable
	$\s$-module.	Let $\psi\colon\,\s\to\,\gl(V)$ denote the corresponding representation of $\s$. All such representations are classified in \cite{P91}.
	To be more precise, it is known that either there is a rational representation $\rho\colon\,\mathcal{S}\to\,\GL(V)$ such that $\psi={\rm d}_e\rho$ or $V$ is maximal $\s$-submodule of $V(lp+r)$ for some $l\in\Z_{\ge 1}$ and $r\in\{0,1,\ldots, p-2\}$. In the latter case $\dim V=lp$ and $V$ has two composition factors, $V(r)$ and $V(p-2-r)$, both of which appear $l$ times in any composition series of $V$. If $\psi={\rm d}_e\rho$ then either $V$ is isomorphic to one of $V(m)$ or $V(m)^*$ with $p\nmid (m+1)$ or
	$V$ is a projective indecomposable module over the restricted enveloping algebra of $\s$.
	
	Let $d=\min_{s\in\s}\,\dim V_s$. Looking at the minors of matrices of the endomorphisms $\psi(s)$ with $s\in\s$ one observes that the set $$U:=\{s\in\s\,|\,\,\dim V_s=d\}$$ is nonempty, Zariski open in $\s$, and has the property that $k^\times U=U$.
	If $\psi={\rm d}_e\rho$ then $U$ is also $(\rm Ad\,\mathcal{S})$-stable.
	If $\{e,h,f\}$ is a standard basis of $\s$ then any nonzero semisimple element of $\s$ is $(\rm Ad\,\mathcal{S})$-conjugate to a nonzero multiple of
	$h$.  So if $\psi={\rm d}_e\rho$ then $U$ contains all nonzero semisimple elements of $\s$. This proves the lemma in the present case.
	
	Now suppose $V$ is a maximal submodule of $V(m)$, where $m=lp+r$ and $0\le r\le p-2$. If $r$ is even and $x$ is $(\Ad\,\mathcal{S})$-conjugate to a nonzero multiple of $H$ then $\dim V(r)_x=1$ and if
	$r$ is odd then $\dim V(r)_{x}=0$. Since $x^{[p]}=\lambda x$ for some
	$\lambda\in k^\times$, the endomorphism $\psi(x)$ is diagonalisable.
	In view of our remarks earlier in the proof this yields that $\dim V_x=l$
	for any nonzero semisimple element $x\in\s$. Consequently, $d=l$ completing the proof.
\end{proof}
\begin{remark}\label{R0} Let $\s$, $V$ and $\psi$ be as above.
 It follows from the above-mentioned description of finite dimensional indecomposable restricted $\s$-modules that 
$V$ has a reducible indecomposable direct summand only if $\psi(e)^{p-1}\ne 0$ or $\psi(f)^{p-1}\ne 0$.
\end{remark}
\subsection{Standard $\sl_2$-triples}\label{3.4} In this subsection we review some results on $\sl_2$-triples in exceptional Lie algebras over algebraically closed fields of good characteristics. More information on such $\sl_2$-triples can be found in \cite{HSMax} and \cite{ST} where the notation is slightly different.

It is well known that the nilpotent cone $\N(\g)$ coincides with the set of all $(\Ad\,G)$-unstable vectors of $\g$. Therefore, any nonzero $e\in\N(\g)$
admits a cocharacter $\tau\in X_*(G)$ optimal in the sense of the Kempf--Rousseau theory. The adjoint action of the $1$-dimensional torus
$\tau(k^\times)$ gives rise to a $\Z$-grading $\g=\bigoplus_{\,i\in\Z}\,\g(\tau,i)$ where the subspace
$\g(\tau,i)$ consists of all $x\in\g$ such that $(\Ad\,\tau(\lambda))(x)=\lambda^ix$ for all $\lambda\in k^\times$.
The optimal parabolic subgroup $P(e)\subset G$ of $e$ is independent of the choice of $\tau$ and $\Lie(P(e))=\bigoplus_{\,i\ge 0}\,\g(\tau,i)$.
Since the Killing form $\kappa$ of $\g$ is non-degenerate, we can choose an optimal cocharacter $\tau$ in such a way that $e\in\g(\tau,2)$ and $\g_e\subseteq \Lie(P(e))$; see \cite[Theorem~A]{P03}.
Such optimal cocharacters of $e$ form a single conjugacy class under the adjoint action of the centraliser $G_e\subset P(e)$. 
Furthermore, it follows from \cite[Proposition~2.5]{P03} that they coincide with the so-called associated cocharacters introduced by Jantzen in \cite[5.3]{Jan04}.
The Lie algebra $\Lie(\tau(k^\times))$ is a $1$-dimensional torus of $\g$ spanned by the element $h_\tau:=({\rm d}_e\tau)(1)$ which has the property that $[h_\tau,e]=2e$. The centraliser $C_G(\tau)$ of $\tau(k^\times)$ is a Levi subalgebra of $G$ and $\Lie(C_G(\tau))=\g(\tau,0)$.

Put $\g_e(i)=\g_e\cap \g(\tau,i)$. By \cite[Theorem~A]{P03},
the group $C_e:=G_e\cap Z_G(\tau)$ is reductive and $\mathfrak{c}_e:=\Lie(C_e)=\g_e(0)$.
Furthermore, $G_e=C_e\cdot R_u(G_e)$.
The adjoint $G$-orbit $\OO(e)$ of $e$ is uniquely determined by its weighted Dynkin diagram $\Delta=\Delta(e)$ which depicts the weights of $\tau(k^\times)$ on a carefully selected set of simple root vectors of $\g$.
These diagrams are the same as in the characteristic zero case and they can be found in \cite[pp.~401--407]{Car} along with the Dynkin labels of the corresponding nilpotent $G$-orbits.

Let $T_e$ be a maximal torus of $C_e$ and $L=C_G(T_e)$, a
Levi subgroup of $G$. The Lie algebra $\l'=\Lie(\D L_e)$ is $\tau(k^\times)$-stable and contains $h_\tau$. Moreover, $e$ is distinguished in $\l'$, that is $e\in\l'(\tau,2)$ and
$\dim \l'(\tau, 0)=\dim \l'(\tau,2)$; see \cite[2.3--2.7]{P03} for detail.
Since $\g_e\subseteq \Lie(P(e))$, and $\dim \l'(\tau, -2)=\dim \l'(\tau,2)$,
the map $\ad\,e\colon\, \l'(\tau,-2)\to\,\l'(\tau, 0)$ is bijective. As $h_\tau\in\l'(\tau,0)$, there is a unique element $f\in\l'(\tau,-2)$ such that $[e,f]=h_\tau$. By construction, $\{e,h_\tau,f\}$ is an $\sl_2$-triple in $\g$. Since $f^{[p]}$ commutes with $e$ for $p>2$, it lies in $\g_e(\tau,-2p)$. Since $\tau$ has nonnegative weights on $\g_e$, this yields
$f^{[p]}=0$.
\begin{defn}
	An $\sl_2$-triple $\{e',h',f'\}$ of $\g$ is called {\it standard} if it is
	$(\Ad\,G)$-conjugate to one of the $\sl_2$-triples $\{e,h_\tau, f\}$ described above.
\end{defn}
If $\{e,h,f\}$ is a standard $\sl_2$-triple, then necessarily $e\in\N(\g)$ and $f^{[p]}=0$.
However, it may happen in some small characteristics that $e^{[p]}\ne 0$. In particular, if $e^{[p]}\ne 0$ then $e$ and $f$ belong to different nilpotent $G$-orbits. On the other hand, if $e^{[p]}=0$ then there exists a connected subgroup $\mathcal{S}$ of type ${\rm A}_1$ in $G$ such that $\Lie(\mathcal{S})=ke\oplus kh\oplus kf$; see \cite{McN04}. In that case $e$ and $f$ are $(\Ad\,G)$-conjugate.

Our earlier remarks in this subsection show that any nilpotent element of $\g$ can be included into a standard $\sl_2$-triple.
Now suppose $\{e,h,f\}$ is an arbitrary $\sl_2$-triple in $\g$ with $e\in\N(\g)$ and $h$ semisimple. Let $\tau$ be an optimal cocharacter for $e$ such that $e\in\g(\tau,2)$. All eigenvalues of the toral element $h_\tau$ belong to $\F_p$ and we write $\g(h_\tau,\bar{i})$ for the eigenspace of $\ad\,h_\tau$ corresponding to eigenvalue $\bar{i}\in\F_p$. It is straightforward to see that
\begin{equation}\label{h-tau}
\g(h_\tau,\bar i)\,=\,\textstyle{\bigoplus}_{\,j\in\Z}\,\g(\tau,i+jp).
\end{equation}
Since $e$ and $h-h_\tau$ commute, $h$ is a semisimple element element of the restricted Lie algebra
$kh_\tau\oplus \g_e$. Since $\g_e=\Lie(G_e)$ the latter coincides with the Lie algebra of the normaliser $N_G(k e)=\tau(k^\times)\cdot G_e$. As $\tau(k^\times)\cdot T_e$ is a maximal torus of $N_G(ke)$ contained in $\tau(k^\times)\cdot C_e$, it follows from \cite[11.8]{Borel} that
$h$ is conjugate under the adjoint action of $N_G(ke)$ to an element of $kh_\tau\oplus \g_e(0)$.

So assume from now that $h\in kh_\tau+\g_e(0)$. Then $h-h_\tau\in \g_e(0)\cap [e,\g]$. If $h\ne h_\tau$ then the linear map $(\ad\,h)^2\colon\,\g(\tau,-2)\to\,\g(\tau,2)$ is not injective. The computations in \cite{P95} then imply that the orbit $\OO(e)$ has Dynkin label
${\rm A}_{p-1}{\rm A}_r$ for some $r\ge 0$. In other words, $e$ is a regular nilpotent element of $\Lie(L)$ where $L$ is a Levi subgroup of type ${\rm A}_{p-1}{\rm A}_r$ in $G$.
\begin{remark}\label{nonstandard}
	The preceding remark implies that if $\g$ contains a non-standard $\sl_2$-triple $\{e,h,f\}$ with $e\in\N(\g)$ and $h$ semisimple, then
	$e\in\OO({\rm A}_{p-1}{\rm A}_r)$ for some $r\ge 0$. As a consequence, $G$ is a group of type ${\rm E}$ and $p\in\{5,7\}$.
\end{remark}
\subsection{A remark on exponentiation}\label{3.8} Let $K$ be an algebraically closed field of characteristic $0$. In this subsection we assume that $G$ is a simple, simply connected algebraic group over an algebraically closed field $k$ of good characteristic $p>0$ and we write
$G_K$ for the  simple, simply connected algebraic group over $K$ of the same type as $G$. Both groups are obtained by base change from a Chevalley group scheme $G_\Z$.
The Lie algebra
$\g$ of $G$ is obtained by base change from a minimal admissible lattice
$\g_\Z$ in the simple Lie algebra $\g_K:=\Lie(G_K)$.
For any $p$-power $q\in\Z$ the field $k$ contains a unique copy of the finite field $\F_q$ and the finite Lie algebra $\g_{\F_q}:=\g_\Z\otimes_\Z \F_q$ is an $\F_q$-form
of $\g$ closed under taking $[p]$th powers in $\g$. The {\it restricted nullcone}
$\N_p(\g)$ consists of all $x\in\g$ with $x^{[p]}=0$. This is a Zariski closed, conical subset of the nilpotent variety $\N(\g)$ and it arises naturally when one studies exponentiation. Indeed, if $U$ is a one-parameter unipotent subgroup of $G$ then $\Lie(U)$ is a $1$-dimensional
$[p]$-nilpotent restricted subalgebra of $\g$ and hence lies in $\N_p(\g)$.

It is well known that each nilpotent $(\Ad\,G)$-orbit $\OO$ has a  representative $e\in \g_{\F_p}$ such that $e=\tilde{e}\otimes_\Z 1$
for some nilpotent element $\tilde{e}$ of $\g_K$ contained in $\g_\Z$. By \cite[2.6]{P03}, one can choose $\tilde{e}$  in such a way that the unstable vectors $e\in\g$ and $\tilde{e}\in\g_K$ admit optimal cocharacters obtained by base-changing a cocharacter $\tau\in X_*(G_\Z)$. Moreover, $\tau(\mathbb{G}_{m,\Z})$ is contained in a split maximal torus $T_\Z$ of $G_\Z$ and $\dim_k\OO=\dim_K\,(\Ad\,G_K)\,\tilde{e}$.
Various  properties of the cocharacters $\tau$ have already been discussed in \S~\ref{3.4} and we are going to use the notation introduced there for both $e$ and $\tilde{e}$. In particular, we
write $\g_{\tilde{e},K}$ for the centraliser of $\tilde{e}$
in $\g_K$ and $\g_{\tilde{e},K}(i)$ for the intersection
of  $\g_{\tilde{e},K}$ with the $i$-weight space $\g_{K}(\tau, i)$ of $\Ad\,\tau(K^\times)$.
More generally, given a commutative ring $A$ with $1$, we set $\g_A:=\g_\Z\otimes_\Z A$. If $A\subseteq k$, we put
$\g_{e,A}:=\,\g_{e}\cap\g_{A}$ and
$\g_{e,A}(i):=\,
\g(\tau,i)\cap \g_{e,A}$. If $A\subseteq K$, we define
$\g_{\tilde{e},A}$ and $\g_{\tilde{e},A}(i)$ in a similar fashion.

Our next result shows that by modifying the exponentiation techniques of \cite[\S~1]{Tes} one can construct one-parameter unipotent subgroups of $G_e$ whose Lie algebras are spanned by prescribed elements of $\N_p(\g)$ contained in the nilradical of $\g_{e}$. These subgroups respect filtrations associated with optimal cocharacters of $e$.
\begin{prop}\label{expo}
	Let $e$ be a nonzero nilpotent element of $\g$ and let $\tau$ be an optimal cocharacter for $e$ such that $e\in\g(\tau,2)$.  If $d$ is a positive integer, then for any nonzero $x\in
	\bigoplus_{i\ge d}\,\g_{e}(i)$ with $x^{[p]}=0$  there exists a collection of  endomorphisms $\{X^{(i)}\,|\,\,0\le i\le p^2-1\}$ of $\g$
	and a one-parameter unipotent subgroup $\mathcal{U}_x=\{x(t)\,|\,\,t\in k\}$
	of $G_e$ with Lie algebra $kx$
	such that $X^{(i)}=\frac{1}{i!}(\ad x)^i$ for $0\le i\le p-1$
	and
	$$\big(\Ad\,x(t)\big)(v)=\textstyle{\sum}_{i=0}^{p^2-1}\,t^iX^{(i)}(v)\qquad
	\quad\big(\forall\,v\in\g\big).$$ Furthermore, each endomorphism $X^{(i)}$ can be expressed as a sum of weight vectors of weight $\ge d i$ with respect to the natural action of the torus $\tau(k^\times)$ on $\End(\g)$.
\end{prop}
\begin{proof}
First assume the root system of $G$ is classical. This case is more elementary.  The group $G$ admits a rational representation $\rho\colon\,G\to \GL(V)$ defined over $\F_p$ whose kernel is central and whose image is either $\SL(V)$ or the stabiliser in $\SL(V)$ of a non-degenerate bilinear form $\Psi$ on $V$. Let $\bar{X}=({\rm d}_e\rho)(x)$ and $\bar{X}^{(i)}=\frac{1}{i!}\bar{X}^i$, where $0\le i\le p-1$.
	Since $x^{[p]}=0$ and  ${\rm d}_e\rho\colon\,\g\to \gl(V)$ is a faithful restricted representation of $\g$,  the exponentials $\bar{x}(t):=\exp(t\bar{X})\,=\,\sum_{i=0}^{p-1}\,t^i\bar{X}^i$ with $t\in k$ form a one-parameter unipotent subgroup of $\SL(V)$. Moreover, each $\bar{X}^j$ with $0\le j\le p-1$ is a sum of weight vectors of weight $\ge d\cdot j$ with respect to the conjugation action of $\rho(\tau(k^\times))$ on $\End(V)$.
	If $\rho(G)$ fixes $\Psi$ then $\bar{X}$ is skew-adjoint with respect to $\Psi$ and hence $\big\{\bar{x}(t)\,|\,\,t\in k\big\}$ is a one-parameter unipotent subgroup of $\rho(G)$. Also,
	\begin{equation}
	\label{Ad}
	\big(\Ad\,\bar{x}(t)\big)(Y)\,=\,\exp(t\bar{X})\cdot Y\cdot \exp(-t\bar{X})\,=\, \textstyle{\sum}_{r=0}^{2p-2}\,t^r\big(\sum_{i=0}^r\,(-1)^{i}\bar{X}^{(r-i)}\cdot Y\cdot \bar{X}^{(i)}\big)\end{equation}
	for all $Y\in \sl(V)$. It should be mentioned here that $2p-2\le p^2-1$ for any prime number $p$.
	
	Restricting $\Ad\,\bar{x}(t)$ to $({\rm d}_e\rho)(\g)$ and identifying the latter with $\g$ gives rise to a one-parameter unipotent subgroup of ${\rm Aut}(\g)^\circ$ which we call $U_x$.
	As $x=\bar{X}$ commutes with $e$, it follows from (\ref{Ad}) that $U_x$ fixes $e$. Taking the identity component of the inverse image of $U_x$ under a central isogeny $G\twoheadrightarrow {\rm Aut}(\g)^\circ$
	we then obtain a one-parameter unipotent subgroup of $G_e$ that satisfies all our requirements.

Now suppose $G$ is exceptional. By the remarks immediately preceding the proposition, we may assume that $e=\tilde{e}\otimes_\Z 1$ and $\tau\in X_*(G_\Z)$.		
	Let $\mathcal{X}$ be the set of all tuples $\big(x,X^{(p)},\ldots, X^{(p^2-1)}\big)$ in $\big(\N_p(\g)\bigcap\bigoplus_{i\ge d}\,\g_e(i)\big)\times \End(\g)^{p^2-p}$ such that each
	$X^{(i)}\in\bigoplus_{j\ge d i}\big(\End(\g)\big)(\tau,j)$ annihilates $e$ and
	the set
	$$\Big\{\sum_{i=0}^{p-1}\frac{t^i}{i!}(\ad x)^i+\sum_{i=p}^{p^2-1}t^iX^{(i)}\,|\,\,t\in k\Big\}$$
	forms a one-parameter subgroup of ${\rm Aut}(\g)$. If we choose $k$-bases of $\g$ and $\End(\g)$ contained in $\g_{\F_p}$ and $\End(\g_{\F_p})$, respectively, then the above conditions	
	can be rewritten in the form of polynomial equations with coefficients in $\F_p$ on the coordinates of $x$ and the $X^{(i)}$'s. In other words, $\mathcal{X}$ is a Zariski closed subset of
	$\big(\N_p(\g)\bigcap\bigoplus_{i\ge d}\,\g_e(i)\big)\times \End(\g)^{p^2-p}$
	defined over $\F_p$. The projection
	$\pi\colon\,\mathcal{X}\longrightarrow 	\N_p(\g)\bigcap\bigoplus_{i\ge d}\,\g_e(i)$ sending $\big(x,X^{(p)},\ldots, X^{(p^2-1)}\big)$ to $x$ is
	a morphism of affine varieties defined over $\F_p$. Evidently, we wish to show this map is surjective.
	
	Let $k_0$ denote the algebraic closure of $\F_p$ in $k$. We claim that \begin{equation}\tag{*}\pi\colon \mathcal{X}(k_0)\longrightarrow\, \N_p(\g_{k_0})\bigcap\textstyle{\bigoplus}_{i\ge d}\,\g_{e,k_0}(i)\ \ \text{ is surjective.}\end{equation}
	Given the claim, since $k_0$ is an algebraically closed subfield of $k$, it follows from general results of algebraic geometry that the morphism $\pi\colon \mathcal{X}(k)\longrightarrow 	\N_p(\g)\bigcap\bigoplus_{i\ge d}\,\g_e(i)$ is surjective as well; see \cite[Exercise~10.6]{GW}, for example.
	This means that a one-parameter unipotent group $U_x$ with the required properties exists for any $x\in	\N_p(\g)\bigcap\bigoplus_{i\ge d}\,\g_e(i)$.
	Since $G$ is simply connected and $p$ is good for $G$, there exists a central isogeny $\iota\colon\, G\twoheadrightarrow {\rm Aut}(\g)^\circ$. So we can take for $\mathcal{U}_x$ the identity component of $\iota^{-1}(U_x)$.
	
The claim will follow if we can show that $\pi_{\F_q}:\mathcal{X}(\F_q)\longrightarrow \N_p(\g_{\F_q})\bigcap\bigoplus_{i\ge d}\,\g_{e,\F_q}(i)$ is surjective, since $k_0$ is the union of its finite subfields. Thus we  assume that $x\in \g_{\F_q}$ for some $q=p^n$. As usual, we denote by $\Q_p$ and $\Z_p$ the field of $p$-adic numbers and the ring of $p$-adic integers, respectively. Let $K$ be an algebraic closure of $\Q_p$ and let
	$L/\Q_p$ be an unramified Galois extension of degree $n$ (we can take for $L$ the field $\Q_p(\zeta_n)$ where $\zeta_n$ is a primitive
	$(p^n-1)$st root of unity in $K$).
	Let $A$ be the ring of integers
	of $L$.  Since the extension $L/\Q_p$ is unramified, the field
	$A/pA$ has degree $n$ over its subfield $\Z_p/p\Z_p\cong \F_p$ Therefore, $A/pA\cong \F_q$.
	
	By construction, $A$ is a local ring with maximal ideal $pA$ and any bad prime for $G$ is invertible in $A$. Since the torus $T_\Z$ is split over $\Z$ and $\tau(\mathbb{G}_m)\subset T_\Z$ we have that $\g_\Z\,=\,\bigoplus_{i\in\Z}\,\g_\Z(\tau,i)$.	
	Since $\tau(K^\times)$ is optimal for $\tilde{e}$, arguing as in \cite[p.~285]{Spa} one observes that each $A$-module
	$\big[\tilde{e},\g_A(\tau, i)\big]$ is a direct summand of $\g_A(\tau, i+2)$
	and $[\tilde{e},\g_A(\tau, i)]\,=\,\g_A(i+2)$ for all $i\ge 0$. From this it is follows that each $\g_{\tilde{e},A}(i)$ is a direct summand of the free $A$-module $\g_A(\tau, i)$. Since $A/pA\cong \F_q$ the natural map $\g_A\twoheadrightarrow \g_A\otimes_A\F_q\cong\g_{\F_q}$ sends $\tilde{e}$ onto $e$ and $\g_{\tilde{e},A}$ onto $\g_{e,\F_q}$. It follows that
	there exists
	an element $\tilde{x}\in \bigoplus_{i\ge d}\,\g_{\tilde{e},A}(i)$ such that
	$x=\tilde{x}\otimes_A 1$.
	
	Suppose $G$ is an exceptional group.  Since $p$ is a good prime for $G$, the tables in \cite{Ste} show that
	$(\ad \tilde{x})^{p^2-1}=0$ unless either
	$G$ is of type ${\rm E}_8$, $p=7$, and $e$ is regular or
	$G$ is of type ${\rm E}_7$, $p=5$, and $e\in\N(\g)$ is regular or subregular.
	In view of \cite[pp.~122, 185]{LT11} this implies that if $(\ad \tilde{x})^{p^2-1}\ne 0$ then $\g_e(i)=0$ for $i$ odd and $\g_e(2)=ke$. Since in all these cases $e^{[p]}\ne 0$ and $x\in\bigoplus_{i\ge d}\,\g_e(i)$ lies in $\N_p(\g)$ by our assumption, we must have $d\ge 4$. From this it is immediate that $(\ad\tilde{x})^{p^2-1}=0$ in all cases.
	
	We now put
	$\tilde{X}^{(i)}:=\frac{1}{i!}(\ad \tilde{x})^i$ for $0\le i\le p^2-1$.
	The set $U_{\tilde{x}}$ of all
	linear operators \begin{equation*}
	\exp(t\,\ad\tilde{x})\,=\,\textstyle{\sum}_{i=1}^{p^2-1}
	\,t^i\tilde{X}^{(i)}\end{equation*} with $t\in K$ forms a one-parameter unipotent subgroup of ${\rm Aut}(\g_{K})^\circ$.
	Since $(\ad x)^p=0$ and $\ad x=\ad(\tilde{x}\otimes_A 1)=(\ad \tilde{x})\otimes_A 1$ as endomorphisms of $\g_{\F_p}$ we have that $(\ad \tilde{x})^p( \g_{A}) \subseteq  p\cdot \g_A$.
	But then $\frac{1}{p}(\ad \tilde{x})^p(\g_A)\subseteq \g_A$.
	Since $(p-1)!$ is invertible in $A$, it follows that $\tilde{X}^{(p)}$ preserves $\g_A$. Any positive integer $n\le p^2-1$ can be uniquely presented as $n=ap+b$ with $0\le a,b\le p-1$. Since $n!\,=\,p^a\cdot r_n$ for some $r_n\in\Z$ coprime to $p$ and
	$(\ad \tilde{x})^n\,=\,\big((\ad\tilde{x})^p\big)^a\cdot(\ad \tilde{x})^b$, we have that $$\tilde{X}^{(n)}\,=\,\frac{1}{r_n}\big(\tilde{X}^{(p)}\big)^a\cdot
	\Big(\tilde{X}^{(1)}\big)^b, \quad\ r_n\in A^\times.$$
	This implies that each endomorphism $\tilde{X}^{(i)}$ preserves $\g_A$.
	As a consequence, the closed subgroup $U_{\tilde{x}}$ of ${\rm Aut}(\g_{K})^\circ$ is defined over $A$.
	We now set
	$X^{(i)}:=\tilde{X}^{(i)}\otimes_A 1$. In view of our earlier remarks
	it is straightforward to check that the collection of endomorphisms $\{X^{(i)}\,|\,\,0\le i\le p^2-1\}$ of $\g$ possesses all required properties.
	Since the unipotent group $U_{\tilde{x}}$ is defined over $A$, the
	set $U_x:=\big\{\sum_{i=0}^{p^2-1}\,t^iX^{(i)}\,|\,\,t\in k\big\}$, obtained by base-changing $U_{\tilde{x}}$, forms a one-parameter unipotent subgroup of
	${\rm Aut}(\g)^\circ$.
	
	The claim is proved.
\end{proof}
\begin{remark}
	It seems plausible that one can always find a one-parameter subgroup $\mathcal{U}_x$ in $G_e$ satisfying the conditions of Proposition~\ref{expo} and such that $X^{(i)}=0$ for $i>2p-2$. By \cite{McN04},
	for any nonzero $x\in\N_p(\g)$ the optimal parabolic $P(x)$ contains a nice one-parameter subgroup $U_x$ with $\Lie(U_x)=kx$.
	In fact, $U_x$ lies in a connected subgroup of type ${\rm A}_1$ whose Lie algebra contains $x$. By \cite{Sei},
	the number of nonzero $X^{(i)}$'s with $i>0$ associated with  $U_x$ is always bounded by $2p-2$. Moreover, it follows from \cite[Lemma~4.2 and Corollary~4.3(i)]{Sob} that
	if $[x,e]=0$ for some nonzero $e\in\N(\g)$ then $U_x\subseteq G_e$. However, it is not clear from the constructions in {\it loc.\,cit.} that the endomorphisms $X^{(i)}$ with $i\ge p$ coming from the distribution algebra of $U_x$ have the desired weight properties  with respect to an optimal cocharacter for $e$.	
\end{remark}






\section{Lie subalgebras with non-semisimple socles and exotic semidirect products}\label{esdps}
\subsection{The general setup}\label{setup} In this section we always assume that $G$ is an exceptional algebraic group of rank $\ell$ defined over an algebraically closed field $k$ of characteristic $p>0$. We let $\h$ be a semisimple restricted Lie subalgebra of $\g=\Lie(G)$ whose socle is not semisimple. Since $\h$ is restricted, it follows from  Block's theorem
\cite[Corollary~3.3.5]{Str04} that $\h$ contains a minimal ideal $I$ such that $I\cong S\otimes \OO(m;\underline{1})$ for some simple Lie algebra $S$ and $m\ge 1$ and $\h/\c_\h(I)$ is sandwiched between $S\otimes \OO(m;\underline{1})$ and $({\rm Der}(S)\otimes
\OO(m;\underline{1}))\rtimes ({\rm Id}_S\otimes W(m;\underline{1}))$. Moreover, the canonical projection
$\pi\colon\,\h\to W(m;\underline{1})$ maps $\h$ onto a {\it transitive} Lie subalgebra of $W(m;\underline{1})$. Recall the latter means that $\pi(\h)\subseteq W(m;\underline{1})$ does not preserve any nonzero proper ideals of $\OO(m;\underline{1})$.

We identify $I$ with $S\otimes \OO(m;\underline{1})$ and let $\mathcal{I}$ and $\mathcal{S}$ denote the $p$-envelopes of $I$ and $S\otimes 1$ in $\g$, respectively. Since $\h$ is restricted, $\mathcal{I}$ is an ideal
of $\h$. Since $\h$ is semisimple and $\z(\mathcal{I})$ is an abelian ideal of $\h$, it must be that $\z(\mathcal{I})=0$.
Since $\z(\mathcal{S})\subseteq \z(\mathcal{I})$ due to the nature of Lie multiplication in $S\otimes \OO(m;\underline{1})$,
we must have $\z(\mathcal{S})=0$. Our
discussion in Section~\ref{Prelim} then shows that $\mathcal{S}\cong S_p$ as restricted Lie algebras.

Since $S\otimes 1$ is a simple Lie subalgebra, it follows from \cite[Theorem~1.3]{HSMax} that either $S\cong\Lie(\mathcal{H})$
for some simple algebraic $k$-group $\mathcal{H}$ or $S\cong \psl_{rp}$ for $r\ge 1$ or $S\cong W(1;\underline{1})$. In any event, $S$ is a restricted Lie algebra, so that $S=S_p$. Since $\psl_{rp}=\sl_{rp}/\z(\sl_{rp})$ contains toral subalgebras of dimension $rp-2$ and $\rk(G)\le 8$, the case $S\cong \psl_{rp}$ may occur only if $r=1$.

We let $\t$ be a toral subalgebra of maximal dimension in the restricted Lie algebra $S$ and denote by $\t_{\rm reg}$ the set of all $t\in \t$ with the property that $\t={\rm span}\,\{t^{[p]^i}\,|\,i\ge 1\}$. At $\t^{\rm tor}$ is
an $\F_p$-form of $\t$, it is a routine exercise to check that $\t_{\rm reg}$ is nonempty and Zariski open in $\t$.

The above discussion then shows that $\t\otimes 1$ is a toral
subalgebra of $\g$. We pick $z\in \t_{\rm reg}$.
Since $\ad\,(z\otimes 1)$ is semisimple and
$[z\otimes 1,\h]\subseteq I$ we also have that
$\h=\h_{z\otimes 1}+I$, where $\h_{z\otimes 1}$ denotes the centraliser $\c_\h(z\otimes 1)$ of $z\otimes 1$ in $\h$. The latter yields
\begin{equation}\pi(\h)=\pi(\h_{z\otimes 1})\label{projeq}\end{equation} showing that $\pi(\h_{z\otimes 1})$ is a transitive Lie subalgebra of $W(m;\underline{1})$. Since $\c_\h(I)\subseteq \h_{z\otimes 1}$, the factor algebra $\h_{z\otimes 1}/\c_\h(I)$ identifies with a Lie subalgebra of $(\Der(S)_z\otimes\OO(m;\underline{1}))\rtimes ({\rm Id}_S\otimes \pi(\h_{z\otimes 1}))$.

Since $z\otimes 1$ is a semisimple element of $\g$, Lemma~\ref{ss-der} shows  that $[\g_{z\otimes 1},\g_{z\otimes 1}]$ is a restricted Lie subalgebra of $\g$.
On the other hand,
our characterisation of $S$ yields that $\t$ is a
self-centralising torus of $S$. Hence $I_{z\otimes 1}=\t\otimes\OO(m;\underline{1})$. Let $\OO(m;\underline{1})_{(1)}$ denote the maximal ideal of the local ring $\OO(m;\underline{1})$.
Since $\t\otimes \OO(m;\underline{1})_{(1)}$ is stable under the action of $\Der(S)_z\otimes \OO(m;\underline{1})$
and $\pi(\h_{z\otimes 1})$ is a transitive Lie subalgebra
of $W(m;\underline{1})$, there exist $u\in \h_{z\otimes 1}$ and $v\in k z\otimes \OO(m;\underline{1})_{(1)}$ such that
\begin{equation*}
[u,v]\,\equiv z\otimes 1 \mod \t\otimes \OO(m;\underline{1})_{(1)}.\end{equation*}
Since $\h\cong \h_p$ is semisimple, the uniqueness of the restricted Lie algebra structure on $\h$ gives $\big(S\otimes \OO(m;\underline{1})_{(1)}\big)^{[p]}=0$. As $\t\otimes \OO(m;\underline{1})$ is abelian and $[\g_{z\otimes 1},\g_{z\otimes 1}]$ is restricted, we deduce that $[u,v]^{[p]}=z^{[p]}\otimes 1$ lies in $[\g_{z\otimes1},\g_{z\otimes 1}]$.
Since $z\in \t_{\rm reg}$ and $\z(\g_{z\otimes 1})$ is restricted, we obtain that
\begin{equation}\label{derived}\t\otimes 1\subseteq \z(\g_{z\otimes 1})\cap
[\g_{z\otimes 1},\g_{z\otimes 1}].\end{equation}
\subsection{Describing the socle of $\h$}\label{3.2}
In this subsection we are going to give a more precise description of the socle of $\h$. A Lie algebra $L$ is said to be {\it decomposable} if it can be presented as a direct sum of two commuting nonzero ideals of $L$.
\begin{lemma}\label{socle}
Let $\h$, $I$, $z$ and $m$ be as in \S~\ref{setup}. Then $G$ is of type ${\rm E}$ and the following hold:
\begin{itemize}
\item[(i)\ ] $\g_{z\otimes 1}$ is a Levi subalgebra of type ${\rm A}_{p-1}{\rm A}_{\ell-p}$ in $\g$ and $\z(\g_{z\otimes 1})=k(z\otimes 1)$.

\smallskip
\item[(ii)\ ] $S$ is either $\sl_2$ or $W(1;\underline{1})$ and $m=1$.

\smallskip

\item[(iii)\ ] $\h$ has no minimal ideals isomorphic to $\psl_{rp}$.

\smallskip

\item[(iv)\ ] $I$ is the unique non-simple minimal ideal of $\h$ and  ${\rm Soc}(\h)=I\oplus \c_\h(I)$. Any minimal ideal of $\c_\h(I)=\soc(\c_\h(I))$ is isomorphic to the Lie algebra of a simple algebraic $k$-group.

\smallskip

\item[(v)\ ] If $\c_\h(I)\ne 0$ then $\h$ is decomposable. More precisely, $\h\cong (\h/\c_\h(I))\oplus \c_\h(I)$ as Lie algebras and
$\h/\c_\h(I)\cong (S\otimes \OO(1;\underline{1}))\rtimes ({\rm Id}_S\otimes \mathcal{D})$ for some transitive Lie subalgebra $\mathcal{D}$ of $W(1;\underline{1})$.
\end{itemize}
 \end{lemma}
\begin{proof}
Replacing $z\otimes 1$ by an $(\Ad G)$-conjugate we may assume that $\g_{z\otimes 1}=\Lie(G_{z\otimes 1})$ is a standard Levi subalgebra of $\g$. Let $a=\dim\big([\g_{z\otimes 1},\g_{z\otimes 1}]\cap\z(\g_{z\otimes 1})\big)$.
By \cite[2.1]{PSt}, the restricted Lie algebra $\g_{z\otimes 1}$ decomposes into a direct sum of its restricted ideals each of which either has form $\gl_{rp}$ for some $r\ge 0$ or is isomorphic to $\Lie(\mathcal{H})$ for some simple algebraic $k$-group $\mathcal{H}$.
Since $\dim\z(\gl_{rp})=\dim\z(\sl_{rp})=1$ for all $r\ge 1$ and each $\Lie(\mathcal{H})$ is simple,
$a$ coincides with the number of
irreducible components of type ${\rm A}_{rp-1}$ with $r\ge 1$ of the standard Levi subgroup $G_{z\otimes 1}$.

In view of \eqref{derived} we have that $a\ge \dim\t$. In particular, $a\ge 1$.
Since $p$ is a good prime for $G$, examining the subdiagrams of the Dynkin diagram of $G$ yields that $G$ is of type ${\rm E}$ and $r=1$. Since $\t$ is a torus of maximal dimension in $S$ and $1=d\ge\dim\t$, we now deduce that $S$
is either $\sl_2$ or $W(1;\underline{1})$.
Since $z\otimes 1$ lies in the Lie algebra of the ${\rm A}_{p-1}$-component of $G_{z\otimes 1}$, it also follows that
 $G_{z\otimes 1}$ has type ${\rm A}_{p-1}A_{\ell-p}$ proving (i).

Suppose $I'$ is a simple ideal of $\h$ isomorphic to $\psl_{rp}$ for some $r\ge 1$. Then $[I,I']\subseteq I\cap I'=0$ showing that $I'\subseteq \g_{z\otimes 1}$. The restriction of the Killing form $\kappa$ of $\g$ to the Levi subalgebra $\g_{z\otimes 1}$  is non-degenerate. Since $\dim I'\ge p^2-2>(\dim \g_{z\otimes 1})/2$ by part~(i), it follows that the Lie algebra $\psl_{rp}$
admits a faithful representation with a nonzero trace form. This, however, contradicts \cite{Bl}, proving (iii). In view of \cite[Lemma~2.7]{BGP} and our earlier remarks this entails that any minimal ideal
$I'=S'\otimes \OO(m';\underline{1})$ of $\h$ has the property that $\Der(S')=\ad S'$.
Applying \cite[Corollary~3.6.6]{Str04} then yields that $\h$ is sandwiched between
${\rm Soc}(\h)=\bigoplus_{i=1}^d(S_i\otimes\OO(m_i;\underline{1}))$ and $\bigoplus_{i=1}^d\,(S_i\otimes \OO(m_i;\underline{1}))\rtimes ({\rm Id}_{S_i}\otimes W(m_i;\underline{1}))$. For $1\le i\le d$ we denote by $\pi_i$ the projection from $\h$ to $W(m_i;\underline{1})$. Each $\pi_i(\h)$ acts transitively on $\OO(m_i;\underline{1})$.

To keep the notation introduced earlier we assume that $I_1=I$ so that $m_1=m$, $S_1=S$ and $\pi_1=\pi$. For $2\le i\le d$ let $I_i$ be the minimal ideal $S_i\otimes \OO(m_i;\underline{1})$.
Since $[I,I_i]=0$ and $\pi_i(\h)=\pi_i(\h_{z\otimes 1})$ by (\ref{projeq}), each $I_i$ is a minimal ideal of
$\h_{z\otimes 1}$. Since $\g_{z\otimes 1}\cong \gl_p\times \sl_{l-p}$ acts faithfully on a vector  space of dimension $\ell +1=p+(\ell-p+1)$
and $I_i=[I_i,I_i]$, the Lie algebra $\h_{z\otimes 1}$ affords an irreducible restricted representation $\rho_i$
of dimension $\le \ell+1$ such that $\rho(I_i)\ne 0$. We let $V_i$ be the corresponding $\h_{z\otimes 1}$-module.

 Suppose $m_i\ge 1$ for some $i\in\{2,\ldots,m\}$. Then it follows from Block's theorem on differentiably simple modules that there is a faithful
$S_i$-module $V_i'$ such that $V_i\cong V_{i}'\otimes \OO(m_i;\underline{1})$ as vector spaces and $\rho_i(\h_{z\otimes 1})$ embeds into the Lie subalgebra $$(\gl(V_{i}')\otimes \OO(m_i;\underline{1}))
\rtimes({\rm Id}_{\gl(V_{i}')}\otimes W(m_i;\underline{1}))$$ of $\gl(V_i)$; see \cite[Theorem~3.3.10 and Corllary~3.3.11]{Str04}. Since the
$S_i$-module $V_{i}'$ is faithful, it is of dimension at least $2$ and so it must be that $\dim V_i= p^{m_i}\dim V_{i}'\ge 2p^{m_i}$. But then $\ell+1\ge \dim V\ge 2p$ which is impossible as $p$ is a good prime for $G$. We thus deduce that
$m_i=0$ for all $i\ge 2$. It follows that $\h$ contains a restricted ideal $\h'$ isomorphic to $(S\otimes W(m;\underline{1}))\rtimes({\rm Id}_S\otimes \mathcal{D})$
for some transitive Lie subalgebra $\mathcal{D}$ of $W(m;\underline{1})$ and such that $\h=\h'\oplus \c_\h(\h')$.
 Moreover, any minimal ideal of $\c_\h(\h')=\c_\h(I)$ (if any) is isomorphic to the Lie algebra of a simple algebraic $k$-group. This proves (iv).

In order to finish the proof of the lemma it remains to show that $m=1$. Suppose $m>1$. If $m\ge 3$ then $\dim I\ge 3p^3>\dim \g$ which is false. Hence $m=2$ and $\dim I\ge 3p^2>(\dim \g)/2$. It follows that the restriction of the Killing form $\kappa$ to $I$ is nonzero. Write $\bar{\kappa}$ for the restriction of $\kappa$ to $\h'$. Since
$I$ is a minimal ideal of $\h'$ and $\bar{\kappa}(I,I)\ne 0$ it must be that $I\cap {\rm Rad}\,\bar{\kappa}=0$.

Let $x_1$, $x_2$ be generators of the local ring $\OO(2;\underline{1})$ contained in $\OO(2;\underline{1})_{(1)}$ and
put $x^\mathbf{p-1}:=x_1^{p-1}x_2^{p-1}$.
Pick any nonzero element $e\in S$. Then $e\otimes x^\mathbf{p-1}\in I$ and the above yields that $\kappa(e\otimes x^\mathbf{p-1},y)\ne 0$ for some $y\in\h'$. Let $V$ be any composition factor of the $(\ad \h')$-module $\g$ and write $\psi$ for the corresponding (restricted) representation of $\h'$.
If $\psi(I)=0$ then, of course,
$\psi(e\otimes x^\mathbf{p-1})\circ \psi(y)=0$. If $\psi(I)\ne 0$ then it follows from Block's theorem on differentiably simple modules that there is a $k$-vector space $V_0$ and a linear isomorphism
$V\cong V_0\otimes \OO(2;\underline{1})$ such that
$\psi(\h')$ embeds into $(\gl(V_0)\otimes \OO(2;\underline{1}))
\rtimes({\rm Id}_{\gl(V_0)}\otimes W(2;\underline{1}))$
and
$\psi\big(S\otimes \OO(2;\underline{1})_{(1)}\big)$ embeds into $\gl(V_0)\otimes \OO(2;\underline{1})_{(1)}$; see \cite[3.3]{Str04}.

Given a Lie algebra $L$ we write $L^d$ for the $d$th member of the lower central series of $L$.
As
$S$ is simple, $e\in S^{\,2p-2}$.  From this it is immediate that
$e\otimes x^\mathbf{p-1}\in
S\otimes \OO(2;\underline{1}\big)_{(2p-2)}$. But then
$\psi(e\otimes x^\mathbf{p-1})=E\otimes x^\mathbf{p-1}$ for some $E\in\gl(V_0)$. As $\psi(y)\in (\gl(V_0)\otimes \OO(2;\underline{1}))
\rtimes({\rm Id}_{\gl(V_0)}\otimes W(2;\underline{1}))$, this implies that $\psi(e\otimes x^\mathbf{p-1})\circ \psi(y)$ is a square-zero endomorphism. It follows that $\ad(e\otimes x^\mathbf{p-1})\circ \ad y$ acts nilpotently on any composition factor of the $(\ad \h')$-module $\g$. But then
$\kappa(e\otimes x^\mathbf{p-1},y)=\tr(\ad(e\otimes x^\mathbf{p-1})\circ \ad y)=0$ contrary to our choice of $y$. Therefore, $m=1$ and our proof is complete.
 \end{proof}

Lemma~\ref{socle}(i) shows that $z$ is a scalar multiple of a nonzero toral element of $S$.
Since $S$ is either $\sl_2$ or $W(1;\underline{1})$ by Lemma~\ref{socle}(ii) we may assume without loss that there is an $\sl_2$-triple $\{e,h,f\}\subset S$ such that $h=z$. Indeed, this is clear when $S=\sl_2$ and when $S=W(1;\underline{1})$ it follows from the fact that any toral element of $S$ is conjugate under the action of ${\rm Aut}(S)$
to either $(1+x)\partial$ or to a multiple of $x\partial$; see \cite{P92}, for example.

\subsection{Determining the conjugacy class of $S\otimes 1$}\label{3.5}
Recall from \S~\ref{3.2} that $S$ is either $\sl_2$ or $W(1;\underline{1})$ and
$S\otimes 1$ is a restricted subalgebra of $\g$. It follows that $(e\otimes 1)^{[p]}=(f\otimes 1)^{[p]}=0$ and $(h\otimes 1)^{[p]}=h\otimes 1$.
\begin{prop}\label{S}
Let $\h$ and $S$ be as in \S~\ref{3.2} and suppose further that $\h$ is indecomposable. Then the following hold:
\begin{itemize}
\item[(i)\,] $G$ is of type ${\rm E}_7$ or ${\rm E}_8$ and $p\in\{5,7\}$.
	
\smallskip	
	
\item[(ii)\,] If $p=5$ then $G$ is of type ${\rm E}_7$ and $e\otimes 1\in\OO({\rm A}_3{\rm A}_2{\rm A}_1)$.

\smallskip

\item[(iii)\,] If $p=7$ then $G$ is of type ${\rm E}_7$ or ${\rm E}_8$ and $e\otimes 1\in \OO({\rm A}_2{{\rm A}_1}^3)$.

\smallskip

\item[(iv)\,] The $\sl_2$-triple $\{e\otimes 1,h\otimes 1,f\otimes 1\}$ is standard and $S \cong \sl_2$.
\end{itemize}	
\end{prop}
\begin{proof} Since $p$ is a good prime for $G$, it follows from
Lemma~\ref{socle}(i) and our choice of $h\in S$ that $\dim \g_{h\otimes 1}=28$ if $G$ is of type
${\rm E}_6$;  $\dim \g_{h\otimes 1}=33$ if $G$ is of type ${\rm E}_7$ and $p=5$; $\dim \g_{h\otimes 1}=49$ if $G$ is of type ${\rm E}_7$ and $p=7$, and $\dim \g_{h\otimes 1}=52$ if $G$ is of type ${\rm E}_8$.
On the other hand, $\dim G_{e\otimes 1}=\dim \g_{e\otimes 1}\ge \dim \g_{h\otimes 1}$ by Lemma~\ref{sl2}. Looking through the tables in \cite[pp.~402--407]{Car} it is now straightforward to see that if $e\otimes 1\in \OO({\rm A}_{p-1}{\rm A}_r)$ for some $r\ge 0$, then $G$ is of type ${\rm E}_7$, $p=5$ and $r=0$. If $e\otimes 1$ is not of that type then the $\sl_2$-triple $\{e\otimes 1,h\otimes 1,f\otimes 1\}$ must be standard by Remark~\ref{nonstandard}.

Let $\Pi=\{\alpha_1,\ldots,\alpha_\ell\}$ be a basis of simple roots in the root system $\Phi=\Phi(G,T)$ and let $\tilde{\alpha}=
\sum_{i=1}^\ell n_i\alpha_i$ be the highest root of the positive system $\Phi^+(\Pi)$. In what follows we are going to use Bourbaki's numbering of simple roots in $\Pi$; see \cite[Planches~I--IX]{Bour}. Given a nilpotent $G$-orbit $\OO\subset\g$ we write $\Delta=(a_1,\ldots, a_\ell)$ for the weighted Dynkin diagram of $\OO$. We know that $a_i\in\{0,1,2\}$ and there is a nice representative $e'\in\OO$ which admits an optimal cocharacter $\tau\in X_*(T)$ such that
$e'\in\g(\tau,2)$ and $(\alpha_i\circ\tau)(\lambda)=\lambda^{a_i}$ for all $\lambda\in k^\times$ and $1\le i\le \ell$; see \cite[2.6]{P03}.  We put $\Delta'=(a_1'\ldots,a_\ell'):=\frac{1}{2}\Delta$ if all $a_i$ are even,
and $\Delta'=\Delta$ if $a_j$ is odd for some $j\le \ell$.


We first suppose that the $\sl_2$-triple $\{e\otimes 1,h\otimes 1,f\otimes 1\}$ is standard. If $e\otimes 1\in\OO(\Delta)$ and $\Delta'=(a'_1,\ldots, a'_\ell)$, then
looking through the tables in\cite[pp.~402--407]{Car} once again one observes that in the majority of cases specified at the beginning of the proof the inequality
\begin{equation}\label{high}
\textstyle{\sum}_{i=1}^\ell n_ia_i'<p\end{equation}
holds. If this happens then
$\sum_{i=1}^\ell m_ia_i'<p$ for all positive roots $\beta=\sum_{i=1}^\ell m_i\alpha_i\in\Phi^+(\Pi)$ and (\ref{h-tau}) yields that $\g(h_\tau,\bar{0})\cong\g(\tau,0)$ as Lie algebras. In type ${\rm E}_6$ the inequality (\ref{high}) holds in all cases of interest and there is no $\tau$ as above for which $\g(\tau,0)$ has type ${\rm A}_{p-1}{\rm A}_{\ell-p}$. In view of Remark~\ref{nonstandard} this rules out the case where $G$ is of type ${\rm E}_6$ thereby proving (i).

Suppose $G$ is of type ${\rm E}_7$ and $p=5$.
We first consider the case where $\{e\otimes 1,h\otimes 1,f\otimes 1\}$ is a standard $\sl_2$-triple. If (\ref{high}) holds for $\Delta'$ then
$\g_{h\otimes 1}\cong \g(\tau,0)$ must have type ${\rm A}_4{\rm A}_2$.
Since $\dim G_{e\otimes 1}\ge 33$, examining the Dynkin diagrams in
\cite[p.~403]{Car} one observes that $e\otimes 1\in \OO({\rm A}_3{\rm A}_2{\rm A}_1)$.
If (\ref{high}) does not hold for $\Delta'$ then \cite[p.~403]{Car} reveals that $\OO$ must have one of the following labels:
$$  {\rm A}_3,\, {{\rm A}_2}^2{\rm A}_1,\,  ({\rm A}_3{\rm A}_1)',\,
{\rm A}_3{{\rm A}_1}^2,\,{\rm D}_4,\,{\rm D}_4({\rm a}_1){\rm A}_1,\,{\rm A}_3{\rm A}_2.$$ Since a root system of type ${\rm A}_4{\rm A}_2$ does not contain subsystems of type ${\rm D}_4$, ${\rm A}_3{{\rm A}_1}^2$, ${\rm A}_2{{\rm A}_1}^3$ and ${\rm A}_5$, looking at the simple roots in $\Pi$
corresponding to those $i$  for which $a_i=0$ one can see that only the case where $e\otimes 1\in\OO\big({\rm D}_4({\rm a}_1){\rm A}_1\big)$ is possible. Since $p=5$, one can check directly that in this case the root system of $\g(h_\tau,\bar{0})$ with respect to $T$ has basis consisting of
$\alpha_1$, $\alpha_4$, $\alpha_5$, $\alpha_6$ and $
\beta={\scriptstyle{1\atop}\!{2\atop }\!{3\atop 2}\!{2\atop }\!
{1\atop }\!{1\atop }}$.
As a consequence, it has type ${\rm A}_4{\rm A}_1$, a contradiction.

Now suppose that $\{e\otimes 1,h\otimes 1,f\otimes 1\}$ is a non-standard $\sl_2$-triple. Then $e\otimes 1\in  \OO({\rm A}_4)$ by our remarks at the beginning of the proof. By \cite[p.~104]{LT11}, the reductive subgroup $C_{e\otimes 1}$ of $G_e$ has form $T_1\cdot \D C_{e\otimes 1}$, where $\D C_{e\otimes 1}$ is of type ${\rm A}_2$ and $T_1$ is a $1$-dimensional central torus in $C_{e\otimes 1}$. It is straightforward to see that $\c':=\Lie(\D C_{e\otimes 1})$ is generated by the simple root vectors $e_{\pm \alpha_6}$ and $e_{\pm\alpha_7}$.
Also, $T_0=\varpi_5^\vee(k^\times)$ where, as usual, $\varpi_i^\vee\in X_*(T)$ stands for the fundamental coweight of $\alpha_i\in \Pi$.  It follows that the restriction of the Killing form of $\g$ to $\c'$ is non-degenerate. Since $[e\otimes 1,\g(\tau,-2)]$ is orthogonal to $\g_e(0)=\Lie(C_{e\otimes 1})$ with respect to $\kappa$ and
$(h\otimes 1)-h_\tau\in \g_{e\otimes 1}(0)\cap [e\otimes 1,\g(\tau,-2)]$ by our discussion in \S~\ref{3.4}, it must be that $(h\otimes 1)-h_\tau\in \Lie(\varpi_5^\vee(k^\times))$.

For $1\le i\le \ell$ set $t_i:=({\rm d}_e\varpi_i^\vee)(1)$ and write $s_i$ for the simple reflection in the Weyl group $W=N_G(T)/T$ corresponding to the simple root $\alpha_i\in\Pi$. The toral elements $t_1,t_2,\ldots, t_7$ span $\Lie(T)$. Since $h\otimes 1$ is a toral element of $\Lie(T)$ and $\Lie(\varpi_5^\vee(k^\times))=k t_5$ we have that $\frac{1}{2}(h\otimes 1)= \frac{1}{2}h_\tau+r t_5$ for some $r\in \F_5$. It is immediate from \cite[p.~104]{LT11} that $\frac{1}{2}h_\tau=t_1+t_2+t_3+t_4-3t_5$. Since $p=5$, direct computations shows that $s_2s_4s_3s_1(\textstyle{\frac{1}{2}}h_\tau)\,=\, \textstyle{\frac{1}{2}}h_\tau+3t_5.$ As $s_2s_4s_3s_1$ fixes $t_5$, this yields that all elements in the set
$\frac{1}{2}h_\tau+\F_5\cdot t_5$ are conjugate under the action of $W$.
But then $h\otimes 1$ and $h_\tau$ have isomorphic centralisers. Since it is immediate from \cite[p.~403]{Car} that the centraliser of $h_\tau$ has type ${\rm D}_4{\rm A}_1$, this contradicts Lemma~\ref{socle}(i). We thus conclude that when $p=5$ the $\sl_2$-triple $\{e\otimes 1,h\otimes 1,f\otimes 1\}$  is standard and statement
(ii) holds.

Next suppose that $G$ is of type ${\rm E}_7$ and $p=7$. Since in this case $\dim \g_{e\otimes 1} \ge 49$ by Lemma~\ref{sl2}, analysing the weighted Dynkin diagrams in \cite[p.~403]{Car} shows that $e\otimes 1\not\in\OO({\rm A}_6)$ and the inequality (\ref{high}) holds for $\Delta'(e\otimes 1)$. Moreover, there is a unique weighted Dynkin diagram $\Delta'$ for which the centraliser of $h_\tau$ has type ${\rm A}_6$ and the corresponding nilpotent orbit has Dynkin label ${\rm A}_2{{\rm A}_1}^3$. This proves (iii) and shows that the $\sl_2$-triple $\{e\otimes 1,h\otimes 1,f\otimes 1\}$  is standard in type ${\rm E}_7$.

Now suppose $G$ is of type ${\rm E}_8$. In this case, $p=7$ and $\g_{h\otimes 1}$ has type ${\rm A}_6{\rm A}_1$. Since $\dim \g_{e\otimes 1}\ge 52$ by Lemma~\ref{sl2},
it follows from \cite[p.~406]{Car} that the orbit $\OO(e\otimes 1)$ is not of type ${\rm A}_6{\rm A}_r$ for $r\ge 0$. So the $\sl_2$-triple
$\{e\otimes 1,h\otimes 1,f\otimes 1\}$  must be standard.
Analysing the weighted Dynkin diagrams in \cite[pp.~405, 406]{Car} shows that either (\ref{high}) holds for $\Delta'(e\otimes 1)$ or the orbit $\OO(e\otimes 1)$ has one of the following labels:
$$  {\rm A}_4{\rm A}_1,\,
{\rm D}_5({\rm a}_1),\,
{\rm A}_5,\,{{\rm A}_3}^2,\,
{\rm D}_5({\rm a}_1){\rm A}_1,\,
{\rm A}_4{\rm A}_2{\rm A}_1,\,{\rm D}_4{\rm A}_1,\,
{\rm A}_4{{\rm A}_1}^2.$$
The first three labels cannot occur since in each of them the root system of $\g(\tau,0)$ contains a subsystem of type ${\rm D}_4$. The fourth label cannot occur either since a root system of type ${\rm A}_6{\rm A}_1$ does not contain subsystems of type ${{\rm A}_3}^2$. The label ${\rm D}_5({\rm a}_1){\rm A}_1$ cannot occur because $e_{\scriptstyle{1\atop }\!{3\atop }\!{5\atop 2}\!{4\atop }\!{3\atop }\!{2\atop }\!{1\atop }}$ commutes with $h_\tau$ and hence
the root system of $\g(h_\tau,\bar{0})$ contains a subsystem of type ${\rm A}_4{\rm A}_3$, forcing $\g(h_\tau,\bar{0})\not\cong\g_{h\otimes 1}$.
The label ${\rm A}_4{\rm A}_2{\rm A}_1$ cannot occur since $e_{\scriptstyle{2\atop }\!{4\atop }\!{5\atop 2}\!{4\atop }\!{3\atop }\!{2\atop }\!{1\atop }}$ commutes with $h_\tau$. This implies that the root system of $\g(h_\tau,\bar{0})$ contains a subsystem of type ${\rm A}_4{\rm A}_2{\rm A}_1$ which is not the case for the root system of $\g_{h\otimes 1}$. The label ${\rm A}_4{{\rm A}_1}^2$ cannot occur for the same reason: $e_{\scriptstyle{2\atop }\!{4\atop }\!{6\atop 3}\!{4\atop }\!{3\atop }\!{2\atop }\!{1\atop }}$ commutes with $h_\tau$ and hence the root system of $\g(h_\tau,\bar{0})$ contains a subsystem of type ${\rm A}_4{\rm A}_2{\rm A}_1$.

Finally, suppose $e\otimes 1\in\OO({\rm D}_4{\rm A}_1)$. By
\cite[p.~405]{Car},
$h_\tau$ is conjugate to $h'_\tau:=t_2+t_7+2t_8$ under the action of the Weyl group $W=N_G(T)/T$.
Using \cite[Planche~VII]{Bour} (and the fact that $p=7$) it is straightforward to see that the roots
$\alpha_1,\alpha_3,\alpha_4,\alpha_5,\alpha_6$ and ${\scriptstyle{1\atop }\!{3\atop }\!{5\atop 3}\!{4\atop }\!{3\atop }\!{2\atop }\!{1\atop }}$
form a basis of simple roots of the root system of $\g(h'_\tau,\bar{0})$ with respect to $T$. So the present case cannot occur and hence (\ref{high}) holds for $\Delta'(e\otimes 1)$. Looking once again at the weighted Dynkin diagrams
in \cite[p.~405]{Car} for which (\ref{high}) holds one finds out that
$\OO({\rm A}_2{{\rm A}_1}^3)$ is the only
nilpotent orbit in $\g$ for which $\g(\tau,0)=\g(h_\tau,\bar{0})$ has type ${\rm A}_6{\rm A}_1$. This proves (iii).

It remains to show that $S\cong\sl_2$. So suppose the contrary. Then $S$ may be identified with $\Der(\OO(1;\underline{1}))$ by Lemma~\ref{socle}(ii). The $k$-algebra $\OO(1;\underline{1})$ is generated by an element $x$ with $x^p=0$ and $S$ is spanned by the derivations $x^i\partial$ with $0\le i\le p-1$.
We may choose $\{e,h,f\}\subset S$ so that
$f=\partial$, $h=2x\partial$ and $e=x^2\partial$. Let $u=x^{p-1}\partial$. Then $[e,u]=0$ and $[h,u]=2(p-2)u=-\bar{4}u$. Since we have already proved that the $\sl_2$-triple $\{e\otimes 1,\h\otimes 1,f\otimes 1\}\subset \g$ is standard we may assume further that $h\otimes 1=h_\tau$.
Since $e\otimes 1\in \OO({\rm A}_3{\rm A}_2{\rm A}_1)$ if $p=5$ and $e\otimes 1\in \OO({\rm A}_2{{\rm A}_1}^3)$ if $p=7$, it follows from \cite[pp.~97, 105]{LT11} that
$$\g_{e\otimes 1}\cap \g(h_\tau, -\bar{4})\,=\,\g_{e\otimes 1}(p-4)
\,=\,\{0\}$$ if $G$ is of type ${\rm E}_7$.
Since
$u\otimes 1\in \g_{e\otimes 1}(-\bar{4})$, we conclude that $G$ is a group of type ${\rm E}_8$.
In this case we may assume that
$e\otimes 1=\,\sum_{\,i\ne 4,6,8}\,e_{\alpha_i}$.
By \cite[p.~127]{LT11}, $u\otimes 1\in \g_{e\otimes 1}(3)$ and the $C_{e\otimes 1}$-module $\g_{e\otimes 1}(3)$ is generated by
$e_{\scriptstyle{2\atop }\!{4\atop }\!{5\atop 3}\!{4\atop }\!{3\atop }\!{2\atop }\!{1\atop }}$ as in the characteristic zero case. From this it follows that $(\ad\,(f\otimes 1))^4(x)=0$ for all $x\in\g_{e\otimes 1}(3)$.
Since this contradicts the fact that $(\ad\,\partial)^4(x^6\partial)\ne 0$, we now deduce that the case  where $S\cong W(1;\underline{1})$ is impossible. Then $S\cong \sl_2$ by Lemma~\ref{socle}(ii) and our proof is complete.
\end{proof}
\begin{cor}\label{nonmax} If $G$ is of type ${\rm E}_8$ and $\h$ is as in Proposition~\ref{S}, then there exists an involution $\sigma\in G$ such that $G^\sigma$ is of type ${\rm E}_7{\rm A}_1$ and $\h\subset\g^\sigma$. In particular, $\h$ is not maximal in $\g$.
\end{cor}
\begin{proof} Since $G$ is of type ${\rm E}_8$, we have $p=7$.
By Proposition~\ref{S}, the $\sl_2$-triple $\{e\otimes 1,h\otimes 1,f\otimes 1\}$ is standard and $e\otimes 1\in\OO({\rm A}_1{{\rm A}_1}^3)$. So we may assume that $e\otimes 1=\,\sum_{\,i\ne 4,6,8}\,e_{\alpha_i}$ and
$h\otimes 1=h_\tau$ where the optimal cocharacter $\tau\in X_*(T)$ is as in \cite[p.~127]{LT11}. Let $\sigma=\tau(-1)$. Using {\it loc.\,cit.} it is straightforward to see that a root element $e_\beta\in \g^\sigma$ if and only if
$\beta=\sum_{\,i=1}^8\,m_i\alpha_i$ and $m_8$ is even.
From this it is immediate that $G^\sigma$ is a group of type ${\rm E}_7{\rm A}_1$. Clearly
$e\otimes \OO(1;\underline{1})\subseteq \g(h_\tau,\bar{2})$ and $f\otimes \OO(1;\underline{1})\subseteq \g(h_\tau,-\bar{2})$. As
 $\big[e\otimes \OO(1;\underline{1}),f\otimes \OO(1;\underline{1})\big]=h\otimes \OO(1;\underline{1})$ and
 $\g(h_\tau,\pm \bar{2})=\g(\tau,\pm 2)$ by \cite[p.~405]{Car}, we have that $S\otimes\OO(1;\underline{1})\subset \g^\sigma$. As
$$\h=(S\otimes \OO(1;\underline{1}))\rtimes ({\rm Id}\otimes\mathcal{D})\subseteq (S\otimes \OO(1;\underline{1}))+\c_{e\otimes 1}$$
and $\c_{e\otimes 1}=\Lie(C_{e\otimes 1})\subset \g(\tau,0)\subset\g^\sigma$ we obtain $\h\subset \g^\sigma$, as claimed.
\end{proof}

\subsection{The existence of $\h$}\label{3.6}
 The goal of this subsection is to give explicit examples of exotic semidirect products $\h\cong (\sl_2\otimes \OO(1;\underline{1}))\rtimes({\rm Id}\otimes \mathcal{D})$. In view of our discussion in \ref{3.5} we shall assume that
$G$ is of type ${\rm E}_7$ and $p\in\{5,7\}$. The notation introduced in the previous subsections will be used without further notice. We write $U^+$ and $U^-$ for the maximal unipotent subgroups of $G$ generated by the root subgroups $U_\gamma$ with $\gamma\in\Phi^+(\Pi)$ and $\gamma\in-\Phi^+(\Pi)$, respectively.
Combining Proposition~\ref{S} with \cite[p.~403]{Car}
one observes that $(\ad (e\otimes 1))^{p-1}=(\ad (f\otimes 1))^{p-1}=0$. In view of Remark~\ref{R0} this implies that $\g$ is a completely reducible
$\ad(S\otimes 1)$-module. As a consequence, $(\ad (f\otimes 1))^{3}$ annihilates
$\g_{\,e\otimes 1}(2)$.

We first suppose that $p=5$. In view of Proposition~\ref{S} we may assume that $e\otimes 1=\sum_{\,i\ne 4}\,e_{\alpha_i}$ and $h\otimes 1=h_\tau$  where $\tau\colon\,k^\times\to T$ is as in
\cite[p.~104]{LT11}.
The group $C_e$ is connected of type
${\rm A}_1$ and contains $T_0:=\varpi^\vee_4(k^\times)$ as a maximal torus.
The group $C_{e\otimes 1}^+:=C_{e\otimes 1}\cap U^+$ is normalised by $T_0$ and $B_{e\otimes 1}^+:=T_0\cdot C_{e\otimes 1}^+$ is a Borel subgroup of $C_{e\otimes 1}$.
In the notation of \cite{LT11} the maximal unipotent subgroup $C_{e\otimes 1}\cap U^-$ of $C_{e\otimes 1}$ consists of all $x_{-\beta_1}(t)$ with $t\in k$. (Here $-\beta_1$ is not a root of $G$ with respect to $T$.)
By the theory of rational $\SL(2)$-modules,
there exist nilpotent endomorphisms $X^{(i)}$ of $\g$ such that
$$\big(\Ad\,x_{-\beta_1}(t)\big)(v)\,=\,\textstyle{\sum}_{i\ge 0}\,t^iX^{(i)}(v)\qquad\ \big(\forall\,v\in\g\big)$$ (these endomorphisms may be different from those in
Proposition~\ref{expo}).
Each $X^{(i)}(v)$ is a weight vector of weight $2i$ for the action of $T_0$ on
$\End(\g)$. Furthermore, there exists a generator $X$ of $\Lie(C_{e\otimes 1}\cap U^-)$ such that
$X^{(i)}=\frac{1}{i!}(\ad X)^i$ for $0\le i\le p-1$.

As mentioned in \cite[p.~105]{LT11}, in characteristic $5$ the $C_{e\otimes 1}$-submodule
generated by $e_{\tilde{\alpha}}\in\g_{\,e\otimes 1}(2)$ is contained in
the $10$-dimensional indecomposable summand of $\g_{e\otimes 1}(2)$ isomorphic to
to the tilting module $T_{{\rm A}_1}(8)$.
The other indecomposable summand of the $15$-dimensional $C_{e\otimes 1}$-module $\g_{\,e\otimes 1}(2)$ is isomorphic to the Steinberg module $V(4)$.
Since $T_{{\rm A}_1}(8)$ is a projective module over the restricted enveloping algebra of $\sl_2$ by \cite[\S~2]{Do},
the vector space $\g_{X}\cap \g_{\,e\otimes 1}(2)$ is $3$-dimensional with  basis consisting of $T_0$-weight vectors of weight $-8$, $-4$ and $0$. It follows that the fixed point space $(\g_X\cap\g_{\,e\otimes 1}(2))^{T_0}$ is spanned by $e\otimes 1$.
On the other hand, since the tilting module $T_{{\rm A}_1}(8)$ admits a Weyl filtration, the $C_{e\otimes 1}$-submodule of $\g_{\,e\otimes 1}(2)$  generated by the highest weight vector $e_{\tilde{\alpha}}$ of weight $8$ for $B_{e\otimes 1}^+$ is isomorphic to the Weyl module $V(8)$. Since $X^{[5]}=0$ we have that $(\ad\,X)^4(e_{\tilde{\alpha}})\in (\g_X\cap\g_{\,e\otimes 1}(2))^{T_0}$, implying $k(e\otimes 1)\,=\, k(\ad\,X)^4(e_{\tilde{\alpha}})$.

Let $A$ denote the $k$-span of all $(\ad\,X)^i(e_{\tilde{\alpha}})$ with
$0\le i\le 4$. By construction, this is a $B_{e\otimes 1}^+$-submodule of
$\g_{\,e\otimes 1}(2)$ stable under the action of $\ad X$. As $\c_{e\otimes 1}=kX\oplus \Lie(B_{e\otimes 1}^+)$, it follows that
$A$ is a $\c_{e\otimes 1}$-submodule of $\g$. Since it is immediate from \cite{LT11} that $A\subset \Lie(U^+)$ we have that $[e_{\tilde{\alpha}}, A]=0$.
By the Jacobi identity, $$\big[(\ad\,X)^i(e_{\tilde{\alpha}}),(\ad\,X)^j(e_{\tilde{\alpha}})\big]\,\subseteq\,
\textstyle{\sum}_{k\le i}\,(\ad\,X)^k\big([e_{\tilde{\alpha}}, A]\big)=0.$$
Therefore, $A$ is an abelian subalgebra of $\g$. Since
$(\ad(f\otimes 1))^4=0$ and
$(\ad(f\otimes 1))^3(A)=0$, the Jacobi identity also yields that
$(\ad (f\otimes 1))^2(A)$ is a $5$-dimensional commutative subalgebra of $\g$ contained in $\g_{\,f\otimes 1}(-2)$. Since $f\otimes 1\in \sum_{\,i\ne 4}\,ke_{-\alpha_i}$ it must be that $[f\otimes 1,e_{\tilde{\alpha}}]\in
k^\times
e_{\scriptstyle{1\atop}\!{3\atop }\!{4\atop 2}\!{3\atop }\!
{2\atop }\!{1\atop }}$. Arguing as before it is now straightforward to see that
$[f\otimes 1,A]$ is a $5$-dimensional abelian subalgebra of
$\g(\tau,0)$ contained in $k(h\otimes 1)\oplus\, \Lie(U^+)$. Therefore,
$[[f\otimes 1,A],e_{\tilde{\alpha}}]=[k(h\otimes 1),e_{\tilde{\alpha}}]=k e_{\tilde{\alpha}}$. Applying $(\ad X)^i$ with $i\ge 0$ to both sides of this equality  one obtains that $[[f\otimes 1,A],A]=A$.
Since $[f\otimes 1,A]$ is abelian, this yields
$$\big[(\ad (f\otimes 1))^2(A),A\big]=[f\otimes 1,A]\ \, \mbox { and }\ \,  \big[(\ad (f\otimes 1))^2(A),[f\otimes 1,A]\big]=(\ad(f\otimes 1))^2(A).$$
As a consequence, $\tilde{A}:=A\oplus [f\otimes 1,A]\oplus (\ad (f\otimes 1))^2(A)$ is a $15$-dimensional Lie subalgebra of $\g$ containing $S\otimes 1$ and normalised by $\c_{e\otimes 1}$. Set $\h=\c_{e\otimes 1}\oplus \tilde{A}$ and
\begin{equation}
\label{nilA}
\tilde{A}^+:=\Lie(U^+)\cap\Big((\ad (f\otimes 1))^2(A)\oplus [f\otimes 1,A]\Big)\oplus\big([\Lie(U^+),\Lie(U^+)]\cap A\big).\end{equation} Then $\tilde{A}^+\subset \N(\g)$ and $\tilde{A}=(S\otimes 1)\oplus \tilde{A}^+$. Since $[\tilde{A},\tilde{A}^+]\subseteq \tilde{A}^+$, it follows that $\tilde{A}^+={\rm nil}(\tilde{A})$ and $\tilde{A}/{\rm nil}(\tilde{A})\cong \sl_2$. On the other hand, if $I$ is a nonzero ideal of $\tilde{A}$ stable under the action of $\ad X$ then $I\cap(S\otimes 1)\ne 0$. But then $f\otimes 1\in I$ implying that $[f\otimes 1, A]\subseteq I$ and $(\ad (f\otimes 1))^2(A)\subseteq I$. Hence $A=[[f\otimes 1,A],e_{\tilde{\alpha}}]\subseteq I$ forcing $I=\tilde{A}$. This shows that the Lie algebra $\tilde{A}$ is derivation simple.
As $\c_{e\otimes 1}\cong \sl_2$ is simple, it must act faithfully on $\tilde{A}$.
In conjunction with Block's theorem \cite[Corollary~3.3.5]{Str04} this entails
that $\h\cong(\sl_2\otimes\OO(1;\underline{1}))\rtimes ({\rm Id}\otimes \mathcal{D})$ where $\mathcal{D}\cong\sl_2$.

Now suppose $p=7$. By Propositon~\ref{S}, we may assume that $e\otimes 1=\sum_{\,i\ne 4,6}\,e_{\alpha_i}$. This case is very similar to the previous one, the only complication being that $C_{e\otimes 1}$ is now a connected group of type ${\rm G}_2$. In the notation of \cite[p.~97]{LT11} it is generated by the unipotent one-parameter subgroups $x_{\pm \beta_i}(k)$ where $i=1,2$. Here
$\beta_1$ and $\beta_2$ are simple roots of $C_{e\otimes 1}$ with respect to the maximal torus $T_0$ of $C_{e\otimes 1}$ generated by $\varpi_i^\vee(k^\times)$ with $i\in\{4,6\}$. Note that $\beta_1$ is a short root of $\Phi(C_{e\otimes 1},T_0)$.

For $i=1,2$ we pick a nonzero element $F_i$ in the Lie algebra of $x_{-\beta_i}(k)$. Then $F_1+F_2$ is a regular nilpotent element of $\c_{e\otimes 1}$
Let $\tilde{\beta}$ be the highest root of $\Phi(C_{e\otimes 1},T_0)$ with respect
to its basis of simple roots $\{\beta_1,\beta_2\}$ and let $\tilde{E}$ be a root vector of $\c_{e\otimes 1}$ corresponding to $\tilde{\beta}=3\beta_1+2\beta_2$.
As $p=7$, the Lie subalgebra $\mathfrak{w}$ of $\c_{e\otimes 1}$ generated by $F_1+F_2$ and $\tilde{E}$ is isomorphic to the Witt algebra $W(1;\underline{1})$; see \cite[Lemma~13]{P85} or \cite[Lemma~3.6]{HSMax}.
Moreover, one can choose an isomorphism $\psi\colon,\mathfrak{w}
\stackrel{\sim}{\longrightarrow}\,W(1;\underline{1})$ such that
$\psi(F_1+F_2)=\partial$ and $\psi(\tilde{E})\in k(x^6\partial).$
For $i\ge -1$, we denote by $\mathfrak{w}_{(i)}$ the preimage under $\psi$ of the $i$th member of the canonical filtration of $W(1;\underline{1})$.
By construction, $\mathfrak{w}_{(1)}\subset \Lie(U_+)$.

It follows from \cite[p.~97]{LT11} that $F_1+F_2$ is a weight vector of weight $-2$ with respect to the $1$-dimensional subtorus $T_1:=\,\big\{\varpi_4^\vee(c^2)\varpi_6^\vee(c^2)\,|\,\,c\in k^\times\big\}$ of $T_0$. The torus $T_1$ normalises $\mathfrak{w}$ and
$\Lie(T_1)=k \psi^{-1}( x\del)$. In the present case the Levi subgroup $C_G(\tau)$ acts rationally on a $7$-dimensional
vector space $V$ over $k$ and $\Lie(C_G(\tau))=\g(\tau,0)\cong \gl(V)$ as restricted Lie algebras. Then $\c_{e\otimes 1}\hookrightarrow \gl(V)$ and $V$ may be regarded as a faithful restricted $7$-dimensional $\c_{e\otimes 1}$-module. Since $C_{e\otimes 1}$ is a group of type ${\rm G}_2$, it must be that $V\cong L(\varpi_1)$ as $C_{e\otimes 1}$-modules (this is immediate from \cite{P88}, for example). Hence $F_1+F_2$ acts
on $V$ as a nilpotent Jordan block of size $7$; see \cite[p.~95]{P85} for more detail. We thus deduce that $F_1+F_2$ is a regular nilpotent element of $\g(\tau,0)$.

Since $\g$ is a completely reducible $\ad(S\otimes 1)$-module and $\g_{\,e\otimes1}(r)=0$ for $r\not\in\{0,\pm2,\pm4\}$ by \cite[p.~97]{LT11},
we have that \begin{equation}\label{dir}
\g(\tau,0)\,=\,\c_{e\otimes 1}\oplus[f\otimes 1,\g_{e\otimes 1}(2)]\oplus
[f\otimes 1,[f\otimes 1,\g_{e\otimes 1}(4)]].\end{equation}
Hence $[f\otimes 1,\g_{\,e\otimes 1}(2)]\cong \g_{\,e\otimes 1}(2)$ is a direct summand of the $(\Ad\, C_{e\otimes 1})$-module $\g(\tau,0)\cong V\otimes V^*$.
By \cite[p.~97]{LT11},  $\dim \g_{\,e\otimes1}(2)=28$ and $e_{\tilde{\alpha}}\in \g_{\,e\otimes1}(2)$ is a highest weight vector of weight $2\varpi_1$ for the Borel subgroup $B_{e\otimes 1}^+=T_0\cdot
(C_{e\otimes 1}\cap U^+)$ of $C_{e\otimes 1}$.
 The Weyl module $V(2\varpi_1)$ for $C_{e\otimes 1}$ has dimension $27$ by Weyl's dimension formula. Since
 in characteristic $7$ the irreducible $C_{e\otimes 1}$-module $L(2\varpi_1)$ has dimension $26$ by \cite[6.49]{Lu}, we have that ${\rm Ext}^1_{C_{e\otimes 1}}(L(2\varpi_1),k)\ne 0$, using \cite[II.2.14]{Jan03}.
 Since $\g(\tau,0)\cong V\otimes V^*$ is a tilting module for $C_{e\otimes 1}$, it follows from (\ref{dir}) that the $(\Ad\,C_{e\otimes 1})$-module
 $\g_{\,e\otimes 1}(2)$ is isomorphic to the tilting module $T_{{\rm G}_2}(2\varpi_1)$ (one should keep in mind here that $T_{{\rm G}_2}(2\varpi_1)$ admits both a good filtration and a Weyl filtration). We mention in passing that  $T_{{\rm G}_2}(2\varpi_1)\cong S^2(L(\varpi_1))$  (this will not be required in what follows).

 Since $V$ is a free module over the restricted enveloping algebra of $k(F_1+F_2)$, so are $V\otimes V^*$ and $[f\otimes 1,\g_{\,e\otimes1}(2)]$. This implies that
 $$\g_{\,e\otimes1}(2)\cap\g_{F_1+F_2}\,=\,(\ad(F_1+F_2))^6(\g_{\,e\otimes1}(2)).$$ Note that $e_{\tilde{\alpha}}$ has weight $12$ with respect to the action of $T_1$. Conversely, any $T_1$-weight vector of weight $12$ in $\g_{\,e\otimes1}(2)$ is proportional to
  $e_{\tilde{\alpha}}$ (here we use the fact that $\g_{\,e\otimes1}(2)\cong T_{{\rm G}_2}(2\varpi_1)$ as $C_{e\otimes 1}$-modules).
 This gives $e\otimes 1\in k(\ad(F_1+F_2))^6(e_{\tilde{\alpha}})$.

 We now let $A$ be the $k$-span of all $\ad(F_1+F_2)^i(e_{\tilde{\alpha}})$ with $0\le i\le 6$. Repeating  verbatim the argument used in the previous case we observe that $\tilde{A}:=A\oplus  [f\otimes 1, A]\oplus (\ad (f\otimes 1))^2(A)$ is a Lie subalgebra of $\g$ normalised by $\mathfrak{w}$. We then define $\tilde{A}^+$ as in (\ref{nilA}) and check directly that $\tilde{A}^+={\rm nil}(\tilde{A})$ and $\tilde{A}/{\rm nil}{\tilde{A}}\cong \sl_2$.

We set $\h=\mathfrak{w}\oplus\tilde{A}$.
If $I$ is a nonzero ideal of $\tilde{A}$ stable under the action of $\ad (F_1+F_2)$ then our earlier remarks imply that $I\cap(S\otimes 1)\ne 0$. But then $f\otimes 1\in I$ forcing $[f\otimes 1, A]\subseteq I$ and $(\ad (f\otimes 1))^2(A)\subseteq I$. Hence  $I=\tilde{A}$ as in the previous case. This shows that the Lie algebra $\tilde{A}$ is derivation simple.
As the Lie algebra $\mathfrak{w}\cong W(1;\underline{1})$ is simple, it acts faithfully on $\tilde{A}$.
Applying Block's theorem \cite[Corollary~3.3.5]{Str04} we get
$\h\cong(\sl_2\otimes\OO(1;\underline{1}))\rtimes ({\rm Id}\otimes \mathcal{D})$ where $\mathcal{D}\cong W(1;\underline{1})$.

\subsection{The uniqueness of ${\rm Soc}(\h)$}\label{3.7}



\def\arraystretch{0.5} \arraycolsep=0pt
In this subsection we assume that
$\h'= (S'\otimes\OO(1;\underline{1}))\rtimes(\rm Id\otimes \mathcal{D'})$ is an esdp of $\g$ where $\g$ is of type ${\rm E}_7$ and $p\in\{5,7\}$. Our goal is to show that ${\rm Soc}(\h')=S'\otimes\OO(1;\underline{1})$ is
conjugate to the socle of the subalgebra $\h$ described in \S~\ref{3.6}. This will imply that $\n_\g({\rm Soc}(\h'))=(\Ad\, g)(\h)$ for some $g\in G$. By Proposition~\ref{S}, we may assume that $S'\otimes 1=S\otimes 1$ is spanned by
the $\sl_2$-triple $\{e\otimes 1,h\otimes 1, f\otimes 1\}$ described in \cite[p.~104]{LT11} for $p=5$ and in \cite[p.~97]{LT11} for $p=7$.

To ease notation we identify ${\rm Id}\otimes\mathcal{D'}$ with $\mathcal D'$, a  Lie subalgebra of $W(1;\underline{1})$. Since $\h'$ is semisimple,
$\mathcal{D}'$ is not contained in the standard maximal subalgebra $W(1;\underline{1})_{(0)}$ of $W(1;\underline{1})$ (otherwise
$S\otimes\OO(1;\underline{1})$ would contain a nonzero nilpotent ideal of $\h'$).
In view of Lemma~\ref{tran}, we may assume further that either $D=\partial$ or $D=(1+x)\partial$. Unfortunately, this is the best one can say as it may well be that $\dim \mathcal{D}'=1$.

Let $\c_{e\otimes 1}=\Lie(C_{e\otimes 1})$. Since $D$ commutes with both $e\otimes 1$ and $h\otimes 1$ and $\g(h_\tau,\bar{0})=\g(\tau,0)$ in all cases of interest, we have that
$D\in \c_{e\otimes 1}$. Our first task will be to determine
 the conjugacy class of the subspace $k D$ under the adjoint action of $C_{e\otimes 1}$. For that we are going to use some properties of irreducible representations of the completely solvable Lie algebra $kD\ltimes(h\otimes\OO(1;\underline{1}))$.

\begin{lemma}\label{pindownd} Under the above assumptions on $\h'$, a scalar multiple of $D$ is contained in a standard regular $\sl_2$-triple of $\c_{e\otimes 1}$. In particular, the subspace $kD$ is unique up to $(\Ad\, C_{e\otimes 1})$-conjugacy.\end{lemma}
\begin{proof}
If $p=5$ then $\c_{e\otimes 1}\cong\sl_2$ by \cite[p.~104]{LT11}. In this case, the statement of the lemma is elementary.
So we assume from now that $p=7$.

Let   $\r:=kD\ltimes(h\otimes\OO(1\;\underline{1}))$. Since $h\otimes 1=h_\tau$ and
$[h\otimes 1,\r]=0$ we have that $\r\subset \g(h_\tau,\bar{0})$.
We have already mentioned in \S~\ref{3.6} that the restricted Lie algebra $\g(h_\tau,\bar{0})=\g(\tau,0)$ is isomorphic to $\gl(V)$ where $V$ is a $7$-dimensional vector space over $k$. Since $\h'$ is a restricted Lie subalgebra of $\g$, so is $\h'\cap \g(\tau,0)=
(h\otimes\OO(1;\underline{1}))\rtimes\mathcal{D}'$. It follows that
$\r$ is a restricted Lie subalgebra of $\g(\tau,0)$ and
$\r':= h\otimes\OO(1\;\underline{1})$ is a restricted abelian ideal of  $\r$.
Since $D\in\{\partial, (1+x)\partial\}$, there is an $x\in\r'$ such that $[D,x]^p=h\otimes 1$.
As $k(h\otimes 1)$ acts on $V$ by scalar multiplications, we see that the derived subalgebra of $\r$ does not act nilpotently on $V$.
Therefore, the $\r$-module $V$ has a composition factor of dimension $>1$.
Since $\r$ is completely solvable and $\dim V=p$ it follows from \cite[Corollary~3.3.9]{Str04}, for example, that $V$ is an irreducible restricted $\r$-module. As $\r'$ is abelian, $V$ contains a $1$-dimensional $\r'$-submodule, $V_0$ The irreducibility of the $\r$-module $V$ then implies that $V$ is a homomorphic image of the induced $\r$-module
$\widetilde{V}_0:=u(\r)\otimes_{u(\r')}V_0$. Since $\r'$ has codimension $1$ in $\r$ we have that $\dim \widetilde{V}_0=p\dim V_0=\dim V$. It follows that $V\cong \widetilde{V}_0$.
As $D\not\in\r'$, this implies that $V$ is a free $u(kD)$-module of rank $1$.
In other words, if $D^p=0$ then $D$ acts on $V$ as a nilpotent Jordan block of size $p$ and
if $D^p=D$ then the set of the eigenvalues of $D$ on $V$ equals $\F_p$ and all eigenvalues appear with multiplicity $1$.

Recall from \S~\ref{3.6} that $C_{e\otimes 1}$ has type ${\rm G}_2$ and
$V\cong L(\varpi_1)$ as $C_{e\otimes 1}$-modules. If $D^p=0$ then the above says that $D^6$ does not vanish on $V$. Applying \cite[Table~4]{Ste} we then deduce that $D$ lies in the regular nilpotent $C_{e\otimes 1}$-orbit of $\c_{e\otimes 1}$. In view of our discussion in \S~\ref{3.4} it clear now that the statement of the lemma holds when $D^p=0$.

If $D^p=D$ then $D$ is a toral element of $\c_{e\otimes 1}$ and
the commutative algebra $u(kD)$ is semisimple.
Furthermore, no generality will be lost by assuming that $D$ lies in the maximal toral subalgebra $\Lie(T_0)$ of $\c_{e\otimes 1}$. The set of all toral elements if $\Lie(T_0)$ is spanned over $\F_7$ by two elements, $t_1$ and $t_2$, such that
$\beta_i(t_j)=\delta_{ij}$ for $i,j\in\{1,2\}$ (the simple roots $\beta_1,\beta_2\in \Phi(C_{e\otimes1}, T_0)$ were introduced in \ref{3.6}).
Write $D=at_1+bt_2$ for some $a,b\in \F_7$.
As $\beta_1$ is a short root, we have that $\varpi_1=2\beta_1+\beta_2$.
It is well known that the set of $T_0$-weights of $L(\varpi_1)$ equals
$$\{\varpi_1,\,\varpi_1-\beta_1,\, \varpi_1-\beta_1-\beta_2,\,\varpi_1-2\beta_1-\beta_2,\,\varpi_1-
3\beta_1-\beta_2,\,
\varpi_1-3\beta_1-2\beta_2,\,\varpi_1-4\beta_1-2\beta_2\}.$$
Since each eigenvalue of $D$ on $V\cong L(\varpi_1)$ has multiplicity
$1$ we may assume after rescaling that $a=1$. Then the set of the eigenvalues of $D$ on $V$ equals
$\{2+b,1+b,1,0,-1,-1-b,-2-b\}.$ One can see by inspection that there are two values of $b\in\F_7$ for which this set consists of $7$ elements, namely
$b=1$ and $b=3$. As $\tilde{\beta}=3\beta_1+2\beta_2$ and
$({\rm d}_e\tilde{\beta}^\vee)(1)=t_2$ we have that $$s_{\tilde{\beta}}(t_1+t_2)=(t_1+t_2)-5t_2=t_1+3t_2.$$
This shows that $D$ is $(\Ad\,C_{e\otimes 1})$-congugate to $t_1+t_2$ completing the proof.
\end{proof}
We are now ready to prove that the socle $S\otimes\OO(1;\underline{1})$ of $\h'$ is conjugate under the adjoint action of the group $G$ to the socle of the esdp $\h$ constructed in \S~\ref{3.6}. This will imply that $\m:=\n_\g({\rm Soc}(\h'))$ is also unique up to $(\Ad\,G)$-conjugacy. Recall that ${\rm Soc}(\h)$ and ${\rm Soc}(\h')$ share the same $\sl_2$-triple $(e\otimes 1,h\otimes 1,f\otimes 1)$ described explicitly in \S~\ref{3.6}. In light of Lemma \ref{pindownd}, we also know that $D\in\h'$ is either a nilpotent or toral element in $\c_{e\otimes 1}=\Lie(C_{e\otimes 1})$
and its nonzero scalar multiple is contained in a regular
$\sl_2$-triple of  $\c_{e\otimes 1}$.

The group $C_{e\otimes 1}=G_{h\otimes 1}\cap G_{e\otimes 1}$ has type ${\rm A}_1$ (resp., ${\rm G}_2$) if $p=5$ (resp., $p=7$).
We shall treat each of four cases by turn, using GAP and imposing well-chosen relations between the elements of $kD$ and ${\rm Soc}(\h')$ to show that there is only one choice for ${\rm Soc}(\h')$ up to conjugacy by $G$.

Let $p=5$. Then $e\in\OO({\rm A_3}{\rm A}_2{\rm A}_1)$ and the following elements give an $\sl_2$-triple $\{e\otimes 1,h\otimes 1,f\otimes 1\}$ generating the subalgebra $S\otimes 1$ of ${\rm Soc}(\h')$:
\begin{align*}e\otimes 1\,&=\,e_{\Small\begin{array}{c c c c c c c c}1&0&0&0&0&0\\&&0\end{array}}+
e_{\Small\begin{array}{c c c c c c c c}0&0&0&0&0&0\\&&1\end{array}}+
e_{\Small\begin{array}{c c c c c c c c}0&1&0&0&0&0\\&&0\end{array}}+
e_{\Small\begin{array}{c c c c c c c c}0&0&0&1&0&0\\&&0\end{array}}+
e_{\Small\begin{array}{c c c c c c c c}0&0&0&0&1&0\\&&0\end{array}}+
e_{\Small\begin{array}{c c c c c c c c}0&0&0&0&0&1\\&&0\end{array}},\\
h\otimes 1\,&=\,2\cdot h_{\alpha_1}+
h_{\alpha_2}+
2\cdot h_{\alpha_3}+
3\cdot h_{\alpha_5}+
4\cdot h_{\alpha_6}+
3\cdot h_{\alpha_7},\\
f\otimes 1\,&=\,2\cdot e_{\Small\begin{array}{c c c c c c c c}-1&0&0&0&0&0\\&&0\end{array}}+
e_{\Small\begin{array}{c c c c c c c c}-0&0&0&0&0&0\\&&1\end{array}}+
2\cdot e_{\Small\begin{array}{c c c c c c c c}-0&1&0&0&0&0\\&&0\end{array}}+
3\cdot e_{\Small\begin{array}{c c c c c c c c}-0&0&0&1&0&0\\&&0\end{array}}+
4\cdot e_{\Small\begin{array}{c c c c c c c c}-0&0&0&0&1&0\\&&0\end{array}}+
3\cdot e_{\Small\begin{array}{c c c c c c c c}-0&0&0&0&0&1\\&&0\end{array}}.
\end{align*}

(i) Suppose $D^p=0$. Then $D$ is a regular nilpotent element of $\c_{e\otimes 1}$ (Lemma~\ref{pindownd}) and in view of \cite[p.~104]{LT11} we may assume that
\[D\,=\,3\cdot e_{\Small\begin{array}{c c c c c c c c}1&1&1&0&0&0\\&&1\end{array}}+
e_{\Small\begin{array}{c c c c c c c c}1&1&1&1&0&0\\&&0\end{array}}+
2\cdot e_{\Small\begin{array}{c c c c c c c c}0&1&1&1&0&0\\&&1\end{array}}+
e_{\Small\begin{array}{c c c c c c c c}0&0&1&1&1&0\\&&1\end{array}}+
e_{\Small\begin{array}{c c c c c c c c}0&1&1&1&1&0\\&&0\end{array}}+
e_{\Small\begin{array}{c c c c c c c c}0&0&1&1&1&1\\&&0\end{array}}.\]
Recall the maximal unipotent subgroup  $C_{e\otimes 1}^+$ of $C_{e\otimes 1}$ introduced in \S~\ref{3.6}. We will search for the element $e\otimes x\in {\rm Soc}(\h')$ and show up to the adjoint action of $C_{e\otimes 1}^+$ there is precisely one candidate with the choices we have already made. Since $[h\otimes 1,e\otimes x]=2\cdot e\otimes x$, and $\g_{e\otimes 1}\cap\g(h_\tau,\bar{2})\,=\,\g_{e\otimes 1}(2)$  it must be that $e\otimes x\in\g_{e\otimes 1}(2)$. A general element of $\g_{e\otimes 1}(2)$ is
\begin{align*}v\,=\,\,&x_3\cdot e_{\Small\begin{array}{c c c c c c c c}1&0&0&0&0&0\\&&0\end{array}}+
x_2\cdot e_{\Small\begin{array}{c c c c c c c c}0&0&0&0&0&0\\&&1\end{array}}+
x_3\cdot e_{\Small\begin{array}{c c c c c c c c}0&1&0&0&0&0\\&&0\end{array}}+
x_7\cdot e_{\Small\begin{array}{c c c c c c c c}0&0&0&1&0&0\\&&0\end{array}}+
x_7\cdot e_{\Small\begin{array}{c c c c c c c c}0&0&0&0&1&0\\&&0\end{array}}+
x_7\cdot e_{\Small\begin{array}{c c c c c c c c}0&0&0&0&0&1\\&&0\end{array}}\\
&+
(x_{30}+2\cdot x_{31})\cdot e_{\Small\begin{array}{c c c c c c c c}1&1&1&1&0&0\\&&1\end{array}}+
x_{31}\cdot e_{\Small\begin{array}{c c c c c c c c}1&1&1&1&1&0\\&&0\end{array}}+
(x_{30}+x_{31})\cdot e_{\Small\begin{array}{c c c c c c c c}0&1&1&1&1&0\\&&1\end{array}}+
x_{30}\cdot e_{\Small\begin{array}{c c c c c c c c}0&0&1&1&1&1\\&&1\end{array}}+
x_{31}\cdot e_{\Small\begin{array}{c c c c c c c c}0&1&1&1&1&1\\&&0\end{array}}\\
&+
(4\cdot x_{47}+x_{49})\cdot e_{\Small\begin{array}{c c c c c c c c}1&2&2&2&1&0\\&&1\end{array}}+
x_{47}\cdot e_{\Small\begin{array}{c c c c c c c c}1&2&2&1&1&1\\&&1\end{array}}+
x_{49}\cdot e_{\Small\begin{array}{c c c c c c c c}1&1&2&2&1&1\\&&1\end{array}}+
x_{49}\cdot e_{\Small\begin{array}{c c c c c c c c}0&1&2&2&2&1\\&&1\end{array}}\\
&+
4\cdot x_{59}\cdot e_{\Small\begin{array}{c c c c c c c c}1&2&3&2&2&1\\&&2\end{array}}+
x_{59}\cdot e_{\Small\begin{array}{c c c c c c c c}1&2&3&3&2&1\\&&1\end{array}}+
x_{63}\cdot e_{\Small\begin{array}{c c c c c c c c}2&3&4&3&2&1\\&&2\end{array}}\\
&+
(2\cdot x_{80}+4\cdot x_{81})\cdot e_{\Small\begin{array}{c c c c c c c c}-1&1&1&0&0&0\\&&0\end{array}}+
(4\cdot x_{80}+x_{81})\cdot e_{\Small\begin{array}{c c c c c c c c}-0&1&1&0&0&0\\&&1\end{array}}+
(4\cdot x_{80}+x_{81})\cdot e_{\Small\begin{array}{c c c c c c c c}-0&0&1&1&0&0\\&&1\end{array}}\\
&+
x_{80}\cdot e_{\Small\begin{array}{c c c c c c c c}-0&1&1&1&0&0\\&&0\end{array}}+
x_{81}\cdot e_{\Small\begin{array}{c c c c c c c c}-0&0&1&1&1&0\\&&0\end{array}}+
x_{103}+x_{104}\cdot e_{\Small\begin{array}{c c c c c c c c}-1&2&2&1&0&0\\&&1\end{array}}+
x_{103}+x_{104}\cdot e_{\Small\begin{array}{c c c c c c c c}-1&1&2&1&1&0\\&&1\end{array}}\\
&+
x_{103}\cdot e_{\Small\begin{array}{c c c c c c c c}-0&1&2&2&1&0\\&&1\end{array}}+
x_{104}\cdot e_{\Small\begin{array}{c c c c c c c c}-0&1&2&1&1&1\\&&1\end{array}}+
4\cdot x_{117}\cdot e_{\Small\begin{array}{c c c c c c c c}-1&2&3&2&1&0\\&&2\end{array}}+
x_{117}\cdot e_{\Small\begin{array}{c c c c c c c c}-1&2&3&2&1&1\\&&1\end{array}}+
x_{124}\cdot e_{\Small\begin{array}{c c c c c c c c}-1&2&4&3&2&1\\&&2\end{array}},\end{align*}
with each $x_i\in k$.
The relation $[D,e\otimes x]=e\otimes 1$ puts many linear constraints on the $x_i$. Imposing them on $v$ gives
\begin{align*}v\,=\,\,&x_7\cdot e_{\Small\begin{array}{c c c c c c c c}1&0&0&0&0&0\\&&0\end{array}}+
x_7\cdot e_{\Small\begin{array}{c c c c c c c c}0&0&0&0&0&0\\&&1\end{array}}+
x_7\cdot e_{\Small\begin{array}{c c c c c c c c}0&1&0&0&0&0\\&&0\end{array}}+
x_7\cdot e_{\Small\begin{array}{c c c c c c c c}0&0&0&1&0&0\\&&0\end{array}}+
x_7\cdot e_{\Small\begin{array}{c c c c c c c c}0&0&0&0&1&0\\&&0\end{array}}+
x_7\cdot e_{\Small\begin{array}{c c c c c c c c}0&0&0&0&0&1\\&&0\end{array}}\\
&+
3\cdot x_{49}\cdot e_{\Small\begin{array}{c c c c c c c c}1&2&2&2&1&0\\&&1\end{array}}+
3\cdot x_{49}\cdot e_{\Small\begin{array}{c c c c c c c c}1&2&2&1&1&1\\&&1\end{array}}+
x_{49}\cdot e_{\Small\begin{array}{c c c c c c c c}1&1&2&2&1&1\\&&1\end{array}}+
x_{49}\cdot e_{\Small\begin{array}{c c c c c c c c}0&1&2&2&2&1\\&&1\end{array}}+
x_{63}\cdot e_{\Small\begin{array}{c c c c c c c c}2&3&4&3&2&1\\&&2\end{array}}\\&+
3\cdot e_{\Small\begin{array}{c c c c c c c c}-1&1&1&0&0&0\\&&0\end{array}}+
4\cdot e_{\Small\begin{array}{c c c c c c c c}-0&1&1&0&0&0\\&&1\end{array}}+
4\cdot e_{\Small\begin{array}{c c c c c c c c}-0&0&1&1&0&0\\&&1\end{array}}+
2\cdot e_{\Small\begin{array}{c c c c c c c c}-0&1&1&1&0&0\\&&0\end{array}}+
e_{\Small\begin{array}{c c c c c c c c}-0&0&1&1&1&0\\&&0\end{array}},\end{align*}
Note that $h\otimes x=[e\otimes x,f\otimes 1]$ is an element of ${\rm Soc}(\h')$ whose $[p]$th power in $\g$ is zero. In particular, we must have
$(\ad [v,f\otimes 1])^5(e\otimes 1)=0$. Imposing this condition on $v$
results in the equation $x_{63}=3\cdot x_7^5$. Since $[[[h\otimes x],e\otimes x],e\otimes x]=0$, we should insist similarly that $[[[v,f\otimes 1],v],v]=0$.
This yields $2\cdot x_{49}^2=0$ forcing $x_{49}=0$
and implying that \begin{align*}v\,=\,\,&x_7\cdot e_{\Small\begin{array}{c c c c c c c c}1&0&0&0&0&0\\&&0\end{array}}+
x_7\cdot e_{\Small\begin{array}{c c c c c c c c}0&0&0&0&0&0\\&&1\end{array}}+
x_7\cdot e_{\Small\begin{array}{c c c c c c c c}0&1&0&0&0&0\\&&0\end{array}}+
x_7\cdot e_{\Small\begin{array}{c c c c c c c c}0&0&0&1&0&0\\&&0\end{array}}+
x_7\cdot e_{\Small\begin{array}{c c c c c c c c}0&0&0&0&1&0\\&&0\end{array}}+
x_7\cdot e_{\Small\begin{array}{c c c c c c c c}0&0&0&0&0&1\\&&0\end{array}}\\
&+3\cdot x_7^5\cdot e_{\Small\begin{array}{c c c c c c c c}2&3&4&3&2&1\\&&2\end{array}}+3\cdot e_{\Small\begin{array}{c c c c c c c c}-1&1&1&0&0&0\\&&0\end{array}}+
4\cdot e_{\Small\begin{array}{c c c c c c c c}-0&1&1&0&0&0\\&&1\end{array}}+
4\cdot e_{\Small\begin{array}{c c c c c c c c}-0&0&1&1&0&0\\&&1\end{array}}+
2\cdot e_{\Small\begin{array}{c c c c c c c c}-0&1&1&1&0&0\\&&0\end{array}}+
e_{\Small\begin{array}{c c c c c c c c}-0&0&1&1&1&0\\&&0\end{array}}.\end{align*}

 Since the subspace $\Lie(U^+)\cap\Lie( C_{e\otimes 1}^+)$ is $1$-dimensional, $kD$ is the Lie algebra of
$C_{e\otimes 1}^+$. Set
$$v_0:=\,
\,3\cdot e_{\Small\begin{array}{c c c c c c c c}-1&1&1&0&0&0\\&&0\end{array}}+
4\cdot e_{\Small\begin{array}{c c c c c c c c}-0&1&1&0&0&0\\&&1\end{array}}+
4\cdot e_{\Small\begin{array}{c c c c c c c c}-0&0&1&1&0&0\\&&1\end{array}}+
2\cdot e_{\Small\begin{array}{c c c c c c c c}-0&1&1&1&0&0\\&&0\end{array}}+
e_{\Small\begin{array}{c c c c c c c c}-0&0&1&1&1&0\\&&0\end{array}}.$$
Since $[D,v_0]$ is a nonzero scalar multiple of $e\otimes 1$ and
$v=v_0+x_7(e\otimes 1)+x_7^5\cdot e_{\Small\begin{array}{c c c c c c c c}2&3&4&3&2&1\\&&2\end{array}}$, there exists $u\in C_{e\otimes 1}^+$ such that
$(\Ad\, u)(v)=v_0$. Since $C_{e\otimes 1}^+$ fixes both $D$ and $S\otimes 1$ we may assume without loss of generality that $v=v_0$.
On the other hand, it is straightforward to see that the Lie algebra $kD\oplus {\rm Soc}(\h')$ is generated by $D$, $e\otimes x$ and $S\otimes 1$. As these are uniquely determined up to conjugacy, so is
${\rm Soc}(\h')=\big[kD\oplus {\rm Soc}(\h'),kD\oplus {\rm Soc}(\h')\big]$.
This settles the present case.

(ii) Now suppose $D\in C_{e\otimes 1}$ is toral, acting as $(1+x)\del$ on $\OO(1;\underline{1})$. In view of Lemma~\ref{pindownd} we may assume that
\[D\,=\,4\cdot h_{\alpha_1}+
h_{\alpha_2}+
3\cdot h_{\alpha_3}+
2\cdot h_{\alpha_4}+
4\cdot h_{\alpha_5}+
h_{\alpha_6}+
3\cdot h_{\alpha_7}.\]
As before, we represent $e\otimes (1+x)$ by a general element $v$ of $\g_{e\otimes 1}(2)$. Imposing the condition $[D,e\otimes(1+x)]=e\otimes(1+x)$ on $v$ shows that $v$ equals \[(x_{30}+2\cdot x_{31})\cdot e_{\Small\begin{array}{c c c c c c c c}1&1&1&1&0&0\\&&1\end{array}}+
x_{31}\cdot e_{\Small\begin{array}{c c c c c c c c}1&1&1&1&1&0\\&&0\end{array}}+
(x_{30}+x_{31})\cdot e_{\Small\begin{array}{c c c c c c c c}0&1&1&1&1&0\\&&1\end{array}}+
x_{30}\cdot e_{\Small\begin{array}{c c c c c c c c}0&0&1&1&1&1\\&&1\end{array}}+
x_{31}\cdot e_{\Small\begin{array}{c c c c c c c c}0&1&1&1&1&1\\&&0\end{array}}+
x_{124}\cdot e_{\Small\begin{array}{c c c c c c c c}-1&2&4&3&2&1\\&&2\end{array}}.\]
for some $x_i\in k$. Again, we must have $[[[v,f\otimes 1],v],v]=0$ which yields a cubic equation, \begin{align*}x_{30}^3+4\cdot x_{30}^2\cdot x_{31}+2\cdot x_{30}\cdot x_{31}^2+2\cdot x_{31}^3&=0.
\end{align*}
In characteristic $5$ it conveniently factorises as $(x_{30}+3\cdot x_{31})^3=0$ giving $x_{30}=2\cdot x_{31}$. Now, recognising that $(\ad(h\otimes (1+x)))^4(e\otimes (1+x))=e\otimes 1$ and imposing this on $v$, we get a condition \[x_{31}^4\cdot x_{124}=2.\]
In particular, $x_{31}\ne 0$.
The maximal torus $T_0= \varpi_4^\vee(k^\times)$ of $C_{e\otimes 1}$ fixes
$D$ and we may
replace $v$ by $(\Ad\, t)(v)$ for any $t\in T_0$ without affecting the conditions imposed earlier. The explicit form of $v$ given above shows
that it can be replaced by an $(\Ad\,T_0)$-conjugate in such a way that $x_{31}=1$.
Then $x_{124}=2$ which implies that the $v$'s representing $e\otimes (1+x)$ form a single conjugacy class under the action of $\Ad\,T_0$.
At this point we can argue as in the previous case to conclude that ${\rm Soc}(\h')$ is unique up to conjugacy.

For the next two cases, let $p=7$. In this case $e\in\OO({\rm A}_2{{\rm A}_1}^3)$ and the following $\sl_2$-triple spans the subalgebra $S\otimes 1$ of ${\rm Soc}(\h')$:
\begin{align*}e\otimes 1\,&=\,e_{\Small\begin{array}{c c c c c c c c}1&0&0&0&0&0\\&&0\end{array}}+
e_{\Small\begin{array}{c c c c c c c c}0&0&0&0&0&0\\&&1\end{array}}+
e_{\Small\begin{array}{c c c c c c c c}0&1&0&0&0&0\\&&0\end{array}}+
e_{\Small\begin{array}{c c c c c c c c}0&0&0&1&0&0\\&&0\end{array}}+
e_{\Small\begin{array}{c c c c c c c c}0&0&0&0&0&1\\&&0\end{array}},\\
h\otimes 1\,&=\,2\cdot h_{\alpha_1}+
h_{\alpha_2}+
2\cdot h_{\alpha_3}+
h_{\alpha_5}+
h_{\alpha_7},\\
f\otimes 1\,&=\,2\cdot e_{\Small\begin{array}{c c c c c c c c}-1&0&0&0&0&0\\&&0\end{array}}+
e_{\Small\begin{array}{c c c c c c c c}-0&0&0&0&0&0\\&&1\end{array}}+
2\cdot e_{\Small\begin{array}{c c c c c c c c}-0&1&0&0&0&0\\&&0\end{array}}+
e_{\Small\begin{array}{c c c c c c c c}-0&0&0&1&0&0\\&&0\end{array}}+
e_{\Small\begin{array}{c c c c c c c c}-0&0&0&0&0&1\\&&0\end{array}}.
\end{align*}
This choice is compatible with \cite[p.~97]{LT11}.

(iii) Suppose $D^p=0$. Then we proceed as before. By Lemma~\ref{pindownd}, we may take
\[D\,=\,e_{\Small\begin{array}{c c c c c c c c}0&0&0&1&1&0\\&&0\end{array}}+
e_{\Small\begin{array}{c c c c c c c c}0&0&0&0&1&1\\&&0\end{array}}+
e_{\Small\begin{array}{c c c c c c c c}1&1&1&0&0&0\\&&0\end{array}}
- e_{\Small\begin{array}{c c c c c c c c}0&1&1&0&0&0\\&&1\end{array}}
-2\cdot e_{\Small\begin{array}{c c c c c c c c}0&0&1&1&0&0\\&&1\end{array}}+
e_{\Small\begin{array}{c c c c c c c c}0&1&1&1&0&0\\&&0\end{array}},\]
a regular nilpotent element of $\c_{e\otimes 1}$.
Now we let $v$ be a general element representing $e\otimes x\in {\rm Soc}(\h')$. As before, we have that $v\in \g_{e\otimes 1}(2)$. Since $[d,v]=e\otimes 1$ a GAP computation tells that
\begin{align*}v\,=\,\,x_7&\cdot e_{\Small\begin{array}{c c c c c c c c}1&0&0&0&0&0\\&&0\end{array}}+
x_7\cdot e_{\Small\begin{array}{c c c c c c c c}0&0&0&0&0&0\\&&1\end{array}}+
x_7\cdot e_{\Small\begin{array}{c c c c c c c c}0&1&0&0&0&0\\&&0\end{array}}+
x_7\cdot e_{\Small\begin{array}{c c c c c c c c}0&0&0&1&0&0\\&&0\end{array}}+
x_7\cdot e_{\Small\begin{array}{c c c c c c c c}0&0&0&0&0&1\\&&0\end{array}}\\&+
x_{37}\cdot e_{\Small\begin{array}{c c c c c c c c}1&1&1&1&1&0\\&&1\end{array}}+
x_{37}\cdot e_{\Small\begin{array}{c c c c c c c c}0&1&1&1&1&1\\&&1\end{array}}+
x_{37}\cdot e_{\Small\begin{array}{c c c c c c c c}1&2&2&1&0&0\\&&1\end{array}}+
4\cdot x_{56}\cdot e_{\Small\begin{array}{c c c c c c c c}1&2&2&2&2&1\\&&1\end{array}\
}+
x_{56}\cdot e_{\Small\begin{array}{c c c c c c c c}1&2&3&2&1&1\\&&2\end{array}}\\&+
x_{63}\cdot e_{\Small\begin{array}{c c c c c c c c}2&3&4&3&2&1\\&&2\end{array}}+
e_{\Small\begin{array}{c c c c c c c c}-0&0&0&0&1&0\\&&0\end{array}}+
2\cdot e_{\Small\begin{array}{c c c c c c c c}-0&0&1&0&0&0\\&&1\end{array}}+
6\cdot e_{\Small\begin{array}{c c c c c c c c}-0&1&1&0&0&0\\&&0\end{array}}+
e_{\Small\begin{array}{c c c c c c c c}-0&0&1&1&0&0\\&&0\end{array}}\end{align*} for some $x_{7}, x_{37}, x_{56}, x_{63}\in k$.
Now we proceed as in case~(i). As $h\otimes x=[e\otimes x,f\otimes 1]$ and the latter element has zero $[p]$th power in $\g$ we must insist that $(\ad\,[v,f\otimes 1])^7(e\otimes 1)=0$. This leads to the condition $x_{63}=3\cdot x_7^7$. Furthermore, $[[h\otimes x,e\otimes x],e\otimes x]=0$ leads to the condition $3\cdot x_{37}^2=0$. As a result, $x_{37}=0$ and
\begin{align*}v\,=\,\,&x_7\cdot e_{\Small\begin{array}{c c c c c c c c}1&0&0&0&0&0\\&&0\end{array}}+
x_7\cdot e_{\Small\begin{array}{c c c c c c c c}0&0&0&0&0&0\\&&1\end{array}}+
x_7\cdot e_{\Small\begin{array}{c c c c c c c c}0&1&0&0&0&0\\&&0\end{array}}+
x_7\cdot e_{\Small\begin{array}{c c c c c c c c}0&0&0&1&0&0\\&&0\end{array}}+
x_7\cdot e_{\Small\begin{array}{c c c c c c c c}0&0&0&0&0&1\\&&0\end{array}}+
4\cdot x_{56}\cdot e_{\Small\begin{array}{c c c c c c c c}1&2&2&2&2&1\\&&1\end{array}\
}\\&+
x_{56}\cdot e_{\Small\begin{array}{c c c c c c c c}1&2&3&2&1&1\\&&2\end{array}}+
3\cdot x_7^7\cdot e_{\Small\begin{array}{c c c c c c c c}2&3&4&3&2&1\\&&2\end{array}}+
e_{\Small\begin{array}{c c c c c c c c}-0&0&0&0&1&0\\&&0\end{array}}+
2\cdot e_{\Small\begin{array}{c c c c c c c c}-0&0&1&0&0&0\\&&1\end{array}}+
6\cdot e_{\Small\begin{array}{c c c c c c c c}-0&1&1&0&0&0\\&&0\end{array}}+
e_{\Small\begin{array}{c c c c c c c c}-0&0&1&1&0&0\\&&0\end{array}}.\end{align*}
We denote the set of all such $v$'s by $\mathcal{V}$ and
put
$$v_0\,:=\,e_{\Small\begin{array}{c c c c c c c c}-0&0&0&0&1&0\\&&0\end{array}}+
2\cdot e_{\Small\begin{array}{c c c c c c c c}-0&0&1&0&0&0\\&&1\end{array}}+
6\cdot e_{\Small\begin{array}{c c c c c c c c}-0&1&1&0&0&0\\&&0\end{array}}+
e_{\Small\begin{array}{c c c c c c c c}-0&0&1&1&0&0\\&&0\end{array}}$$
Note that $v_0\in\mathcal{V}$ and $[D,v_0]=e\otimes 1$.

Recall the $1$-dimensional torus $T_1$ of $C_{e\otimes 1}$ defined in \S~\ref{3.6}. Since $D^p=0$, it follows from
\cite{McN04} that there exists a one-parameter unipotent subgroup
$U_1=\{x(t)\,|\,\,t\in k\}$ of $C_{e\otimes 1}$ such that $\Lie(U_1)=kD$ and $T_1\subset N_G(U_1)$.
By the general theory of algebraic groups, $$\big(\Ad\,x(t)\big)(v_0)\,=\,\textstyle{\sum}_{i\ge 0}\,t^iX^{(i)}(v_0)$$ for some endomorphisms $X^{(i)}$ of $\g$ independent of $t$ (these endomorphisms may be different from those in Proposition~\ref{expo}). Each endomorphism $X^{(i)}$ has weight $2i$ with respect to the action of $T_1$ on $\End(\g)$ and there exists a nonzero scalar $r\in k^\times $ such that
$X^{(i)}=\frac{r^i}{i!}(\ad D)^i$ for $1\le i\le p-1$. Since the set $\mathcal{V}$ is $(\Ad\,U_1)$-stable, looking at the $T_1$ weights of the root vectors involved one observes that the orbit $(\Ad\ U_1)\, v_0\subset \mathcal{V}$ contains a vector of the form
$$v_\lambda\,=\,v_0+\lambda\Big(4\cdot e_{\Small\begin{array}{c c c c c c c c}1&2&2&2&2&1\\&&1\end{array}}
+ e_{\Small\begin{array}{c c c c c c c c}1&2&3&2&1&1\\&&2\end{array}}\Big)$$
for some $\lambda\in k$.
In order to simplify $v_\lambda$ further we are going to employ the
group of type ${\rm A}_1$ generated by the
unipotent root subgroups $U_{\pm\tilde{\beta}}$ of $C_{e\otimes 1}$; see \S~\ref{3.6} for detail. We call it  $\mathcal{S}_{\tilde{\beta}}$
and let $X_{\tilde{\beta}}$ denote a spanning vector of $\Lie(U_{\tilde{\beta}})$.
Since $\tilde{\beta}=3\beta_1+2\beta_2$ is the highest root of $\Phi(C_{e\otimes 1}, T_0)$ with respect to the basis of simple roots $\{\beta_1,\beta_2\}$ and $T_0$ is generated by
$\varpi_i^\vee(k^\times)$ with $i=4,6$, it is immediate from \cite[p.~97]{LT11} that $\mathcal{S}_{\tilde{\beta}}$ contains $\varpi_6^\vee(k^\times)$ as a maximal torus acting on the line $k X_{\tilde{\beta}}$ with weight $2$. Since all weights of $\g$ with respect to
$\varpi_6^\vee(k^\times)$ lie in the set
$\{\pm 2,\pm 1, 0\}$, representation theory of $\SL(2)$ yields that the
$\mathcal{S}_{\tilde{\beta}}$-module $\g$ is completely reducible.
This implies that if $u=u_0+u_{-1}\in\g$ is the sum of (nonzero) weight vectors $u_0$ and $u_{-1}$ for $\varpi_6^\vee(k^\times)$ corresponding to weights $0$ and $-1$, respectively, then
$[X_{\tilde{\beta}}, u]\ne 0$.

Since the group $U_{\tilde{\beta}}$ fixes $D$ and $S\otimes 1$ point-wise, it acts on $\mathcal{V}$.
Applying the preceding remark with $u=v_0$ it is straightforward to see that each $v_\lambda$ lies in the orbit $(\Ad\,U_{\tilde{\beta}})\,v_0$.
This determines $v$ uniquely up to conjugacy under the adjoint action of centraliser of $D$ in $C_{e\otimes 1}$. So we may repeat the argument used at the end of part~(i) and move to the last case.

(iv) Suppose $p=7$ and $D^p=D$. Then $S\otimes 1$ is as in part~(iii) and in view of Lemma~\ref{pindownd} we may assume that
$$D\,=\,5\cdot h_{\alpha_1}+
4\cdot h_{\alpha_2}+
3\cdot h_{\alpha_3}+
h_{\alpha_4}+
6\cdot h_{\alpha_6}+
3\cdot h_{\alpha_7}.$$ As in part~(ii),
we represent $e\otimes (1+x)$ by a general element $v$ of $\g_{e\otimes 1}(2)$. Imposing the condition $[D,e\otimes(1+x)]=e\otimes(1+x)$ on $v$ gives
\[v\,=\,\,x_{55}\cdot e_{\Small\begin{array}{c c c c c c c c}1&2&2&2&2&1\\&&1\end{array}}+
x_{56}\cdot e_{\Small\begin{array}{c c c c c c c c}1&2&3&2&1&1\\&&2\end{array}}+
(x_{103}+x_{104})\cdot e_{\Small\begin{array}{c c c c c c c c}-1&1&2&1&1&0\\&&1\end{array}}+
x_{103}\cdot e_{\Small\begin{array}{c c c c c c c c}-0&1&2&2&1&0\\&&1\end{array}}+
x_{104}\cdot e_{\Small\begin{array}{c c c c c c c c}-0&1&2&1&1&1\\&&1\end{array}}.\]
As $[[e\otimes (1+x),f\otimes 1],e\otimes (1+x)]$ commutes with $e\otimes (1+x)$ and $(\ad (h\otimes (1+x)))^6(e\otimes (1+x))=e\otimes 1$, we must have $[[[v,f\otimes 1],v,v]=0$ and $(\ad\,[v,f\otimes 1])^6(v)=e\otimes 1$. Imposing these conditions on $v$ leads to two algebraic equations
\begin{align}x_{56}\cdot x_{103}^2+4\cdot x_{56}\cdot x_{103}\cdot x_{104}+4\cdot x_{56}\cdot x_{104}^2&=0\label{A}\\
x_{55}\cdot x_{56}^2\cdot x_{103}^4+5\cdot x_{55}\cdot x_{56}^2\cdot x_{103}^2\cdot x_{104}^2+x_{55}\cdot x_{56}^2\cdot x_{104}^4&=1.\label{B}
\end{align}
Since $x_{56}\neq 0$ by (\ref{B}), we may divide (\ref{A}) through by $x_{56}$. This yields
$$0\,=\,x_{103}^2+4x_{103}\cdot x_{104}+4x_{104}^2\,=\,(x_{103}+2x_{104})^2,$$ so that $x_{103}=-2\cdot x_{103}$. Then $(\ref{B})$ can be rewritten as
$$1\,=\,x_{55}\cdot x_{56}^2\cdot \big(16\cdot x_{104}^4+20\cdot x_{104}+x_{104}^4\big)\,=\,2\cdot x_{55}\cdot x_{56}^2\cdot x_{104}^4.$$
As a result, both $x_{56}$ and $x_{104}$ are nonzero.
Since the maximal torus $T_0$ of $C_{e\otimes 1}$ generated by $\varpi_i^\vee(k^\times)$ with $i\in\{4,6\}$ fixes $D$ and $S\otimes 1$ point-wise, we may replace $v$ by any element of the form $(\Ad\,t)(v)$ with $t\in T_0$.
In doing so we may arrange for both $x_{56}$ and $x_{104}$ to be equal to $1$. Then $x_{55}=4$ implying that $v$ is uniquely determined up to conjugacy under the action of centraliser of $D$ in $C_{e\otimes 1}$. Then so is $v-(e\otimes 1)$ and we can argue as before to finish the proof.

\subsection{The normaliser of ${\rm Soc}(\h)$ in $G$}\label{3.9} Let $\widetilde{\h}=\n_\g({\rm Soc}(\h))$.
In this subsection we are going to prove the remaining statements of
Theorem~\ref{thm:esdps} except for the maximality of $\widetilde{\h}$. We may assume that $G={\rm Aut}(\g)^\circ$ is an adjoint group of type ${\rm E}_7$.
Let $\h$ be an esdp of $\g$, and put $N:=N_G({\rm Soc}(\h))$.
By \S~\ref{3.7}, we know that ${\rm Soc}(\h)=S\otimes \OO(1;\underline{1})$ is unique
up to $(\Ad\,G)$-conjugacy.  Therefore, we may take for ${\rm Soc}(\h)$ the subalgebra of $\g$ described in \S~\ref{3.6}.
Put $\widetilde{\h}:=\n_\g({\rm Soc}(\h))$.

Recall that the role of $e\otimes x^{p-1}$ is played by a multiple of the highest root vector $e_{\tilde{\alpha}}
\in\g_{e\otimes1}(2)$ whilst $\del\in\mathcal{D}$ is represented by a regular nilpotent element of $\c_e$ contained in $\Lie(C_{e\otimes 1}\cap U^-)$;
we call it $F$.
Since $\h$ is a restricted subalgebra of $\g$, the elements $y_s:=(\ad F)^s(e_{\tilde{\alpha}})\in\g_{e\otimes 1}(2)$ with $0\le s\le p-2$ which represent nonzero multiples of $e\otimes x^{p-s-1}$ lie in
$\N_p(\g)$. By Proposition~\ref{expo}, there exist one-parameter unipotent subgroups $\mathcal{Y}_s=\big\{y_s(t)\,|\,\,t\in k\big\}$ of $G_{e\otimes 1}$ such that $y_s\in \Lie(\mathcal{Y}_s)$ and
$$\big(\Ad\,y_s(t)\big)(v)\,\equiv\,\,\sum_{i=0}^{p-1}\,\frac{1}{i!}(\ad y_s)^i(v)\mod \bigoplus_{j\ge i+2p} \g(\tau,j)\qquad \big(\forall v\in \g(\tau,i)\big).$$ Since $\g(\tau, j)=0$ for $j\ge 2p-2$ by \cite[pp.~97, 104]{LT11}, the groups $\Ad\, \mathcal{Y}_s$ fix  $A:=e\otimes \OO(1;\underline{1})$  point-wise and send $f\otimes 1$ into the space $(f\otimes 1)+[A,f\otimes1]+[A,[A,f\otimes1]]$. (Recall that $(\ad A)^3 (f\otimes 1)=0$.) Since $e\otimes\OO(1;\underline{1})$ and $f\otimes 1$ generate the Lie algebra
$S\otimes \OO(1;\underline{1})$ we thus deduce that $\mathcal{Y}_s\subset N^\circ$ for all $0\le s\le p-2$.

Since $e\otimes 1$ is a regular nilpotent element in a standard Levi subalgebra $\l=\Lie(L)$ of $\g$ of type ${\rm A}_3{\rm A}_2{\rm A}_1$ or ${\rm A}_2{{\rm A}_1}^3$, using \cite[pp.~97, 104]{LT11} it is easy to observe that there is a regular subgroup $\mathcal{S}$ of type ${\rm A}_1$ in $L$ which commutes with $C_{e\otimes 1}$ and has the property that
$\Lie(\mathcal{S})=S\otimes 1$. The subgroup $\mathcal{S}$ contains $\tau(k^\times)$ as a maximal torus and $e_{\tilde{\alpha}}$ is a highest weight vector of weight $2$ for $\mathcal{S}$. Hence the $(\Ad\,\mathcal{S})$-submodule of $\g$ generated by $e_{\tilde{\alpha}}$ is spanned by $(\ad(f\otimes 1))^i(e_{\tilde{\alpha}})$ with $i\in\{0,1,2\}$.
As $\mathcal{S}$ fixes $F\in\Lie(C_{e\otimes 1})$, it follow that
$(\Ad\,\mathcal{S})(A)\subset {\rm Soc}(\h)$. Since the Lie algebra ${\rm Soc}(\h)$ is generated by $A$ and $\Lie(\mathcal{S})=S\otimes 1$, we now deduce that $\mathcal{S}\subset N^\circ$.

It is immediate from the above remarks that ${\rm Soc}(\h)\subseteq\Lie(N)$.
If $p=5$ then the Borel subgroup $B_{e\otimes 1}^+$ also normalises ${\rm Soc}(\h)$. If $p=7$ the subalgebra $\mathfrak{w}_{(0)}$ defined in \S~\ref{3.6} lies in $\n_\g({\rm Soc}(\h))$ and is acted upon by the $1$-dimensional torus $T_1\subset C_{e\otimes 1}$. It is easy to  check that $T_1\subset N^\circ$. We claim that there exists a connected unipotent subgroup $W_{(1)}\subset C_{e\otimes 1}\cap N^\circ$ such that $\Lie(W_{(1)})=\mathfrak{w}_{(1)}$.
Indeed, it follows from \cite{Sei} or \cite{McN04} that there exists
a subgroup $\mathcal{S}_{\rm reg}$ of type ${\rm A}_1$ in $C_{e\otimes 1}$ such that $T_1\subset \mathcal{S}_{\rm reg}$ and
$F_1+F_2\in \Lie(\mathcal{S}_{\rm reg})$
(the regular nilpotent element $F_1+F_2\in\c_{e\otimes 1}$ was defined in \S~\ref{3.6}).
Since $T_1$ acts on the maximal unipotent subgroup
$\mathcal{S}_{\rm reg}\cap U^+$ of $\mathcal{S}_{\rm reg}$ with weight $2$
and $(\ad (F_1+F_2))^2$ maps $\Lie(\mathcal{S}_{\rm reg}\cap U^+)$ onto
$k(F_1+F_2)$, it is straightforward to see that
$\Lie(\mathcal{S}_{\rm reg})\subset \mathfrak{w}$.
We  take for $W_{(1)}$ the connected unipotent group generated by $\mathcal{S}_{\rm reg}\cap U^+$ and
the unipotent root subgroups $C_{e\otimes 1,\gamma}$ of $C_{e\otimes 1}$ corresponding to the roots $\gamma$ of height $\ge 2$ with respect to the basis of simple roots $\{\beta_1,\beta_2\}$. Since in characteristic $7$ the groups $\mathcal{S}_{\rm reg}\cap U^+$ and
$C_{e\otimes 1,\gamma}$ with ${\rm ht}(\gamma)\ge 2$ normalise
$\mathfrak{w}=\Lie(\mathcal{S}_{\rm reg})\oplus
\textstyle{\sum}_{{\rm ht}(\gamma)\ge\, 2}\,  \Lie(C_{e\otimes 1,\gamma})$
and fix $e_{\tilde{\alpha}}$, we get $W_{(1)}\subset N^\circ$ proving the claim. The above discussion shows that $W_{(0)}:=T_1\cdot W_{(1)}$ is a $6$-dimensional connected solvable subgroup of $N\cap C_{e\otimes 1}$
normalising $\mathfrak{w}\cong W(1;\underline{1})$. Moreover, it not hard to check that $W_{(0)}$ acts on $\mathfrak{w}$ faithfully.
Since it  follows from \cite{Jac}
that ${\rm Aut}(W(1;\underline{1}))\cong {\rm Aut}(\OO(1;\underline{1}))$
is a $6$-dimensional connected algebraic group, it must be that
$W_{(0)}\cong {\rm Aut}(W(1;\underline{1}))$ as algebraic $k$-groups.

We claim that $\widetilde{\h}$ acts faithfully on ${\rm Soc}(\h)$. Indeed,
let $c\in \c_\g({\rm Soc}(\h))$. Then $c\in \c_{e\otimes 1}$ commutes with $A$ and normalises the Lie algebra $\c_{e\otimes 1}\cap \widetilde{\h}$. If $p=5$ the letter coincides with $\c_{e\otimes 1}$ which acts on $A$ faithfully; see \S~\ref{3.6}. So $c=0$ in this case. If $p=7$ then our discussion in \ref{3.6} shows that the $\c_{e\otimes 1}$-module generated by $e_{\tilde{\alpha}}$ contains $A$ properly.
This means that
$\c_{e\otimes 1}\cap \widetilde{\h}$ is a proper Lie subalgebra of $\c_{e\otimes 1}$ containing $\mathfrak{w}$. Since $\mathfrak{w}$ is maximal in $\c_{e\otimes 1}$ by \cite[Theorem~1.1(iv)]{HSMax}, we now obtain that $\c_{e\otimes 1}\cap \widetilde{\h}=\mathfrak{w}$ and $c\in \mathfrak{w}$.
Since the simple Lie $\mathfrak{w}$ must act faithfully on $A$, we see that $c=0$ in all cases. The claim follows.

Next we show that $N$ acts faithfully on ${\rm Soc}(\h)$. If
$\sigma\in C_G({\rm Soc}(\h))$ that $\sigma\in N\cap C_{e\otimes 1}$ fixes $e_{\tilde{\alpha}}$ which implies that $\sigma\in B_{e\otimes 1}^+$. If $p=5$ then \cite[p.~105]{LT11} implies that both $\g_{e\otimes 1}$ and $\g_{f\otimes 1}$ are spanned by $T_0$-weight vectors of even weights.
Since this holds in any good characteristic (including characteristic zero)
it follows that all weights of the $T_0$-module $\g$ are even.
As a consequence,
 $C_{e\otimes 1}\cong \,
 {\rm PGL}(2,k)$.
Since $\sigma$ fixes $[x,e]\in A$ for any $x\in \Lie(C_{e\otimes 1}\cap U^-)$, this entails that $\sigma=1$ in the present case.

If $p=7$ then $\sigma\in N\cap C_{e\otimes 1}$ fixes $A$ point-wise and
acts as an automorphism on the Lie algebra
$\c_{e\otimes 1}\cap \widetilde{\h}\cong W(1;\underline{1})$.
Since ${\rm Aut}(\c_{e\otimes 1}\cap \widetilde{\h})\cong W_{(0)}$ by our earlier remarks, there exists $w\in W_{(0)}$ such that
$w\sigma$ acts on $\c_{e\otimes 1}\cap \widetilde{\h}$ trivially.
If $z$ is the semisimple part of $w\sigma$ in $C_{e\otimes 1}$ then  $\c_{e\otimes 1}\cap \widetilde{\h}$ is contained in the regular subalgebra
$(\c_{e\otimes 1})^z$ of $\c_{e\otimes 1}$.  Since $\c_{e\otimes 1}\cap \widetilde{\h}$ is maximal in $\c_{e\otimes 1}$
and no regular subalgebra of $\c_{e\otimes 1}$ is isomorphic to $W(1;\underline{1})$,
we obtain $z=1$. But then $w\sigma$ is unipotent and
$\dim\, (\c_{e\otimes 1})^{w\sigma}$ coincides with the number of Jordan blocks of $w\sigma$ on $\c_{e\otimes 1}$. Since $\c_{e\otimes 1}\cap \widetilde{\h}\subseteq (\c_{e\otimes 1})^{w\sigma}$, this number is at least
$7=\frac{1}{2}\dim\,\c_{e\otimes 1}$. This implies that
$(w\sigma-{\rm Id})^2$ annihilates $\c_{e\otimes 1}$. Since $C_{e\otimes 1}$ is a connected simple algebraic group of type ${\rm G}_2$ by \cite[p.~97]{LT11},  applying \cite[Theorem~1]{PSup} yields $\sigma=w^{-1}\in W_{(0)}$. Since
$W_{(0)}\cong {\rm Aut}(\OO(1;\underline{1}))$ acts faithfully on
$A=e\otimes\OO(1;\underline{1})\cong \OO(1;\underline{1})$ this gives $\sigma=1$.

As a result of the above deliberations, we obtain a natural injective homomorphism of algebraic $k$-groups $\psi\colon\, N\to {\rm Aut}(S\otimes\OO(1;\underline{1}))$.
Since $\widetilde{\h}$ acts faithfully ${\rm Soc}(\h)=S\otimes\OO(1;\underline{1})$,
the differential ${\rm d}_e\psi\colon\,\Lie(N)\to \Der(S\otimes \OO(1;\underline{1}))$ is injective as well.
As before, we identify the centreless Lie algebra $S\otimes \OO(1;\underline{1})$ with $\ad(S\otimes \OO(1;\underline{1}))$. It follows from Block's theorem that $$\Der(S\otimes \OO(1;\underline{1}))\,\cong\,(S\otimes\OO(1;\underline{1}))\rtimes
({\rm Id}\otimes W(1;\underline{1}));$$ see \cite[Corollary~3.3.5]{Str04}. Since the Lie algebra of the algebraic group
${\rm Aut}(S\otimes\OO(1;\underline{1}))$ is a subalgebra of
$\Der(S\otimes \OO(1;\underline{1}))$ preserving the nilradical $S\otimes \OO(1;\underline{1})_{(1)}$ of
$S\otimes \OO(1;\underline{1})$, it is contained in $(S\otimes\OO(1;\underline{1}))\rtimes
({\rm Id}\otimes W(1;\underline{1})_{(0)})$. On the other hand, the group
${\rm PGL_2}\big(\OO(1;\underline{1})\big)\rtimes{\rm Aut}(\OO(1;\underline{1}))$ embeds in a natural way into the automorphism group of the Lie algebra $S\otimes \OO(1;\underline{1})\cong \sl_2\big(\OO(1;\underline{1})\big)$. Moreover, since $p>2$ it is straightforward to see that
$$\Lie\big({\rm PGL_2}\big(\OO(1;\underline{1})\big)\rtimes{\rm Aut}(\OO(1;\underline{1}))\big)\,\cong\, (S\otimes\OO(1;\underline{1}))\rtimes
({\rm Id}\otimes W(1;\underline{1})_{(0)}).$$ Since ${\rm Aut}(\OO(1;\underline{1}))$ is a connected group it follows that
\begin{equation}\label{PGL}
{\rm Aut}(S\otimes\OO(1;\underline{1}))^\circ\,\cong\,\,{\rm PGL_2}\big(\OO(1;\underline{1})\big)\rtimes{\rm Aut}(\OO(1;\underline{1}))
\end{equation}
as algebraic $k$-groups. The $k$-algebra $\OO(1;\underline{1})$ acts freely on the Lie algebra $S\otimes\OO(1;\underline{1})$ and identifies with the centroid
$\mathcal{C}:=\,\End_{\,S\otimes\OO(1;\underline{1})}\,\big(S\otimes\OO(1;
\underline{1})\big)$ of the latter. The group ${\rm Aut}(S\otimes \OO(1;\underline{1}))$ acts on $\mathcal{C}$ and we let ${\rm Aut}_{\mathcal{C}}(S\otimes \OO(1;\underline{1}))$ denote the kernel of this action.
As $(\OO(1;\underline{1})_{(1)})^p=0$, it is easy to check that any $x\in S\otimes \OO(1;\underline{1})_{(1)}$
has the property that $(\ad\,x)^p=0$ and
$[(\ad\,x)^{i}(a),(\ad\,x)^{j}(b)]=0$ for all $a,b\in S\otimes\OO(1;\underline{1})$ whenever $i+j\ge p$. From this it is immediate
that $\exp(\ad\,x)\in {\rm Aut}_{\,\mathcal{C}}(S\otimes \OO(1;\underline{1}))$
for all $x\in S\otimes \OO(1;\underline{1})_{(1)}$. Let $\mathcal{R}$ denote the subgroup of
 ${\rm Aut}_{\mathcal{C}}(S\otimes \OO(1;\underline{1}))$
 generated by all $\exp(\ad\,x)$ with  $x\in S\otimes \OO(1;\underline{1})_{(1)}$. Clearly, $\mathcal{R}$ is a connected unipotent group and $\Lie(\mathcal{R})=S\otimes \OO(1;\underline{1})_{(1)}$.
 In view of (\ref{PGL}) the connected subgroup $\mathcal{H}$ of ${\rm Aut}_{\,\mathcal{C}}(S\otimes \OO(1;\underline{1}))$ generated by $\mathcal{R}$ and the simple algebraic subgroup
 ${\rm Aut}(S\otimes 1)$ of $ {\rm Aut}_{\mathcal{C}}(S\otimes \OO(1;\underline{1}))$ identifies with ${\rm PGL_2}\big(\OO(1;\underline{1})\big)$ in such a way that
 $\mathcal{R}=R_u(\mathcal{H})$ identifies with the unipotent radical of ${\rm PGL_2}\big(\OO(1;\underline{1})\big)$.

Let $\{u,v,w\}$ be a nonzero $\sl_2$-triple of $S\otimes \OO(1;\underline{1})$. Replacing $\{u,v,w\}$ with $\{s(u),s(v),s(w)\}$ for some $s\in {\rm Aut}(S\otimes 1)$ we may assume that
$$(u,v,w)\equiv (e\otimes 1,h\otimes 1,f\otimes 1)\mod \big(S\otimes \OO(1;\underline{1})_{(1)}\big)^3.$$
If $v=h\otimes a+e\otimes b+f\otimes c$ and $d\in\Z_{>0}$ are such that  $a\in 1+\OO(1;\underline{1})_{(1)}$, $b,c\in
\OO(1;\underline{1})_{(d)}$  and $\{b,c\}\not\subset \OO(1;\underline{1})_{(d+1)}$, then we can find an element $x\in (ke+kf)\otimes \OO(1;\underline{1})_{(d)}$ such that $\exp(\ad\,x)(v)=h\otimes a'+e\otimes b'+f\otimes c'$ for some
 $a\in 1+\OO(1;\underline{1})_{(1)}$ and $b,c\in
\OO(1;\underline{1})_{(d+1)}$. Therefore, replacing our current $\sl_2$-triple $\{u,v,w\}$ with $\{r(u),r(v), r(w)\}$ for a suitable $r\in\mathcal{R}$ we may assume without loss that $v=h\otimes a_1$ for some $a_1\in 1+\OO(1;\underline{1})_{(1)}$. If $u=e\otimes b_1+h\otimes b_2+
f\otimes b_3$, where $b_i\in \OO(1;\underline{1})$, then the condition
$[v,u]=2u$ yields $2a_1b_1=2b_1$, $2b_2=0$, and $2ab_3=-2b_3$. Since  $p>2$
and $a+1\in\OO(1;\underline{1})^\times$, this entails that $b_2=b_3=0$. So $u=e\otimes b_1$ for some $b_1\in \OO(1;\underline{1})$. Since $[v,w]=-2w$, one can argue similarly to deduce that $w=f\otimes c_1$ for some $c_1\in  \OO(1;\underline{1})$. Since $[u,w]=v$ we must have $b_1c_1=a_1$. As a consequence,
$b_1,c_1\in \OO(1;\underline{1})^\times$. But then the relation
$a_1b_1=b_1$ gives $a_1=1$ forcing $v=h\otimes 1$ and $c_1=b_1^{-1}$. Finally, since $b_1$ is invertible there is $y\in h\otimes \OO(1;\underline{1})_{(1)}$ such that $\exp(\ad\,y)(e\otimes b_1)=e\otimes 1$.
Since $\exp(\ad\,y)$ fixes $v=h\otimes 1$, we conclude that $\exp(\ad\,y)(w)=f\otimes 1$. This shows that all nonzero $\sl_2$-triples of
$S\otimes \OO(1;\underline{1})$ are conjugate under the action of $\mathcal{H}$.

We now claim that the group ${\rm Aut}(S\otimes \OO(1;\underline{1}))$ is connected. Indeed, let $g$ be an arbitrary element of ${\rm Aut}(S\otimes \OO(1;\underline{1}))$. It follows from (\ref{PGL}) that there is a
$g'\in {\rm Aut}(S\otimes \OO(1;\underline{1}))^\circ$ such that $gg'\in
{\rm Aut}_{\mathcal{C}}(S\otimes \OO(1;\underline{1}))$. Hence we may assume that $g$  fixes $\mathcal{C}$ point-wise. Then the action of $g$ on $S\otimes\OO(1;\underline{1})$ is uniquely  determined by its effect on
$S\otimes 1$. Let $u=g(e\otimes 1)$, $v=g(h\otimes 1)$ and $w=g(f\otimes 1)$.
Clearly, $\{u,v,w\}$ is a nonzero $\sl_2$-triple in $S\otimes \OO(1;\underline{1})$. By the previous paragraph, there is a $g''\in \mathcal{H}\subseteq {\rm Aut}_{\mathcal{C}}(S\otimes \OO(1;\underline{1}))^\circ$ which has the same effect on $S\otimes 1$ as $g$.
Therefore, $g=g''$ and the claim follows. As a byproduct we obtain that
$\mathcal{H}\cong {\rm PGL_2}\big(\OO(1;\underline{1})\big)$ coincides with ${\rm Aut}_{\mathcal{C}}(S\otimes \OO(1;\underline{1})).$

If $p=7$ then the above discussion shows that the connected group $N$ has the same dimension as the connected group
${\rm Aut}(S\otimes \OO(1;\underline{1}))$. Since ${\rm d}_e\psi$ is injective
the map $\psi\colon\,N\to
{\rm Aut}(S\otimes \OO(1;\underline{1})$ is an isomorphism of algebraic $k$-groups. Therefore, $$N\,\cong\, {\rm PGL_2}\big(\OO(1;\underline{1})\big)\rtimes {\rm Aut}(\OO(1;\underline{1}))\, \ \mbox{ when }\ p=7.$$ It follows that $\Lie(N)$ has codimension $1$ in $\widetilde{\h}
\cong (S\otimes \OO(1;\underline{1}))\rtimes ({\rm Id}\otimes W(1;\underline{1}))$. As $\mathcal{D}$ is transitive for any esdp $\h$, we obtain that $\widetilde{\h}=\h+\Lie(N)$.

If $p=5$ then $\c_{e\otimes 1}$ identifies with a transitive subalgebra of $W(1;\underline{1})$ isomorphic to $\sl_2(k)$. Then $\widetilde{\h}\cong (S\otimes\OO(1;\underline{1}))\rtimes
{\rm Id}\otimes \big(k\partial\oplus k(x\partial)\oplus(x^2\partial)\big)$
by Lemma~\ref{tran}.
Also, $\psi(N)$ is a closed subgroup of ${\rm Aut}(S\otimes \OO(1;\underline{1}))$ containing $\mathcal{H}$.
More precisely, $\psi(N) = \mathcal{H}\cdot \mathcal{A}$ where $\mathcal{A}=\psi(B_{e\otimes 1}^+)$, a $2$-dimensional connected solvable subgroup of ${\rm Aut}(S\otimes \OO(1;\underline{1}))$ fixing $S\otimes 1$ point-wise. Therefore, $\mathcal{A}$ acts by automorphisms on the associative algebra $\mathcal{C}\cong \OO(1;\underline{1})$. Since $\mathcal{A}$
preserves the maximal ideal of $\mathcal{C}$ and
$\Lie(B_{e\otimes 1}^+)\subset \c_{e\otimes 1}$, our identification of
$\widetilde{\h}$ implies that $\Lie(\mathcal{A})= k(x\partial)\oplus(x^2\partial)$.

This shows that there exists a connected solvable subgroup
${\rm Aut}_{\le 1}(\OO(1;\underline{1}))\cong  B_{e\otimes 1}^+$ of ${\rm Aut}(\OO(1;\underline{1}))$
with Lie algebra $k(x\partial)\oplus k(x^2\partial)\cong \Lie(B_{e\otimes 1}^+)$ such that
$$N\,\cong\, {\rm PGL_2}\big(\OO(1;\underline{1})\big)\rtimes {\rm Aut}_{\le 1}(\OO(1;\underline{1}))\ \,\mbox{ when }\ p=5.$$
Since $\c_{e\otimes 1}$ identifies with $k\partial\oplus k(x\partial)\oplus(x^2\partial)$  we see that in the present case $\Lie(N)$ has codimension $1$ in $\widetilde{\h}$.
Since the component $\mathcal{D}$ of any esdp $\h$ is transitive, we again obtain that $\widetilde{\h}=\h+\Lie(N)$.

As $\mathcal{H}\subset \psi(N)$ acts transitively on the set of all nonzero $\sl_2$-triples of
$S\otimes \OO(1;\underline{1})$ we have proved all statements of Theorem~\ref{thm:esdps} except for the maximality of $\widetilde{\h}$.
This issue will be addressed after we classify all maximal subalgebras of $\g$ with semisimple socles.

\section{The maximal Lie subalgebras with semisimple socles}\label{socle-ss}
In this section we assume that $\m$ is a maximal Lie subalgebra of $\g$ whose socle ${\rm Soc}(\m)$ is semisimple. This assumption implies that
there exist simple Lie subalgebras $S_1,\ldots, S_r$ of $\g$ such that
$[S_i,S_j]=0$ for $i\ne j$ and ${\rm Soc}(\m)=S_1\oplus\cdots \oplus S_r$.
In the next subsection we deal with the case where $r\ge 2$. Using a theorem of Bate--Martin--R\"ohrle--Tange \cite{BMRT} we are going to show that
in this situation there exists a maximal connected subgroup $M$ of $G$ such that $\m=\Lie(M)$ (see also \cite{H}).

It $\m$ is regular, that is contains a maximal toral subalgebra of $\g$, then  it follows from \cite[Theorem~13.3]{Hum} and \cite[Ch.~II, \S\S~3 and 4]{Sel}, for example, that there exists a maximal torus $T$ of $G$ and a maximal root subsystem $\Psi$ of $\Phi(G,T)$ such that $\m$ is spanned by $\Lie(T)$ and the root spaces $ke_\gamma$ with $\gamma\in \Psi$; see \cite[2.5]{P17} for a related discussion. From this it is immediate that $\m=\Lie(M)$ for some maximal
regular subgroup $M$ of $G$. So from now on we shall always assume $\m$ is a non-regular Lie subalgebra of $\g$.
\subsection{Maximal Lie subalgebras with semisimple decomposable socles}\label{ss-dec}
Recall from \S~\ref{3.2} the definition of decomposable Lie algebras.
The main result of this subsection applies to all simple algebraic groups over algebraically closed fields $k$ of {\it very good} characteristic $p>3$.
Recall that $p$ is said to be very good for $G$ if $p$ is good for $G$ and
$p\nmid (\ell +1)$
if $G$ is a group of type ${\rm A}_\ell$.
\begin{prop}\label{decomp}
Let $G$ be a simple algebraic $k$-group of adjoint type and suppose that ${\rm char}(k)=p>3$ is a very good prime for $G$. Let $\m$ be a maximal semisimple Lie subalgebra of $\g=\Lie(G)$ such that
${\rm Soc}(\m)={S}_1\oplus\cdots\oplus {S}_r$
for some simple Lie subalgebras
$S_1, \ldots, S_r$ of $\g$, where $r\ge 2$. If $\m$ is non-regular then the following hold:
\begin{itemize}
\item [(i)\,] There exists a maximal connected subgroup $M$ of $G$ such that $\m=\Lie(M)$.

\smallskip

\item[(ii)\,] $M$ is a semisimple group of adjoint type and
the simple components $M_1,\ldots, M_r$ of $M$ have the property that
 $\Lie(M_i)\cong \Der(S_i)$ and $[\Lie(M_i),\Lie(M_i)]=S_i$
for all $i$.

\smallskip

\item[(iii)\,] The Lie algebra $\m$ is isomorphic to $\Der(S_1)\oplus\cdots\oplus\Der(S_r)$.

\end{itemize}

\end{prop}

\begin{proof} For $1\le i\le r$, let $\m_i$ and $\tilde{M}_i$ denote the centralisers of $\bigoplus_{j\ne i}\,S_j$	
in $\g$ and $G$, respectively, and set $M_i:=(\tilde{M}_i)^\circ$.
By \cite[Theorem~1.2]{BMRT}, we have that $\Lie(M_i)=\m_i$ for all $i$.
Since $\bigoplus_{j\ne i}\,S_j$ is an ideal of $\m$ we have that $[\m,\m_i]\subseteq \m_i$.
Since $\g$ is a simple Lie algebra and $\m$ is a maximal subalgebra of $\g$,
each $\m_i$ is an ideal of $\m$ containing $S_i$. Set $\r_i:={\rm rad}(\m_i)$.
Since $\m$ is semisimple, it acts faithfully on its socle ${\rm Soc}(\m)$.
Therefore, each $\m_i$ acts faithfully on $S_i$.
If $\r_j\ne 0$ for some $j$ then $[\r_j,{S}_j]$ is a nonzero solvable ideal of $S_j$. Since this contradicts the simplicity of $S_j$ we deduce that each Lie algebra $\m_i$ is semisimple. It follows that each $M_i$ is a semisimple algebraic group.

Let $T_i$ be a maximal torus of $M_i$ and $\t_i=\Lie(T_i)$. Since $\m_i=\Lie(M_i)$ and
$S_i$ is an ideal of $\m_i$, we have that $[\t_i,S_i]\subseteq S_i$.
Since $p>3$ it follows from \cite[\S\S~3 and 4]{Sel} that $S_i$ is spanned by $\t_i\cap S_i$ and some root spaces $ke_\gamma$ with $\gamma\in \Phi(M_i,T_i)$ and any unipotent root subgroup $U_\beta$ of $M_i$ with $\beta\in\Phi(M_i,T_i)$ normalises $S_i$ (see \cite[2.2]{BGP} for a related discussion). This implies that $S_i$ is $M_i$-stable.
Furthermore, each simple Lie algebra $S_i$ satisfies the Mills--Seligman axioms; see \cite[2.2]{BGP} for detail. Since $\m_i$ acts faithfully on
$S_i$ this yields that each $M_i$ is a simple algebraic group.

Since $\m_i$ acts faithfully on its ideal $S_i$ it embeds into ${\rm Der}(S_i)$. It follows from \cite[Lemma~2.7]{BGP} that $\m_i\cong \Der(S_i)$ and the equality
$\Der(S_i)\,=\,\ad S_i$ holds if and only if $M_i$
is not of type ${\rm A}_{sp-1}$ for some $s\in\Z_{>0}$.
 As $\m_i\subseteq \m$ for all $i$ and $\m$ embeds into $\Der(S_1\oplus\cdots\oplus S_r)$ we now obtain that
$\m\,\cong\, \Der(S_1)\oplus\cdots\oplus\Der(S_r)$.
Note that \cite[Lemma~2.7]{BGP} also shows that if $S_i$ is of type ${\rm A}_{sp-1}$ then $S_i\cong \psl_{sp}(k)$ and
$\m_i\cong \Der({S_i})\cong \pgl_{ps}(k)$. In this case $S_i=[\m_i,\m_i]$ has codimension $1$ in $\m_i$. It follows that $[\m_i,\m_i]=S_i$ for all $i$.
	
Let $M$ denote the identity component of the normaliser $N_G(\rm{Soc}(\m))$. The above discussion shows that
$M_i\subset M$ for all $i$ and hence $\m=\bigoplus_{i=1}^r\,\m_i\subseteq \Lie(M)$.
The maximality of $\m$ in conjunction with the simplicity of $\g$ shows that $M$ is a maximal connected subgroup of $G$ and $\Lie(M)=\m$. Since $M$ is  connected  it fixes each simple ideal $S_i$ of ${\rm Soc}(\m)$ set-wise.
Hence each $M_i$ is a normal subgroup of $M$.
Since $\m$ is a semisimple Lie algebra, $M$ is a semisimple algebraic $k$-group.

If there is $1\ne z\in M$ which fixes ${\rm Soc}(\m)$ point-wise, then $\Ad\,\! z$ acts identically on $\m\cong\Der({\rm Soc}(\m))$. Hence $z$ lies in the kernel of the adjoint representation of $M$. Since the group $M$ is semisimple, it follows that $z$ is a semisimple element of $G$ contained in the centre of $M$.
Since $G$ is a group of adjoint type, we have that $\g^z\ne \g$ and hence $\m=\g^z$ by the maximality of $\m$.  Since $z$ is contained in a maximal torus of $G$ by \cite[11.11]{Borel}, this would imply that $\m$ is regular subalgebra of $\g$ contrary to our assumption.
We thus conclude that $M$ acts faithfully on ${\rm Soc}(\m)$ and $Z(M)=\{1\}$.

As a consequence, the groups $M_i$ pairwise commute.
Since $$\Lie(M)\,\cong\ \bigoplus_{i=1}^r\Lie(M_i)\,\cong\ \bigoplus_{i=1}^r \Der(S_i),$$ it is clear from our earlier remarks that the natural homomorphism of algebraic groups
$$ \mu\colon\, M\longrightarrow \,\prod_{i=1}^r {\rm Aut}(S_i)^\circ\,\cong\,\,\prod_{i=1}^r {\rm Aut}(\m_i)^\circ$$
is injective and separable. Hence $\mu$ is
an isomorphism of algebraic $k$-groups, implying that $M$ is a group of adjoint type and thereby completing the proof.
\end{proof}
	
\subsection{Restricted classical subalgebras of $\g$} \label{expsec}
In this subsection we assume that $G$ is a reductive algebraic $k$-group with a simply connected derived subgroup and the Lie algebra $\g=\Lie(G)$ admits a non-degenerate $G$-invariant symmetric bilinear form. Assume further that
${\rm char}(k)=p>3$ and $p$ is a good prime for $G$.
Let $\h$ be a simple restricted Lie subalgebra of $\g$ and suppose that
$\h$ is classical in the sense of Seligman, i.e. there is a simple algebraic $k$-group $\mathcal{H}$ of adjoint type such that $\h\cong [\Lie(\mathcal{H}),\Lie(\mathcal{H})]$ as Lie algebras.
Since $p>3$ the Lie algebra $\Lie(\mathcal{H})$ is simple unless $\mathcal{H}$ has type ${\rm A}_{rp-1}$ for some $r\ge 1$. In the latter case $\mathcal{H}={\rm PGL}_{rp}(k)$ and $\h=\psl_{rp}(k)$ has codimension $1$ in $\Lie(\mathcal{H})=\pgl_{rp}(k)$. The group $\mathcal{H}$ acts in $\h$ by Lie algebra automorphisms. Since $\z(\h)=0$ the $[p]$-structure of $\h$ is induced by that of $\g$.

Let $\OO_{\rm min}(\h)$ denote the minimal nonzero nilpotent $\mathcal{H}$-orbit
on $\h$. It consists of all nonzero elements $e\in\h$ such that
$[e,[e,\h]]=ke$. It is well known that each $e\in\OO_{\rm min}(\h)$
is conjugate to a long root element of $\h$ with respect to a maximal torus of $\mathcal{H}$. It follows that there is a
short $\Z$-grading
\begin{equation}
\label{Zgr}
\h=\h(-2)\oplus\h(-1)\oplus\h(0)\oplus\h(1)\oplus \h(2)\end{equation} of the Lie algebra $\h$ such that $\h(2)=ke$ and $\h_e$ is an ideal of codimension $1$ in the subalgebra $\bigoplus_{i\ge 0}\h(i)$ of $\h$. Furthermore, if $\rk(\mathcal{H})>1$ then $[\h(1),\h(1)]=\h(2)$. Since $p>3$ the $\mathcal{H}$-orbit $\OO_{\rm min}(\h)$ spans $\h$ (for $\h$ is an irreducible $\mathcal{H}$-module). Since $e^{[p]}\in\z(\h)$ our assumptions on $\h$ imply that $\OO_{\rm min}(\h)\subseteq \N_p(\g)$.
Since any $e\in\OO_{\rm min}(\h)$ is a $G$-unstable vector of $\g$, it admits an optimal cocharacter $\tau_e\in X_*(G)$ with the property that
$e\in\g(\tau_e,2)$; see \cite{P03} for detail. Since $e\in \N_p(\g)$ it follows from \cite{Sei} and \cite{McN04} that $\g(\tau_e,i)=0$ for all $i\ge 2p-1$.
\begin{prop}
\label{class}
Under the above assumptions on $\h$ suppose further that $\g(\tau_e,2p-2)=0$ for all $e\in\OO_{\rm min}(\h)$. Then $N_G(\h)$ acts irreducibly
on $\h$ and $\h\subseteq \Lie(N_G(\h))$.
\end{prop}
\begin{proof}
Let $e\in\OO_{\rm min}(\h)$. Then $(\ad e)^2(\h)=ke \subseteq \g(\tau_e,2)$. We claim that $\h\subseteq
\bigoplus_{i\ge -2}\,\g(\tau_e, i)$.
Indeed, suppose the contrary. Then there is $v=v_1+v_2\in\h$ such that
$0\ne v_1\in\bigoplus_{\le -3}\,\g(\tau_e,i)$ and $v_2\in\bigoplus_{\ge -2}\,\g(\tau_e,i)$.
Since $e\in\g(\tau_e,2)$, the endomorphism $\ad e$ sends each graded component $\g(\tau_e,i)$ to $\g(\tau_e,i+2)$.  Since $\g_e\subseteq\bigoplus_{i\ge 0}\,\g(\tau_e,i)$ by \cite[Theorem~A(i)]{P03}, the map $(\ad e)^2$ is injective on
$\bigoplus_{i\le -3}\,\g(\tau_e,i)$.
On the other hand, $$(\ad e)^2(v)\in \g(\tau_e,2)\cap \big((\ad e)^2(v_1)+\textstyle{\bigoplus}_{i\ge 0}\,\g(\tau_e,i)\big).$$ As $(\ad e)^2(v_1)\ne 0$ this is impossible, whence the claim.
By \cite[Proposition~33]{McN04}, the group $G$ contains a one-parameter unipotent subgroup $U_e=\{x_e(t)\,|\,\,t\in k\}$ normalised by $\tau_e(k^\times)$ and such that $e\in\Lie(U_e)$. It is immediate from the construction in {\it loc.\,cit.} that there are endomorphisms $X_e^{(i)}\in\big(\End(\g)\big)(\tau_e,2i)$ with $0\le i\le 2(p-1)$ such that $X_e^{(i)}=\frac{1}{i!}(\ad e)^{i}$ for $0\le i\le p-1$ and
$$\big({\rm Ad}\,x_e(t)\big)(v)\,=\,\textstyle{\sum}_{i=0}^{2p-2}\,t^iX_e^{(i)}(v)\qquad\ (\forall\,v\in\g).$$

As $\g(\tau_e, 2p-2)=0$ and $\h\subseteq \bigoplus_{i\ge -2}\,\g(\tau_e,i)$ by our earlier remarks, we have that
$$(\Ad\,x_e(t)\big)(v)\,=\,\sum_{i=0}^{p-1}\frac{1}{i!}(\ad e)^i(v)\in\h\qquad\,(\forall\, v\in\h).$$
This shows that $U_e\subset N_G(\h)$ for all $e\in\OO_{\rm min}(\h)$. Since $e\in \Lie(U_e)$ and $\OO_{\rm min}(\h)$ generates the Lie algebra $\h$ we now deduce that $\h\subseteq \Lie(N_G(\h))$. Since $\h$ is a simple Lie algebra, it follows that $N_G(\h)$ acts irreducibly on $\h$. This completes the proof.
\end{proof}
\subsection{Maximal subalgebras with simple socles of toral rank at least two}\label{rk>1}
In this subsection we assume that $G$ is an exceptional simple algebraic $k$-group and $p={\rm char}(k)$ is a good prime for $G$. Our goal here is classify maximal Lie subalgebras $\m$ of $\g=\Lie(G)$ whose socles are
a simple Lie algebras of toral rank at least two.
By \cite[Theorem~1.3]{HSMax}, this assumption on $\m$ implies that $${\rm Soc}(\m)\subseteq \m\subseteq
\Der({\rm Soc}(\m))$$ and there is a simple algebraic $k$-group $\mathcal{H}$ of adjoint type such that $\rk(\mathcal{H})\ge 2$ and ${\rm Soc}(\m)\cong[\Lie(\mathcal{H}),\Lie(\mathcal{H})]$ as Lie algebras. In view of \cite[Lemma~2.7]{BGP} this implies that $\Der(\m)\cong \Lie(\mathcal{H})$ as restricted Lie algebras and either $\m={\rm Soc}(\m)$ or $\mathcal{H}\cong{\rm PGL}_{kp}(k)$ for some $r\ge 1$ and $\m\cong\Lie(\mathcal{H})\cong{\pgl}_{kp}(k)$. In any event, ${\rm Soc}(\m)=[\m,\m]$
 is a restricted Lie subalgebra of $\g$.

 Set $\h:={\rm Soc}(\h)=[\m,\m]$. The main difference with the toral rank $1$ case (to be treated in the next subsection) is that for any $e\in \OO_{\rm min}(\h)$ the grading (\ref{Zgr}) has the property that
 $e\in [\h(1),\h(1)]$. Since $\h(1)\subset \h_e$, it follows that $e\in [\g_e,\g_e]$. As $e^{[p]}=0$, any
 $e\in\OO_{\rm min}(\h)$ is a reachable element of $\N_p(\g)$ in the sense of \cite[Theorem~1.1]{PSt}. We shall rely in a crucial way on the classification of reachable elements of $\N(\g)$ obtained in \emph{loc.~cit.} and combine it with the results of \S~\ref{expsec} to infer that all elements of $\OO_{\rm min}(\h)$ can be exponentiated into one-parameter unipotent subgroups of $G$ contained in the normaliser $N_G(\h)$. The maximality of $\m$ will then imply that $\m=\Lie(N_G(\m))$. As a result, $N_G(\m)$ must be simple and maximal amongst connected closed subgroups of $G$.
\begin{prop}\label{rkge2} Suppose $\m$ is a maximal subalgebra of $\g$ whose socle is a simple Lie algebra of toral rank at least two. Then $\m=\Lie(M)$ for some maximal connected subgroup $M$ of $G$. Moreover, if $G$ is a group of adjoint type and $\m$ is not a regular subalgebra of $\g$, then $M$ is a simple algebraic $k$-group of adjoint type.
\end{prop}
\begin{proof}
As $\m$ is maximal in $\g$ and $\g$ is simple, it is immediate from the preceding discussion that $\h={\rm Soc}(\m)$ is a restricted Lie subalgebra of $\g$.
As $\m\subseteq \Der(\h)$ has toral rank at least $2$, applying \cite[Theorem~1.3]{HSMax}
yields that $\h\cong[\Lie(\mathcal{H}),\Lie(\mathcal{H})]$ for some simple algebraic $k$-group $\mathcal{H}$ of adjoint type with $\rk(\mathcal{H})\ge 2$. It follows that for any $e\in\OO_{\rm min}(\h)$ the $\Z$-grading
(\ref{Zgr}) has the property that $\h(1)\subset\g_e$ and  $e\in [\h(1),\h(1)]$. As a consequence, the set $\OO_{\rm min}(\h)$ consists of reachable elements of $\g$ contained in $\N_p(\g)$.
Let $e\in\OO_{\rm min}(\h)$. A quick look at the list of such elements in \cite[Theorem~1.1]{PSt} together with the list of integers $j$ for which $\g_e(\tau_e,j)\neq 0$ in the tables in \cite{LT11} reveals that either $\g(\tau_e,2p-2)=0$ or we must have $G$ of type ${\rm E}_7$, $p=5$, and $e\in\OO({\rm A}_4{\rm A}_1)$ (one should keep in mind here that if $\g(\tau_e,m)\ne 0$ and $\g(\tau_e,i)=0$ for $i>m$ then $\g(\tau_e,m)=\g_e(\tau_e,m)$).

If $\g(\tau_e,2p-2)=0$ for all $e\in\OO_{\rm min}(\h)$ then Proposition~\ref{class} applies to $\h$ forcing
$\h\subseteq \Lie(N_G(\h))$. The maximality of $\m$ then yields $\Lie(N_G(\h))\subseteq \n_\g(\h)\subseteq \m$. If $\h$ is not isomorphic to $\psl_{rp}(k)$ then $\h\cong\Der(\h)$ and hence
$\m=\Lie(N_G(\h))$. If $\h\cong \psl_{rp}(k)$ then $\h\cong \ad \h$ has codimension $1$ in $\Der(\h)\cong\pgl_{rp}(k)$ and $$\h\subseteq \Lie(N_G(\h))\subseteq \m\subseteq \Der(\h).$$ Being the Lie algebra of an algebraic $k$-group, $\Lie(N_G(\h))$ cannot be isomorphic to $\psl_{rp}(k)$. Therefore,
the equality $\m=\Lie(N_G(\h))$ holds in all cases, and we can set $M:=N_G(\h)^\circ$. The maximality of $\m$ then shows that $M$ is a maximal connected subgroup of $G$.

If $\g(\tau_e,2p-2)\ne 0$ for some $e\in\OO_{\rm min}(\h)$ then $G$ is of type ${\rm E}_7$, $p=5$, and
$e\in\OO({\rm A}_4{\rm A}_1)$.
The relevant table in \cite{LT11} says that the identity component of the reductive part of $G_e$ is a $2$-dimensional torus. Note that $\g_e=\Lie(G_e)$ and $\h(0)$ normalises $\h(2)=ke$. As
$\n_\g(ke)=\Lie(\tau_e(k))\oplus \g_e$ we see that $\h(0)\oplus\h(1)\oplus\h(2)$ is a solvable Lie algebra.
Since $\h\cong [\Lie(\mathcal{H}),\Lie(\mathcal{H})]$ is a simple Lie algebra and $e$ is a long root element of $\m\cong \Lie(\mathcal{H})$, the classification of irreducible root systems
now yields that $\mathcal{H}$ is a group of type ${\rm A}_2$. Since $p>3$, it follows that
$\h=\m\cong\Lie(\mathcal{H})$.
This precise case was tackled in \cite[Lemma~4.9]{ST} where it was proved that a subalgebra $\h$ of type
${\rm A}_2$ in $\g$ with $\OO_{\rm min}(\h)\cap\OO({\rm A}_4{\rm A}_1)\ne \emptyset$ is maximal in $\g$ and has form $\h=\Lie(M)$ for some maximal connected subgroup $M$ of type ${\rm A}_2$ in $G$.

We thus deduce that in all cases $\m=\Lie(M)$ for some maximal connected subgroup $M$ of $G$. Suppose $G$ is a group of adjoint type and $\m$ is not a regular subalgebra of $\g$.
Since $\m$ is semisimple and its derived subalgebra is simple, it is straightforward to see that $M$ is a simple algebraic group. Therefore,
the centre of $M$ is a finite group consisting of semisimple elements of $G$.
As $\m$ is not a regular subalgebra of $\g$ and $Z(M)$ fixes $\m$ point-wise, it must be that $Z(M)=Z(G)=\{1\}$.
Since $\m=\Lie(M)$ acts faithfully on $\h=[\m,\m]$, the natural homomorphism $\psi\colon\, M\to {\rm Aut}(\h)^\circ$ is bijective and its differential ${\rm d}_e\psi\colon\, \m\to \Der(\h)$ is an isomorphism of Lie algebras. This implies that $M\cong {\rm Aut}(\h)^\circ$ is a group of adjoint type finishing the proof.
\end{proof}

\subsection{Maximal subalgebras with simple socles of toral rank one}

The goal of this subsection is to finish the proof of Theorem \ref{classicalthm}. In view of Propositions~\ref{decomp} and ~\ref{rkge2} we  may assume that  ${\rm Soc}(\m)$ is a simple Lie algebra of toral tank one
and ${\rm Soc}(\m)\subseteq \m\subseteq \Der({\rm Soc}(\m))$.
Thanks to \cite[Theorem~1.3]{HSMax} this implies that  either ${\rm Soc}(\m)\cong\sl_2(k)$ or ${\rm Soc}(\m)\cong W(1;\underline{1})$. Since it is well known (and easy to check) that all derivations of the Lie algebras $\sl_2(k)$ and $W(1;\underline{1})$ are inner, we have that $\m={\rm Soc}(\m)$. If $\m$ is a maximal Witt subalgebra of $\g$ then \cite[Theorem~1.1]{HSMax} says that $G$ is not of type ${\rm E}_6$, $p-1$ is the Coxeter number of $G$, and $\m$ is $(\Ad\,G)$-conjugate to the Lie subalgebra of $\g$ generated by a root vector $e_{\tilde{\alpha}}$ and a regular nilpotent element   $\sum_{\alpha\in\Pi}\, e_{-\alpha}$. Here $\Pi$ is a basis of simple roots
in the roots system $\Phi(G,T)$ with respect to a maximal torus $T$ of $G$ and $\tilde{\alpha}$ is the highest root of $\Phi(G,T)$ with respect to $\Pi$.

The above discussion shows that in order to finish the proof of Theorem~\ref{classicalthm} it suffices to establish the following:
\begin{prop} Let $G$ be a simple algebraic $k$-group and suppose that
	$p={\rm char}(k)$ is a very good prime for $G$.
	If $\m$ is a maximal Lie subalgebra of $\g=\Lie(G)$ isomorphic to $\sl_2(k)$, then there exists a maximal connected subgroup $M$ of type ${\rm A}_1$ in $G$ such that $\m=\Lie(M)$.
\end{prop}
\begin{proof}
As $\m$ is a restricted Lie subalgebra of $\g$ we have that $\OO_{\rm min}(\m)\subset \N_p(\g)$. If $\g(\tau_e,2p-2)=0$ for all
$e\in\OO_{\rm min}(\m)$, we may invoke Proposition~\ref{class} to conclude that $\m \subseteq\Lie(N_G(\m))$. It should be mentioned here that
that the assumption that $p>3$ imposed in \S~\ref{rk>1} can
be dropped in the present case because all derivations of $\m$ are inner.
The maximality of $\m$ then yields that $\m =\Lie(N_G(\m))$ and $M:=N_G(\m)^\circ$ is a maximal connected subgroup of $G$.

 So we may assume from now that $\m$ is spanned by an $\sl_2$-triple
 $\{e,h,f\}$ and $e\in \OO_{\rm min}(\m)$ has the property that
 $\g(\tau, 2p-2)\neq 0$, where $\tau=\tau_e$.
 Since $e^{[p]}=0$, it follows from \cite{Sei} and \cite{McN04} that $\g(\tau,i)=0$ for all  $i>2p-2$. Since
 $\g_e\subseteq \bigoplus_{i\ge 0}\,\g(\tau,i)$ by \cite[Theorem~A(i)]{P03}, it follows that the subspace $\g(\tau,2p-2)$ is $(\ad \g_e)$-stable.

 Recall the element $h_\tau$ introduced in \S~\ref{3.4}. Since $h-h_\tau\in \g_e$ and
 $[h_\tau,\g(\tau,i)]\subseteq \g(\tau,i)$ for all $i$, it must be that
 $[h,\g(\tau,2p-2)]\subseteq\g(\tau,2p-2)$.
 Let $v$ be an eigenvector for $\ad h$ contained in $\g(\tau, 2p-2)$ and denote by $V$ the linear span of all $(\ad f)^i(v)$ with $0\le i\le p-1$. By construction, $v$ is a highest weight vector of the $(\ad \m)$-module $\g$ with respect to the Borel subalgebra $kh\oplus ke$ of $\m$. Note that $(\ad f)^p=0$ as
 $\m$ is a restricted Lie subalgebra of $\g$. From this it is immediate that $V$ is an $\m$-submodule of $\g$.
 Since $$[e,f]=h=h_\tau+(h-h_\tau)\in
 \textstyle{\bigoplus}_{i\ge 0}\,\g(\tau,i)$$ and $\ad e\in \ad\, \g(\tau,2)$ is injective on
 $\bigoplus_{i<0}\,\g(\tau,i)$, it must be that $f\in \bigoplus_{i\ge -2}\,\g(\tau,i)$. Therefore, $(\ad f)^i(v)\in \bigoplus_{j\ge 2(p-1-i)}\,\g(\tau,j)$ for all $i\le p-1$. This yields
 $V\subseteq k(\ad f)^{p-1}(v)+\bigoplus_{i>0}\,\g(\tau,i)$.

 Set $w:=(\ad f)^{p-1}(v)$ and $\n_+:=\bigoplus_{i>0}\,\g(\tau,i)$. Then $w\in \g(\tau,0)$ and $[w,\n_+]\subseteq \n_+$, so that $kw+\n_+$ is a solvable Lie subalgebra of $\g$. Hence the Lie subalgebra $\langle V\rangle$ of $\g$ generated by $V$ is solvable as well. Since $\ad \m$ acts on $\g$ by derivations and preserves $V$, it must normalise $\langle V\rangle$. But then
 $\widetilde{\m}:=\m+\langle V\rangle$ is a Lie subalgebra of $\g$
 and ${\rm rad}(\widetilde{\m})$ contains $\langle V\rangle\ne 0$. Since this contradicts our assumptions on $\m$, the case we are considering cannot occur. The proposition follows.
\end{proof}
\subsection{Completion of the proof of Theorem~\ref{thm:esdps}}
Having established Theorem~\ref{classicalthm} we are now in a position to
finish the proof of Theorem~\ref{thm:esdps}.
In view of the results of Section~\ref{esdps} we just need to show that
if $\g=\Lie(G)$ is a Lie algebra of type ${\rm E}_7$ then for any esdp $\h$ of $\g$
the normaliser $\widetilde{\h}$ of the socle of $\h$ is a maximal Lie subalgebra of $\g$. Since under our assumption on $p$ isogenic groups of type ${\rm E}_7$ have isomorphic Lie algebras we may assume without loss of generality that $G={\rm Aut}(\g)^\circ$ is a group of adjoint type.

Let $\m$ be a maximal subalgebra of $\g$ containing $\widetilde{\h}$ and suppose for a contradiction that $\m\ne \widetilde{\h}$. If ${\rm rad}(\m)\ne 0$ then \cite[Theorem~1.1]{P17} says that $\m=\Lie(P)$ for some maximal  parabolic subgroup $P$ of $G$. Then $\m=\l\oplus\n$ where $\l$ is a proper Levi subalgebra of $\g$ and $\n=\Lie(R_u(P))$. Let $\pi\colon\,\widetilde{\h}\to \l$ denote the restriction to $\widetilde{\h}$ of the canonical projection
$\m\twoheadrightarrow \l$. Since $\widetilde{\h}$ is semisimple and $\n$ is nilpotent, the map $\pi$ is an injective homomorphism of restricted Lie algebras. It follows that $\pi(\widetilde{\h})$ is an esdp of $\g$ contained in $\l$. By the uniqueness of ${\rm Soc}(\h)=\soc(\widetilde{\h})$ (proved in  \S~\ref{3.7}), there exists $g\in G$ such that
$g({\rm Soc}(\widetilde{\h}))={\rm Soc}(\pi(\widetilde{\h})$. But then $g$ maps $\widetilde{\h}=\n_\g({\rm Soc}(\widetilde{\h}))$ onto $\n_\g({\rm Soc}(\pi(\widetilde{\h}))$. So no generality will be lost by assuming that $\widetilde{\h}\subset \l$. Then there exists a non-trivial
semisimple element $\sigma\in G$ such that $\widetilde{\h}\subset \g^\sigma$.
In particular, $1\ne \sigma\in C_G({\rm Soc}(\h))$. However, it is established in \S~\ref{3.9} that the group $N=N_G({\rm Soc}(\h))$ acts faithfully on ${\rm Soc}(\h)$. This contradiction shows that $\m$ is not a parabolic subalgebra of $\g$.

Essentially the same reasoning enables one to rule out the case where $\m$ is
a maximal regular Lie subalgebra of $\g$. Indeed, since $p$ is a good prime for $G$ it follows from the classification of closed, symmetric subsystems of the root system of type ${\rm E}_7$ that
$\m=\g^\sigma$ for some semsimple element of prime order in $G$; see \cite[Ex. Ch.~VI, \S~4.4]{Bour}, for example. Again this is impossible since the group $C_G({\rm Soc}(\h))$ is trivial by our discussion in \S~\ref{3.9}.	
	
Finally, suppose $\m$ is semisimple and non-regular.
If ${\rm Soc}(\m)$ is not semisimple then Lemma~\ref{socle} in conjunction with the uniqueness of $\soc(\widetilde{\h})$ yields that
${\rm Soc}(\m)=\soc(\widetilde{\h})$. But then $\m=\n_\g(\soc(\m))=\n_\g(\soc(\widetilde{\h}))=\widetilde{\h}$ contradicting our assumption on $\m$. Therefore, $\soc(\m)$ is semisimple. Since $\dim\m>\dim\widetilde{\h}>\dim W(1;\underline{1})$, it follows from Theorem~\ref{classicalthm} that $\m=\Lie(M)$ for some maximal connected subgroup $M$ of $G$. Moreover, the above discussion shows that the group $M$ is semisimple and does not contain maximal tori of $G$.

A complete list of such subgroups $M$ is obtained by Liebeck--Seitz and can
be found in \cite[Table~1]{LS04}. Since
$\dim\m>\dim\widetilde{\h}\ge 3p+3$, where $p\in\{5,7\}$, only one of the following cases may occur: if $p=5$ then  $M$ is a group of type ${\rm A}_1{\rm G}_2$ or ${\rm G}_2{\rm C_3}$ or ${\rm A}_1{\rm F}_4$ and if $p=7$ then  $M$ is a group of type  ${\rm G}_2{\rm C_3}$ or ${\rm A}_1{\rm F}_4$. Since $p>3$ it follows that $\m$ is a simple Lie algebra and all its derivations are inner. Moreover, in all cases the exist restricted simple ideals $\m_1$ and $\m_2$ of $\m$ such that
$\dim \m_1>\dim \m_2$ and $\m=\m_1\oplus\m_2$. For $i=1,2$, let
$\pi_i\colon\,\widetilde{\h}\to\m_i$ denote the restriction of the canonical projection $\m\twoheadrightarrow\m_i$ to $\widetilde{\h}$. Since $\dim\m_1<\dim \widetilde{\h}$ in all cases, $\ker\pi_1$ contains a nonzero minimal ideal of $\widetilde{\h}$. Since $\soc(\widetilde{\h})$ is the only such ideal of $\widetilde{\h}$, it must be that $\soc(\widetilde{\h})\subset \m_2$. As $\dim \soc(\widetilde{\h})>14$, this rules out the case where $M$ is of type ${\rm A}_1{\rm G}_2$. Since $\widetilde{\h}\cong (S\otimes\OO(1;\underline{1}))\rtimes ({\rm Id}\otimes\mathcal{D})$ acts faithfully on $\soc(\widetilde{\h})\cong S\otimes \OO(1;\underline{1})$ and $\mathcal{D}$ is a simple Lie algebra, we also deduce that $\widetilde{\h}\cong\pi_2(\widetilde{\h})$ as restricted Lie algebras.

As a result, $\widetilde{\h}$ is isomorphic to a restricted Lie subalgebra
of a restricted Lie algebra of type ${\rm F}_4$ or ${\rm C}_3$, both of which are isomorphic to restricted Lie subalgebras of a Lie algebra of type ${\rm E}_6$. 	However, we have proved in Section~\ref{esdps} that any semismple restricted Lie subalgebra of a Lie algebra of type ${\rm E}_6$ has a semisimple socle. This contradiction shows that $\widetilde{\h}$ is a maximal
Lie subalgebra of $\g$.
\section{A version of the Borel--Tits theorem for Lie algebras}
In proving Corollary~\ref{B-T} we are going to use induction on $\rk \D G$, the semisimple rank of $G$. It is straightforward to see that the corollary holds when $\rk \D G=1$. Suppose that it holds for all reductive $k$-groups satisfying the standard hypotheses and having semismiple rank $<\rk\D G$. We may assume without loss of generality that $\h$ is a restricted Lie subalgebra of $\g$.
\subsection{Reduction to the case where $G$ is almost simple}\label{simple case}
Suppose $G$ satisfies the standard hypotheses and let $\h$ be  a Lie subalgebra of $\g=\Lie(G)$ with ${\rm nil}(\h)\ne 0$.
We denote by $\kappa$ a non-degenerate $G$-invariant symmetric bilinear form on $\g$. According to \cite[2.1]{PSt}, the Lie algebra $\g$ decomposes into a direct sum of $G$-stable ideals
\begin{equation}
\label{Lie(G)}
\g=\z\oplus\widetilde{\g}_1\oplus\cdots\oplus \widetilde{\g}_s
\end{equation}
where $\z$ is a central toral subalgebra of $\g$ and either
$\widetilde{\g}_i=\Lie(G_i)$ for some simple component $G_i$ of $G$ not of type ${\rm A}_{rp-1}$ or $\widetilde{\g}_i\cong\gl_{rp}(k)$ for some $r\in\Z_{>0}$. All summands in (\ref{Lie(G)}) are pairwise orthogonal with respect to $\kappa$. Furthermore, if $\widetilde{\g}_i\cong \gl_{rp}(k)$ then there is an irreducible component $G_i$ of type ${\rm A}_{rp-1}$ in $G$ such that $\Lie(G_i)=[\widetilde{\g}_i,\widetilde{\g}_i]$. For each $1\le i\le s$, the canonical projection $\pi_i\colon\,\g\twoheadrightarrow \widetilde{\g}_i$ is a $G$-equivariant homomorphism of restricted Lie algebras.

Suppose $s>1$.
Renumbering the $\widetilde{\g}_i$'s if necessary we may assume that there is $t\le s$ such that $\pi_i({\rm nil}(\h))\ne 0$ for $1\le i\le t$ and $\pi_i({\rm nil}(\h))=0$ for $t<i\le s$. By our induction assumption, Corollary~\ref{B-T}
holds for all Lie subalgebras $\pi_i(\h)$ of $\widetilde{\g}_i$ with $1\le i\le t$.
Hence there exist cocharacters $\lambda_i\in X_*(G_i)$ such that $\pi_i(\h)\subseteq \Lie(P_i(\lambda))$ and $\pi_i({\rm nil}(\h))\subseteq \Lie(R_u(P_i(\lambda_i)))$, where $P_i(\lambda_i)$ is the parabolic subgroup of $G_i$ associated with $\lambda_i$.  Since $X_*(G_i)\subseteq X_*(G)$ for all $i$, we may consider the cocharacter $\lambda=\sum_{i=1}^{t}\lambda_i\in X_*(G)$ and the parabolic subgroup
$P(\lambda)$ associated with $\lambda$ in $G$.
Then ${\rm nil}(\h)\subseteq \sum_{i=1}^{t}\pi_i({\rm nil}(\h))\subseteq \Lie(R_u(P(\lambda))$. Since $\Lie(P(\lambda))$ contains $\z$ and all $\widetilde{\g}_i$ with $i>t$, it must be that $\h\subseteq \Lie(P(\lambda))$. This implies that Corollary~\ref{B-T} holds for $G$.
\subsection{Reduction to the case where $G$ is exceptional}\label{classical case} By \S~\ref{simple case}, we may assume that $\g=\z\oplus\widetilde{\g}_1$. If $\widetilde{\g}_1\cong\gl(V)$, where $p\mid \dim V$, or if  $\widetilde{\g}_1\cong\sl(V)$, where $p\nmid \dim V$, then ${\rm nil}(\h)=\pi_1({\rm nil}(\h))$. Let
\begin{equation}\label{flag}0=V_0\subset V_1\subset \cdots
\subset V_m=V\end{equation} be a composition series of the $\pi(\h)$-module $V$.
By Engel's theorem, $x(V_i)\subseteq V_{i+1}$ for all
$1\le i\le m$. Due to our assumption on $\widetilde{\g}_1$, the stabiliser
$$\widetilde{\p}_1:=\{x\in\widetilde{\g}_1\,|\,\,x(V_i)\subseteq V_i, \ 1\le i\le m\}$$ of the (partial) flag of subspaces (\ref{flag}) is a parabolic subalgebra of $\widetilde{\g}_1$. Moreover, ${\rm nil}(\h)$ is contained in
${\rm nil}(\widetilde{\p}_1)=\{x\in\widetilde{\g}_1\,|\,\,x(V_i)\subseteq V_{i+1}, \ 1\le i\le m\}$.
The inverse image of $\widetilde{\p}_1$ in $\g$ under the natural epimorphism $\g\twoheadrightarrow \widetilde{\g}_1$ is a parabolic subalgebra of $\g$ and its nilradical contains ${\rm nil}(\h)$. It follows that in the present case there exists a
parabolic subgroup $P$ of $G$ such that $\h\subseteq \Lie(P)$ and
${\rm nil}(\h)\subseteq \Lie(R_u(P))$.

Now suppose that $s=1$ and $\widetilde{\g}_1=\Lie(G_1)$ where $G_1$ is a group of type ${\rm B}$, ${\rm C}$  or ${\rm D}$. Since $p\ne 2$, the Lie algebra $\widetilde{\g}_1$ is simple and we may assume without loss of generality that $\z=0$ and $\widetilde{\g}_1=\g=\Lie(G)$. We may also assume that there exists a finite dimensional vector space $V$ over $k$ and a non-degenerate bilinear form $\Psi$ on $V$ such that $\g$ coincides with
the stabiliser $\g(V,\Psi)$ of $\Psi$ in $\sl(V)$. If $W$ is an irreducible $\h$-submodule of $V$ then ${\rm nil}(\h)$ annihilates $W$ (by Engel's theorem) and either $\Psi$ vanishes on $W\times W$ or the restriction of $\Psi$ to $W$ is non-degenerate. In the first case, $W$ is totally isotropic with respect to $\Psi$ and hence $\n_\g(W)$ is a maximal parabolic subalgebra of $\g$.
In the second case, $W$ is a non-degenerate subspace of $V$ and $\h$ preserves the direct sum decomposition $V=W\oplus W^\perp$.

If $W$ is totally isotropic, the above discussion shows that there exists a maximal parabolic subgroup $P=L\cdot R_u(P)$ of $G$ with Levi subgroup $L$ such that $\h\subseteq \Lie(P)$.
If ${\rm nil}(\h)\subseteq \Lie(R_u(P))$ then we are done. If not, we apply the induction assumption to the image of $\h$  under the natural homomorphism $\Lie(P)\twoheadrightarrow \Lie(L)$ (one should keep in mind here that $\rk \D L< \rk \D G$). The inverse image in $P$ of that parabolic subgroup  will have all properties we require.

Suppose $W$ is a non-degenerate subspace of  $V$ and set $\Psi_{W}:=\Psi_{\vert W}$ and
$\Psi_{W^\perp}:=\Psi_{\vert W^\perp}$.
Since ${\rm nil}(\h)$ annihilates $W$ it must act faithfully on $W^\perp\ne 0$. There exists a semisimple connected regular subgroup $G_0=G_{0}'\cdot G_{0}''$ of $G$ such that $\Lie(G_{0}')= \g(W, \Psi_W)$ and $\Lie(G_{0}'')=\g(W^\perp,\Psi_{W^\perp})$. The normal subgroups $G_{0}'$ and $G_{0}''$ of $G_0$ commute and $$\Lie(G_0)= \g(W, \Psi_W)\oplus \Lie(G_{W^\perp},\Psi^\perp).$$ Let $\pi_2$ denote the second projection $\Lie(G_0)\twoheadrightarrow \Lie(G_{W^\perp},\Psi^\perp)=\Lie(G_{0}'')$
Although the group $G_{0}$ is not simply connected, its normal subgroup $G_{0}''$ is. Indeed, looking at the extended Dynkin diagram of the root system of $G$ it is easy to check that $G_{0}''$ coincides with the derived subgroup of a Levi subgroup of $G$. In particular, $G_{0}''$ satisfies the standard hypotheses. Let $\pi_2$ denote the second projection $\Lie(G_0)\twoheadrightarrow \Lie(G_{W^\perp},\Psi^\perp)=\Lie(G_{0}'')$.

Since ${\rm nil}(\h)$ annihilates $W$ and $\g(W,\Psi_W)$ it must be that ${\rm nil}(\h)=\pi_2({\rm nil}(\h))$. Since ${\rm nil}(\h)\ne 0$ and $\rk \D G_{0}''<\rk  \D G$, our induction assumption implies that there exists
$\lambda\in X_*(G_{0}'')$ such that $\pi_2(\h)\subseteq \Lie(P_{0}''(\lambda))$ and $\pi_2({\rm nil}(\h))\subseteq
\Lie(R_u(P_{0}''(\lambda)))$, where $P_{0}''(\lambda)$ is the parabolic subgroup of $G_{0}''$ associated with $\lambda$. Since $X_*(G_{0}'')\subset X_*(G)$ the group $P_{0}''(\lambda)$ is contained in the parabolic subgroup $P(\lambda)$ associated with $\lambda$ in $G$. Furthermore, $R_u(P_{0}''(\lambda))\subseteq R_u(P(\lambda))$.
Since $G_{0}'$ and $G_{0}''$ commute, we also have that $G_{0}'\subset P(\lambda)$. Therefore, $\h\subseteq \Lie(G_{0}')\oplus \pi_2(\h)\subseteq \Lie P(\lambda)$ and ${\rm nil}(\h)=\pi_2({\rm nil}(\h))\subseteq \Lie(R_u(P(\lambda)))$, as wanted.
\subsection{Reduction to the case where $\h$ is contained in an esdp of $\g$}
\label{non-esdp} From now  we may assume that $G$ is an exceptional group. Let $\m$ be a maximal subalgebra of $\g$
containing $\h$. We first suppose that $\m=\Lie(M)$ for some maximal connected subgroup $M$ of $G$.
If $M=L\cdot R_u(M)$ is a parabolic subgroup of $G$, then either ${\rm nil}(\h)\subseteq \Lie(R_u(M))$ or
the image $\bar{h}$ of $\h$ in $\Lie(L)$ under the canonical homomorphism $\m=\Lie(L)\oplus \Lie(R_u(M))\twoheadrightarrow \Lie(L)$ has a nonzero nilradical. In the first case we are done and in the second case our induction assumption entails that there is a parabolic subgroup $P_L$ of $L$ such that
$\bar{\h}\subseteq \Lie(P_L)$ and ${\rm nil}(\bar{\h})\subseteq \Lie(R_u(P_L))$. The inverse image, $P$, of $P_L$ in $M$ is then a parabolic subgroup of $G$  with the property that $\h\subseteq \Lie(P)$ and ${\rm nil}(\h)\subseteq \Lie(R_u(P))$.

 Suppose $M$ is a semisimple regular subgroup of $G$ and let $M_1,\ldots, M_s$  be the simple components of $M$. Analysing the extended Dynkin diagrams of the exceptional root systems and using
 \cite[Ex. Ch.~VI, \S~4.4]{Bour} one observes that $s\in\{1,2,3\}$ and  $M$ has no components of type ${\rm A}_{rp-1}$ for $r\in \Z_{>0}$ (it is crucial here that $p$ is a good prime for $G$).
 In view of \cite[Lemma~2.7]{BGP}, this implies that
 for any $i\le s$ the Lie algebra
 $\Lie(M_i)$ is simple and all its derivations are inner. As a consequence, $\m=\bigoplus_{i=1}^s\Lie(M_i)$, a direct sum of Lie algebras.

 Let $\widetilde{M}_i$ be a simply connected cover of $M_i$. Then the preceding remark
 shows  $\widetilde{M}:=\prod_{i=1}^s \widetilde{M}_i$ is a simply connected cover of $M$ and
 $\Lie(\widetilde{M})\cong \Lie(M)$ as restricted Lie algebras. Since each $\Lie(M_i)$ is a simple algebra and the restriction of the Killing form of $\g$ to $\Lie(M_i)$ is nonzero, the group $\widetilde{M}$ satisfies the standard hypotheses. If $s\ge 2$ we consider the natural projections $\pi_i\colon\, \Lie(\widetilde{M})\cong \Lie(M)\twoheadrightarrow \Lie(M_i)$, where $1\le i\le s$, and use our induction assumption to find parabolic subalgebras $\p_i$ in $\Lie(M_i)$ with $\pi_i(\h)\subseteq \p_i$ and $\pi_i({\rm nil}(\h))\subseteq {\rm nil}(\p_i)$ for all $i\le s$.
 There exists a a cocharacter $\lambda\in X_*(M)$ such that parabolic subgroup $P_M(\lambda)$ of $M$ associated with $\lambda$ has the property that  $\Lie(P_M(\lambda))=\bigoplus_{i=1}^s\,\p_i$ and $\Lie(R_u(P_M(\lambda)))=\bigoplus_{i=1}^s\,{\rm nil}(\p_i)$. Then
 $\h\subseteq \Lie(P_M(\lambda))$ and ${\rm nil}(\h)\subseteq \Lie(R_u(P_M(\lambda)))$.

 If $s=1$ then using \cite[Ex. Ch.~VI, \S~4.4]{Bour} it is straightforward to see that $\m$ is one of $\sl_3(k)$ in type ${\rm G}_2$, $\so_9(k)$ in type ${\rm F_4}$, $\sl_8(k)$ in type ${\rm E}_7$, and  $\sl_9(k)$ or $\so_{16}(k)$ in type ${\rm E}_8$.
 In each of these cases, we can repeat verbatim the arguments used in \S~\ref{classical case} to find a parabolic subgroup $P_M(\lambda)$ of $M$ with similar properties.
 Since $X_*(M)\subset X_*(G)$ we can take the parabolic subgroup $P(\lambda)$ of $G$  associated with $\lambda\in X_*(G)$ for a desired parabolic subgroup $P$ of $G$. Here we use the fact that $\m\cap \Lie(P(\lambda))=
 \Lie(P_M(\lambda))$ and $\m\cap \Lie(R_u(P(\lambda)))=\Lie(R_u(P_M(\lambda)))$.

If $M$ is non-regular then $M$ is semisimple and $\rk \D M<\rk \D G$. Let $\widetilde{M}$ be a simply connected cover of $M$. A quick look at the list in \cite[Table~1]{LS04} reveals that in all cases $\m=\Lie(M)$ is a direct sum of simple Lie algebras each of which admits a non-degenerate Killing form. It follows that $\m\cong \Lie(\widetilde{M})$ and $\widetilde{M}$ satisfies the standard hypotheses.
By our induction assumption, there exists a cocharacter $\tilde{\lambda}\in X_*(\widetilde{M})$ such that
$\h\subseteq\Lie(P_{\widetilde{M}}(\tilde{\lambda}))$ and ${\rm nil}(\h)\subseteq \Lie(R_u(P_{\widetilde{M}}(\tilde{\lambda})))$. A central isogeny $\widetilde{M}\to M$ gives rise to a map $X_*(\widetilde{M})\to X_*(M)$.
Let $\lambda$ denote the image of $\tilde{\lambda}$ in $X_*(M)\subset X_*(G)$.
 Arguing as before, we now take for $P$ the parabolic subgroup $P(\lambda)$ of $G$ associated with $\lambda\in X_*(G)$.

If $\m$ is a maximal Witt subalgebra of $\g$, then $\m$ is a restricted Lie subalgebra of $\g$ and $p-1$ is the Coxeter number of $G$. By Theorem~\ref{classicalthm}(ii), we may assume that $\m=\langle f,e\rangle$ where  $e=e_{\tilde{\alpha}}$ and $f= \sum_{\alpha\in \Pi}\,e_{-\alpha}$. Here $f$ corresponds of a nonzero multiple of
$\partial\in W(1;\underline{1})$  and $e$ plays a role of $x^{p-1}\partial\in W(1;\underline{1})_{(p-2)}$.
There exists a cocharacter $\tau\in X_*(G)$ such that $f\in \g(\tau, -2)$ and $e\in \g(\tau, 2(p-2))$.
Let $P$ denote the parabolic subgroup of $G$ associated with $\tau$. If $0\le r\le p-1$ then
$(\ad f)^r(e)\in \g(\tau, 2(p-1-r))\subset \Lie(P)$.

To ease notation we identify $\m$ with $W(1;\underline{1})$. If $\h$ is contained in the standard maximal subalgebra $W(1;\underline{1})_{(0)}$ then the above show that $\h\subset \Lie(P)$. Since all nilpotent elements of $W(1;\underline{1})_{(0)}$ lie in $W(1;\underline{1})_{(1)}$, i.e. in the span of all
$(\ad f)^r(e)$ with $0\le r\le p-2$, we also have that ${\rm nil}(\h)\subset \Lie(R_u(P))$.
If $\h$ is not contained in $W(1;\underline{1})_{(0)}$ then $\h$ is a transitive restricted subalgebra of $\g$. Since ${\rm nil}(\h)\ne 0$, it follows from Lemma~\ref{tran} that $\dim \h\le 2$. If $\dim \h=1$ then $\h={\rm nil}(\h)$ is contained in the Lie algebra of the optimal parabolic subgroup $P(x)$ of any
nonzero $x\in \h$; see \cite [Theorem~A]{P03}. if $\dim \h=2$ then
Lemma~\ref{tran} implies that $\h$ is spanned by a nilpotent element $x$ and a semisimple element $y$ such that $[x,y]=y$. Then $\h\subseteq \n_\g(k x)$.  Since $\n_\g(k x)$ is contained in $\Lie(P(x))$ and ${\rm nil}(\h)=k x$ is contained in $\Lie(R_u(P(x)))$ by \cite[Theorem~A]{P03}, we conclude that the corollary holds if $\m$ is either isomorphic to a maximal Witt subalgebra of $\g$ or has form $\Lie(M)$ for some maximal connected subgroup $M$ or $G$.
\subsection{Lie subalgebras of exotic semidirect products}
It remains to consider the case where $\m=\n_\g(\soc(\m))$ is a maximal esdp of $\g$ and $\h\subset \m$.
By Theorem~\ref{thm:esdps}, this implies that $G$ is a group of type ${\rm E}_7$ and $p\in\{5,7\}$. In view of  of our results in \S\S~\ref{simple case}--\ref{non-esdp} we just need to find a proper parabolic subgroup $P$ of $G$ such that $\h\subseteq \Lie(P)$. In doing so we may assume without loss of generality that $G$ is a group of adjoint type and $\h=\n_\g({\rm nil}(\h))$ is a restricted Lie subalgebra of $\g$. We shall use freely the notation and conventions of \S\S~\ref{3.6}--\ref{3.9}

As $\m\cong (S\otimes \OO(1;\underline{1}))\rtimes ({\rm Id}\otimes \mathcal{D})$ and $\mathcal{D}\subseteq W(1;\underline{1})$, we have a natural homomorphism of Lie algebras  $\pi\colon\,\m\to W(1;\underline{1})$.
If $\pi(\h)$ is not a transitive subalgebra of $W(1;\underline{1})$, then $\h\subseteq \Lie(N)$, where $N=N_G(\soc(\m))$, and there exists a non-trivial cocharacter $\lambda\in X_*(C_{e\otimes 1})$ such that
$N\subset P(\lambda)$.  More precisely, our discussion in \S~\ref{3.9} shows that $\lambda$ comes from a regular subgroup of type ${\rm A}_1$ in $C_{e\otimes 1}$ and has the property that $\soc(\m)\subset \bigoplus_{i\ge 0}\,\g(\lambda,i)$. Since $\h\subseteq \Lie(P(\lambda))$ this sorts out the case where $\pi(\h)\subseteq W(1;\underline{1})_{(0)}$.
From now we may assume that $\pi(\h)$ is a transitive restricted subalgebra of $W(1;\underline{1})$.

Let ${\rm nil}_0(\h)={\rm nil}(\h)\cap \soc(\m)$ and suppose that ${\rm nil}_0(\h)\ne 0$. This will be our main case and
to ease notation we shall identify $\m$ with $(S\otimes \OO(1;\underline{1}))\rtimes ({\rm Id}\otimes \mathcal{D})$. Let $\varphi\colon\, S\otimes \OO(1;\underline{1})\twoheadrightarrow S$ denote the evaluation map which sends any
$\sum_{i=0}^{p-1} a_i\otimes x^i\in S\otimes\OO(1;\underline{1})$ with $a_i\in S$ to $a_0$. Obviously, $\varphi$ is a homomorphism of Lie algebras and $\ker\varphi=S\otimes \OO(1;\underline{1})_{(1)}$.

 (a) If the transitive subalgebra $\pi(\h)$ of $W(1;\underline{1})$ consists of semisimple elements, then Lemma~\ref{tran} yields that $\sigma(\pi(\h))=k(1+x)\partial$ for some $\sigma \in {\rm Aut}(W(1;\underline{1}))$ (in particular, $\dim\pi({\rm nil}(\h))=1$).
  If $\pi(\h)$ contains nonzero nilpotent elements, then Lemma~\ref{tran} says that there is a $\sigma\in {\rm Aut}(W(1;\underline{1}))$ such that $\partial\in \sigma(\pi(\h))$.
 If $p=7$ then the description of the group $N$ given in \S~\ref{3.9} yields that in any event $\sigma$ comes from the subgroup $C_{e\otimes 1}\cap N$ of $N$. If $p=5$ then the discussion in \S~\ref{3.9} shows that
 $\c_e\cong \sl_2(k)$ identifies with the Lie algebra $k\partial\oplus (x\partial)\oplus k(x^2\partial)$ and $\Lie(N)\cap\c_{e\otimes 1}=\Lie(B_{e\otimes 1}^+)$ identifies with $k(x\partial)\oplus k(x^2\partial)$. Since $\pi({\rm nil}(\h))$ is not contained in the Borel subalgebra $\Lie(B_{e\otimes 1}^+)$ of $\c_{e\otimes 1}\cong\sl_2(k)$, there exists an element $g\in B_{e\otimes 1}^+$ such that in the respective cases either $\pi(g({\rm nil}(\h))=k (1+x)\partial$ or $\partial\in \pi(g({\rm nil}(\h))$. In other words, we may assume without loss of generality that either $\pi(\h)=k (1+x)\partial$ or $\partial\in \pi(\h)$.

 (b) Suppose $\partial\in \pi(\h)$ and let $D\in\h$ be such that $\pi(D)=\partial$.  Then $D=({\rm Id}\otimes \partial)+\sum_{i=0}^{p-1}v_i\otimes x^i$ for some $v_i\in S$. Recall from \S~\ref{3.9} that $N\subseteq {\rm Aut}(S\otimes \OO(1;\underline{1}))$ contains a unipotent normal subgroup $\mathcal{R}$ generated by all
$\exp(\ad v)\in N$ with $v\in S\otimes \OO(1;\underline{1})_{(1)}$.
Since $[{\rm Id}\otimes \partial,v\otimes x^{i+1}]=(i+1)v\otimes x^i$ for all $v\in S$ and $0\le i\le p-2$, we can find an element $g\in \mathcal{R}$ such that $g(D)=({\rm Id}\otimes \partial)+u\otimes x^{p-1}$ for some
$u\in S$. So let us assume from now that $D$ itself has this form.
Since $0\ne {\rm nil}_0(\h)\subseteq S\otimes \OO(1;\underline{1})$ and $\pi(D)=\partial$, it is easy to see that
$\varphi({\rm nil}_0(\h))$ is a nonzero ideal of $\varphi(\h)$ consisting of nilpotent elements of $S\otimes 1$. Since $S\cong\sl_2(k)$, it follows that $\varphi({\rm nil}_0(\h))=k n$ for some nilpotent element $n\in S\otimes 1$. Applying to $n$ a suitable element $g'$ from the subgroup $\mathcal{H}={\rm Aut}(S\otimes 1)$ of $N$, we may assume further that $n=e\otimes 1$. This replacement will transform $D$ to $g'(D)=({\rm Id}\otimes \partial) +u'\otimes x^{p-1}$ where $u'=g(u)\in S$.

We claim that ${\rm nil}_{0}(\h)\subseteq e\otimes \OO(1;\underline{1})$. Indeed, if this is not the case, then ${\rm nil}_0(\h)$ contains an element $v=\sum_{i=0}^{p-1}b_i\otimes x^i$ such that $b_r\not\in k e$ for some
$r\in\{1,\ldots, p-1\}$. Since $D=({\rm Id}\otimes \partial) +u'\otimes x^{p-1}$, it is straightforward to check that $\varphi\big((\ad D)^r(v)\big)
=r!(b_r\otimes 1)\not\in k(e\otimes 1)$. This, however, contradicts the fact that $\ad D$ preserves ${\rm nil}_0(\h)$. The claim follows. Since $\varphi
 ({\rm nil}_0(\h))\ne 0$, there exists $a\in\OO(1;\underline{1})^\times $ such that $e\otimes a\in {\rm nil}_0(\h)$.
Let $$w=({\rm Id}\otimes d)+\textstyle{\sum}_{i=0}^{p-1}\,w_i\otimes x^i$$ be an arbitrary element of $\h$, where $d\in\mathcal{D}$ and $w_i\in S$. Since   $\ad d$ preserves
$e\otimes \OO(1;\underline{1})$ and
$[w,e\otimes a]\in {\rm nil}_0(\h)\subseteq e\otimes \OO(1;\underline{1})$, it must be that $w_i\in
\n_S(ke )$ for all $i$ (one should keep in mind here that the elements $ax^i$ with $0\le i\le p-1$ are linearly independent in $\OO(1;\underline{1})$). As a consequence, $\h\subseteq (\b\otimes \OO(1;\underline{1}))\rtimes ({\rm Id}\otimes \mathcal{D})$ where $\b=kh\oplus k e=\n_S(ke)$. But then $\h\subseteq \Lie(P(\tau))$ where $P(\tau)$ is the optimal parabolic subgroup of $e\otimes 1\in \g$.

(c) Suppose now that $\pi(\h)=k (1+x)\partial$ and let $D\in \h$ be such that $\pi(D)=(1+x)\partial$. Since $\h$ is a restricted subalgebra of $\g$ and $\h\cap (S\otimes \OO(1;\underline{1}))$ is a  restricted ideal of $\h$, it follows from Jacobson's formula for $[p]$th powers that $\varphi(D^{[p]^n})=
(1+x)\partial$ for all $n\in \Z_{>0}$. Therefore,  $D$ can be assumed to be a semisimple element of $\h$. As mentioned in \S~\ref{intro1}, the toral subalgebra of $\h$ spanned by all $D^{[p]^i}$ with $i\ge 0$, has a $k$-basis
consisting of toral elements of $\h$. So we may assume without loss of generality that $D$ is a toral element of $\h$.

Write  $D=({\rm Id}\otimes (1+x)\partial) +D_0$, where $D_0\in S\otimes\OO(1;\underline{1})$, and  recall from \S~\ref{3.9} the centroid $\mathcal{C}\cong \OO(1;\underline{1})$ of $S\otimes \OO(1;\underline{1})$.
Since $\ad D_0$ commutes with $\mathcal{C}$ we have that
\begin{equation}
\label{DD}
[D, c u]\,=\big((1+x)\partial(c)\big)u+c[D,u]\qquad\quad \big(\forall\, c\in\mathcal{C},\,\forall\,u\in S\otimes \OO(1;\underline{1})\big).\end{equation}
It follows that if $u$ is an eigenvector for $\ad D$ and $c$ is an eigenvector for $(1+x)\partial$, then
$c^iu$ is an eigenvector for $\ad D$ for any $i\in\{0,\ldots, p-1\}$.
Since $\big((1+x)\partial\big)((x+1)^i)=i(x+1)^i$ for all $i$, it follows from (\ref{DD}), that the restricted enveloping algebra $u(kD)\cong k[t]/(t^p-t)$ acts freely on $S\otimes \OO(1;\underline{1})$ and the evaluation map $\varphi$ induces a Lie algebra isomorphism between the fixed-point algebra  $(S\otimes\OO(1;\underline{1}))_D$ and $S\otimes 1$. As
$(S\otimes\OO(1;\underline{1}))_D\cong S\otimes 1$ is spanned by a nonzero $\sl_2$-triple of $S\otimes \OO(1;\underline{1})$ and all such
$\sl_2$-triples are conjugate under the action of $N$
by Theorem~\ref{thm:esdps}(iii), we may assume without loss of generality that $(S\otimes \OO(1;\underline{1}))_D=S\otimes 1$. From this it is immediate that $D_0=0$, i.e. $D=(1+x)\partial$ (we refer to \cite[\S~2]{PS2} for more general results on normalising toral subalgebras in the presence of centroids).

Since the subspace ${\rm nil}_0(\h)$ of $\h$ is $(\ad D)$-stable and $\varphi({\rm nil}_0(\h)) =k e$, it must be that $e\otimes (1+x)^i\in
{\rm nil}_0(\h)$ for some $i$. Since $(1+x)^i\in\OO(1;\underline{1})^\times$ for all $i$, we can repeat the argument used at the end of part~(b) to conclude that a $G$-conjugate of $\h$ is contained in $(\b\otimes \OO(1;\underline{1}))\rtimes k({\rm Id}\otimes (1+x)\partial)$. As before, this implies that $\h$ is contained in a proper parabolic subalgebra of $\g$.

(d) Finally, suppose that $\pi(\h)$ is transitive in $W(1;\underline{1})$ and ${\rm nil}_0(\h)=0$. In this case $\pi$ maps ${\rm nil}(\h)$ isomorphically onto a nonzero nilpotent ideal of $\pi(\h)$. Hence the Lie algebra $\pi(\h)$ is not simple and contains nonzero nilpotent elements. In view  of Lemma~\ref{tran} this means that ${\rm nil}(\h)\cong\pi({\rm nil}(\h))$ is spanned by a single nilpotent element of $\g$, say $y$. But then $\h=
\n_\g({\rm nil}(\h))=\n_\g(k y)$ is contained in $\Lie(P(y))$ where $P(y)$ is the optimal parabolic subgroup of $y\in \g$; see \cite[Theorem~A]{P03}.

As a result, if $\m$ is a maximal esdp of $\g$ and $\h\subset\m$, then
$\h$ is contained in a proper parabolic subalgebra of $\g$.
In view of Theorems~\ref{thm:esdps} and \ref{classicalthm} and our discussion in \S\S~\ref{simple case} and \ref{classical case} the proof of Corollary~\ref{B-T} is now complete.


\bibliographystyle{amsalpha}

\end{document}